\documentclass[10pt]{amsart}

\usepackage{amsfonts,amscd}
\usepackage{amssymb}
\usepackage{amsthm}
\usepackage{amsmath}

\usepackage[usenames]{color}
\definecolor{ANDREW}{RGB}{255,127,0}

\usepackage[normalem]{ulem}

\usepackage{tikz}
\usetikzlibrary{arrows,calc}

\usepackage[shortlabels]{enumitem}

\usepackage{hyperref}

\theoremstyle{plain}
\newtheorem{proposition}{Proposition}[section]
\newtheorem{theorem}[proposition]{Theorem}
\newtheorem{lemma}[proposition]{Lemma}
\newtheorem{corollary}[proposition]{Corollary}
\theoremstyle{definition}
\newtheorem{example}[proposition]{Example}
\newtheorem{definition}[proposition]{Definition}
\newtheorem{observation}[proposition]{Observation}
\theoremstyle{remark}
\newtheorem{remark}[proposition]{Remark}

\DeclareMathOperator{\Ad}{Ad}

\DeclareMathOperator{\Aut}{Aut}

\DeclareMathOperator{\SL}{SL}
\DeclareMathOperator{\GL}{GL}

\DeclareMathOperator{\SO}{SO}
\DeclareMathOperator{\PO}{PO}

\DeclareMathOperator{\SU}{SU}
\DeclareMathOperator{\PSL}{PSL}
\DeclareMathOperator{\PGL}{PGL}
\DeclareMathOperator{\Hom}{Hom}

\DeclareMathOperator{\End}{End}

\DeclareMathOperator{\Span}{Span} 
\DeclareMathOperator{\Sym}{Sym} 
 
\DeclareMathOperator{\id}{id}

\DeclareMathOperator{\Gr}{Gr}

\DeclareMathOperator{\Usf}{\mathsf{U}}

\DeclareMathOperator{\Cc}{\mathcal{C}}

\DeclareMathOperator{\Gc}{\mathcal{G}}
\DeclareMathOperator{\Hc}{\mathcal{H}}

\DeclareMathOperator{\Lc}{\mathcal{L}}

\DeclareMathOperator{\Oc}{\mathcal{O}}
\DeclareMathOperator{\Pc}{\mathcal{P}}

\DeclareMathOperator{\Uc}{\mathcal{U}}

\DeclareMathOperator{\Vc}{\mathcal{V}}

\usepackage{mathrsfs}
\DeclareMathOperator{\Es}{\mathscr{E}}
\DeclareMathOperator{\Fs}{\mathscr{F}}

\DeclareMathOperator{\Ab}{\mathbb{A}}
\DeclareMathOperator{\Bb}{\mathbb{B}}
\DeclareMathOperator{\Cb}{\mathbb{C}}

\DeclareMathOperator{\Hb}{\mathbb{H}}

\DeclareMathOperator{\Nb}{\mathbb{N}}
\DeclareMathOperator{\Pb}{\mathbb{P}}
\DeclareMathOperator{\Rb}{\mathbb{R}}

\DeclareMathOperator{\Zb}{\mathbb{Z}}

\DeclareMathOperator{\aL}{\mathfrak{a}}

\DeclareMathOperator{\gL}{\mathfrak{g}}

\DeclareMathOperator{\kL}{\mathfrak{k}}
\DeclareMathOperator{\pL}{\mathfrak{p}}

\DeclareMathOperator{\sL}{\mathfrak{sl}}

\newcommand{\abs}[1]{\left|#1\right|}

\newcommand{\norm}[1]{\left\|#1\right\|}
\newcommand{\wt}[1]{\widetilde{#1}}
\newcommand{\wh}[1]{\widehat{#1}}
\newcommand{\ip}[1]{\left\langle #1\right\rangle}

\title{Regularity of limit sets of Anosov representations}

\thanks{T.Z. was partially supported by the National Science Foundation under grant DMS-1566585, and by the NUS-MOE grants R-146-000-270-133 and A-8000-458-00-00. A.Z. was partially supported by the National Science Foundation under grants DMS-1760233, DMS-2105580, and DMS-2104381. The authors also acknowledge support from the GEAR Network, funded by the National Science Foundation under grant numbers DMS 1107452, 1107263, and 1107367 (``RNMS: GEometric structures And Representation
varieties''.)}

\author{Tengren Zhang}
\address{Department of Mathematics, National University of Singapore, 21 Lower Kent Ridge Road, Singapore 119077}
\email{matzt@nus.edu.sg}

\author{Andrew Zimmer}
\address{Department of Mathematics, University of Wisconsin, Madison, WI, 53706}
\email{amzimmer2@wisc.edu}

\date{\today}
\subjclass[2010]{}

\begin{document}

\begin{abstract} In this paper we establish necessary and sufficient conditions for the limit set of a projective Anosov representation to be a $C^{\alpha}$-submanifold of the real projective space for some $\alpha\in(1,2)$. We also calculate the optimal value of $\alpha$ in terms of the eigenvalue data of the Anosov representation.
\end{abstract}

\maketitle

\tableofcontents

\section{Introduction}
Let $\Hb_{\Rb}^d$ be the real hyperbolic $d$-space, let $\partial_\infty\Hb_{\Rb}^d$ denote the Gromov boundary of $\Hb_{\Rb}^d$, and  let ${ \rm Isom}(\Hb_{\Rb}^d)$ denote the isometry group of $\Hb_{\Rb}^d$. Given a representation $\rho: \Gamma \rightarrow { \rm Isom}(\Hb_{\Rb}^d)$, the \emph{limit set of $\rho$} is defined to be 
\begin{align*}
\Lc_\rho = \overline{\rho(\Gamma) \cdot x_0} \cap \partial_\infty \Hb_{\Rb}^d
\end{align*}
where $x_0 \in \Hb_{\Rb}^d$ is some (equivalently, any) point. If we further assume that $\Gamma$ is a hyperbolic group and $\rho$ is a convex co-compact representation, then there is a $\rho$-equivariant, continuous map from $\partial_\infty \Gamma$, the Gromov boundary of $\Gamma$, to the limit set $\Lc_\rho$. The limit set in this setting is generically very irregular, for instance when $\partial_\infty\Gamma$ is a topological manifold Yue~\cite{Y1996} proved: unless $\rho$ is a co-compact action on a totally geodesic subspace of $\Hb_{\Rb}^d$, its limit set has Hausdorff dimension strictly greater than its topological dimension. In particular, it is fractal-like. 

The group ${ \rm Isom}(\Hb_{\Rb}^d)$ is a semisimple Lie group. For a general semisimple Lie group $G$, there is a rich class of representations from a hyperbolic group $\Gamma$ to $G$ called \emph{Anosov representations}, which generalize the convex co-compact representations from $\Gamma$ to ${ \rm Isom}(\Hb_{\Rb}^d)$. Anosov representations were introduced by Labourie~\cite{L2006} and extended by Guichard-Wienhard~\cite{GW2012}. Since then, they have been heavily studied, \cite{KLP2017,KLP2013,KLP2014b,GGKW2015,BPS2016}. One reason for their popularity is that they are rigid enough to retain many of the good geometric properties that convex co-compact representations have, while at the same time are flexible enough to admit many new and interesting examples. 

In this paper, we investigate the regularity of the limit sets of Anosov representations from $\Gamma$ into $\PGL_d(\Rb)$.
We will give precise definitions in Section~\ref{sec:Anosov_repn} but informally: if $\Gamma$ is a word hyperbolic group with Gromov boundary $\partial_\infty \Gamma$, a representation $\rho:\Gamma \rightarrow \PGL_d(\Rb)$ is said to be $P_k$-Anosov if there exist continuous $\rho$-equivariant  maps 
\begin{align*}
\xi_\rho^k : \partial_\infty \Gamma \rightarrow \Gr_k(\Rb^d) \text{ and } \xi_\rho^{d-k} : \partial_\infty \Gamma \rightarrow \Gr_{d-k}(\Rb^d)
\end{align*}
which satisfy certain dynamical properties. Such maps are unique if they exist. As such, $\xi_\rho^k$ is called the \emph{$P_k$-limit map} of $\rho$ and the image of $\xi_\rho^k$ in $\Gr_k(\Rb^d)$ is called the \emph{$P_k$-limit set of $\rho$}. If $\rho$ is $P_k$-Anosov for all $k \in \{ k_1,\dots, k_j\}$ we say that $\rho$ is \emph{$P_{k_1,\dots, k_j}$-Anosov}. 

We will largely focus our attention on $P_1$-Anosov representations; by a result of Guichard-Wienhard \cite[Proposition 4.3]{GW2012}, for any Anosov representation $\rho : \Gamma \rightarrow G$ into a semisimple Lie group $G$, there exists $d > 0$ and an irreducible representation $\phi : G \rightarrow \PGL_d(\Rb)$ such that $\phi \circ \rho$ is $P_1$-Anosov. Thus, up to post composition with irreducible representations, the class of $P_1$-Anosov representations contains all other types of Anosov representations. Further, the limit maps are related via the composition by a smooth map between the associated flag manifolds (again see \cite[Proposition 4.3]{GW2012}). So all regularity properties of the limit set can be investigated by reducing to the case of $P_1$-Anosov representations.

Our first main result gives a sufficient condition for the $P_1$-limit set of a $P_1$-Anosov representation to be a $C^\alpha$-submanifold of $\Pb(\Rb^d)$ for some $\alpha>1$.

\begin{theorem}\label{thm:main} (Theorem \ref{thm:main_body}) Suppose $\Gamma$ is a hyperbolic group, $\partial_\infty \Gamma$ is a topological $(m-1)$-manifold, and $\rho: \Gamma \rightarrow \PGL_{d}(\Rb)$ is a $P_1$-Anosov representation. If 
\begin{enumerate}
\item[($\dagger$)] $\rho$ is $P_m$-Anosov and $\xi_\rho^1(x) + \xi_\rho^1(z) +  \xi_\rho^{d-m}(y)$ is a direct sum for all pairwise distinct $x,y,z \in \partial_\infty \Gamma$,
\end{enumerate}
then 
\begin{enumerate}
\item[($\ddagger$)] $M:=\xi_\rho^1(\partial_\infty \Gamma)$ is a $C^{\alpha}$-submanifold of $\Pb(\Rb^d)$ for some $\alpha > 1$. 
\end{enumerate}
Moreover, $T_{\xi_\rho^1(x)} M = \xi_\rho^m(x)$ for any $x \in \partial_\infty \Gamma$. 
\end{theorem}

\begin{remark}\label{rmk:open} \
\begin{enumerate}
\item  Pozzetti-Sambarino-Wienhard \cite{PSW18} independently proved a version of Theorem \ref{thm:main_body}, where they deduce that $M$ is a $C^1$-submanifold of $\Pb(\Rb^d)$.

\item Property ($\dagger$) and $P_1$-Anosovness in Theorem~\ref{thm:main} are open conditions in $\Hom(\Gamma, \PGL_d(\Rb))$, see Section \ref{sec:stability}. 
\item The last sentence of Theorem \ref{thm:main} is to be made sense of via the following identifications: We may view every $P\in\Gr_m(\Rb^d)$ as a $(m-1)$-dimensional projective subspace of $\Pb(\Rb^d)$. Thus, if $p\in\Pb(\Rb^d)$ is a point that lies in $P$, then in any affine chart $\Ab$ containing $p$, $P\cap\Ab$ is an affine subspace that contains $p$, and so is canonically identified with a subspace of 
\[T_p\Ab\cong T_p\Pb(\Rb^d).\] This gives a canonical identification between $\{P\in\Gr_m(\Rb^d):p\in P\}$ and the set of $(m-1)$-dimensional subspaces of $T_p\Pb(\Rb^d)$.
\end{enumerate}
 \end{remark}

Theorem \ref{thm:main} is a generalization of the following theorem due to Benoist in the setting of divisible, properly convex domains in $\Pb(\Rb^d)$ (see Section \ref{sec:properly_convex}). A group of projective transformations $\Gamma\subset\PGL_d(\Rb)$ \emph{divides} a properly convex domain $\Omega\subset\Pb(\Rb^d)$ if $\Gamma$ acts properly discontinuously and co-compactly on $\Omega$. 

\begin{theorem}[Theorem 1.1 and Proposition 4.6 of \cite{Ben04}]\label{thm:convex_divisible}
Let $\Gamma\subset\PGL_d(\Rb)$ be a hyperbolic group that divides a properly convex domain $\Omega\subset\Pb(\Rb^d)$. Then the inclusion representation $\iota:\Gamma\hookrightarrow\PGL_d(\Rb)$ is a $P_1$-Anosov representation whose $P_1$-limit set is $\partial\Omega$. Furthermore, $\partial\Omega\subset\Pb(\Rb^d)$ is a $C^\alpha$-submanifold for some $\alpha>1$.
\end{theorem}

\begin{remark} Benoist doesn't use the language of Anosov representations and for this reinterpretation of his results, see Section 6.2 in~\cite{GW2012}. \end{remark}

Theorem \ref{thm:main} also generalizes a result due to Labourie in the setting of Hitchin representations. Let $S$ be a closed orientable hyperbolizable surface and fix a Fuchsian representation $\rho_0: \pi_1(S) \rightarrow \PGL_2(\Rb)$. Then let $\tau_d : \PGL_2(\Rb) \rightarrow \PGL_d(\Rb)$ be the standard irreducible representation (see Section \ref{sec:rhoirred}). A representation $\rho : \pi_1(S) \rightarrow \PGL_d(\Rb)$ is \emph{Hitchin} if it is conjugate to a representation in the connected component of $\Hom(\pi_1(S), \PGL_d(\Rb))$ that contains $\tau_d \circ \rho_0$. 

\begin{theorem}[Theorem 1.4 of \cite{L2006}]
If $\rho:\pi_1(S)\to\PGL_d(\Rb)$ is a Hitchin representation, then $\rho$ is $P_k$-Anosov for every $k \in \{1,\dots, d-1\}$, and the $P_1$-limit set of $\rho$ is a $C^{\alpha}$-submanifold in $\Pb(\Rb^d)$ for some $\alpha>1$. 
\end{theorem}
 
Using Theorem~\ref{thm:main}, we can find more examples of representations that preserve $C^\alpha$-submanifolds in $\Pb(\Rb^d)$. 

\begin{example}\label{cor:hyperbolic_lattices}(See Section~\ref{sec:real_hyp_lattices}) Suppose $\tau: {\rm Isom}(\Hb^m_{\Rb}) \rightarrow \PGL_d(\Rb)$ is an irreducible proximal representation, $\Gamma \leq {\rm Isom}(\Hb_{\Rb}^m)$ is a co-compact lattice, and $\rho := \tau|_{\Gamma} : \Gamma \rightarrow \PGL_d(\Rb)$. Then $\rho$ is $P_1$-Anosov, and there exists a neighborhood $\Oc$ of $\rho$ in $\Hom(\Gamma, \PGL_d(\Rb))$ such that every representation in $\Oc$ is a $P_1$-Anosov representation whose $P_1$-limit set is a $C^{\alpha}$-submanifold of $\Pb(\Rb^d)$ for some $\alpha > 1$. 
\end{example}

\begin{example}\label{cor:hitchin} (See Section \ref{sec:Hitchin}) If $\rho : \pi_1(S) \rightarrow \PGL_d(\Rb)$ is a Hitchin representation, then for all $k=1,\dots,d-1$, there is an open set $\Oc$ of $\bigwedge^k\rho$ in $\Hom\big(\Gamma,\PGL\big(\bigwedge^k\Rb^d\big)\big)$ such that every representation in $\Oc$
is a $P_1$-Anosov representation whose $P_1$-limit set is a $C^{\alpha}$-submanifold of $\Pb\big(\bigwedge^k\Rb^d\big)$ for some $\alpha>1$. See Section \ref{sec:wedge} for the definition of $\bigwedge^k\rho$. In particular, by applying \cite[Proposition 4.4]{GW2012}, the $P_k$-limit set of $\rho$ is a $C^{\alpha}$-submanifold of $\Gr_k(\Rb^d)$ for some $\alpha>1$. 
\end{example}

\begin{remark} Example~\ref{cor:hitchin} was independently observed by Pozzetti, Sambarino, and Wienhard \cite{PSW18}.\end{remark}
 
In fact, Theorem \ref{thm:main} is a consequence of a more general theorem, see Theorem~\ref{thm:main_body}, that is stated using \emph{$\rho$-controlled subsets} $M\subset\Pb(\Rb^d)$, of which the $P_1$-limit set of $\rho$ is an example, see Definition \ref{def:controlled}. In the main body of our paper, all our results will be stated for $\rho$-controlled subsets. These statements are stronger than the results we mention in this introduction, but are more technical to state.
 
We also investigate the extent to which the converse of Theorem \ref{thm:main} holds. In general, there are $P_1$-Anosov representations $\rho$ whose $P_1$-limit set are $C^\infty$-submanifolds of $\Pb(\Rb^d)$, but for which ($\dagger$) in Theorem \ref{thm:main} does not hold, see Examples \ref{ex:irred_bad_example} and~\ref{ex:surface_bad_example}. However, we prove that when $\Gamma$ is virtually a surface group and $\rho$ is irreducible, the conditions in Theorem~\ref{thm:main} are both necessary and sufficient.

\begin{theorem}\label{thm:nec_surface}(Theorem \ref{thm:nec_surface_body}) Suppose $\Gamma$ is a hyperbolic group, $\rho: \Gamma \rightarrow \PGL_{d}(\Rb)$ is an irreducible $P_1$-Anosov representation, and $\partial_\infty \Gamma$ is homeomorphic to a circle. Then the following are equivalent: 
\begin{enumerate}
\item[($\dagger$)] $\rho$ is a $P_2$-Anosov representation and $\xi_\rho^1(x) + \xi_\rho^1(y) + \xi_\rho^{d-2}(z)$ is a direct sum for all $x,y,z \in \partial_\infty \Gamma$ pairwise distinct,
\item[($\ddagger$)] $\xi_\rho^1(\partial_\infty \Gamma)$ is a $C^{\alpha}$-submanifold of $\Pb(\Rb^d)$ for some $\alpha > 1$.
\end{enumerate}
\end{theorem}

\begin{remark} \
\begin{enumerate}
\item Example \ref{ex:surface_bad_example} demonstrates that the irreducibility assumption in Theorem~\ref{thm:nec_surface} is necessary.
\item The condition that $\partial_\infty \Gamma$ is homeomorphic to a circle is equivalent to requiring $\Gamma$ to be virtually a surface group, see \cite{Tukia1988,Gabai1992}.
\end{enumerate}
\end{remark}

From Theorem \ref{thm:nec_surface}, we have the following corollary. 

\begin{corollary} Suppose $\Gamma$ is a hyperbolic group with $\partial_\infty \Gamma$ homeomorphic to a circle. Let $\Oc \subset \Hom(\Gamma, \PGL_d(\Rb))$ denote the set of representations that are irreducible, $P_1$-Anosov, and whose $P_1$-limit set is a $C^{\alpha}$-submanifold of $\Pb(\Rb^d)$ for some $\alpha> 1$ (which may depend on $\rho$).  Then $\Oc$ is an open set in $\Hom(\Gamma, \PGL_d(\Rb))$.
\end{corollary}

For non-surface groups the situation is more complicated; there exist irreducible $P_1$-Anosov representations $\rho: \Gamma \rightarrow \PGL_d(\Rb)$ whose $P_1$-limit set is a $C^{\infty}$-submanifold of $\Pb(\Rb^d)$, but $\rho$ does not satisfy the condition ($\dagger$) in Theorem~\ref{thm:main}, see Example~\ref{ex:irred_bad_example}. However, if one assumes a stronger irreducibility condition on $\rho$, then the conditions in Theorem~\ref{thm:main} are both necessary and sufficient.

\begin{theorem}\label{thm:nec_general} (Theorem \ref{thm:nec_general_body}) Suppose $\Gamma$ is a hyperbolic group such that $\partial_\infty \Gamma$ is a $(m-1)$-dimensional topological manifold, and $\rho: \Gamma \rightarrow \PGL_{d}(\Rb)$ is an irreducible $P_1$-Anosov representation, such that $\bigwedge^m \rho: \Gamma \rightarrow \PGL\big(\bigwedge^m \Rb^d\big)$ is also irreducible. Then the following are equivalent: 
\begin{enumerate}
\item[($\dagger$)] $\rho$ is a $P_m$-Anosov representation and $\xi_\rho^1(x) + \xi_\rho^1(y) + \xi_\rho^{d-m}(x)$
is a direct sum for all pairwise distinct $x,y,z \in \partial_\infty \Gamma$,
\item[($\ddagger$)] $\xi_\rho^1(\partial_\infty \Gamma)$ is a $C^{\alpha}$-submanifold of $\Pb(\Rb^d)$ for some $\alpha > 1$.
\end{enumerate}
\end{theorem}

Recall that if $\rho:\Gamma\to H$ is a Zariski-dense representation and $\tau: H\to \PGL_d(\Rb)$ is an irreducible representation, then $\tau\circ\rho:\Gamma\to\PGL_d(\Rb)$ is irreducible. Thus, Theorem \ref{thm:nec_general}, together with the observation that the obvious action of $\PGL_d(\Rb)$ on $\Pb\big(\bigwedge^m\Rb^d\big)$ is irreducible, gives the following corollary. 

\begin{corollary} Suppose $\Gamma$ is a hyperbolic group. Let $\Oc \subset \Hom(\Gamma, \PGL_d(\Rb))$ denote the set of representations $\rho: \Gamma \rightarrow \PGL_d(\Rb)$ where $\rho$ is $P_1$-Anosov, has Zariski dense image, and whose $P_1$-limit set is a $C^{\alpha}$-submanifold of $\Pb(\Rb^d)$ for some $\alpha> 1$ (which may depend on $\rho$). Then $\Oc$ is an open set in $\Hom(\Gamma, \PGL_d(\Rb))$.
\end{corollary}

Finally, for representations satisfying certain irreducibility conditions, we also determine the optimal regularity of the $P_1$-limit set in terms of the spectral data of $\rho(\Gamma)$. More precisely, given $g\in\PGL_d(\Rb)$, let 
\[\lambda_1(g)\ge\dots\ge\lambda_d(g)\]
denote the moduli of the eigenvalues of a linear representative of $g$ with unit determinant. Then given a representation $\rho : \Gamma \rightarrow \PGL_d(\Rb)$ and $2 \leq m \leq d-1$ define
\begin{align*}
\Es_m(\rho) = \inf_{\gamma\in\Gamma }\left\{\log\frac{\lambda_1(\rho(\gamma))}{\lambda_{m+1}(\rho(\gamma))}\Bigg/\log\frac{\lambda_1(\rho(\gamma))}{\lambda_{m}(\rho(\gamma))}: \frac{\lambda_1(\rho(\gamma))}{\lambda_{m}(\rho(\gamma))} \neq 1 \right\}.
\end{align*}
If $\rho$ is $P_{1,m}$-Anosov, it follows from the definition that $\Es_m(\rho)>1$, see Section~\ref{sec:Anosov_repn}. 

\begin{theorem}\label{thm:regularity} (Theorem \ref{thm:regularity_body} and Corollary \ref{cor:main_body}) Suppose $\Gamma$ is a hyperbolic group such that $\partial_\infty\Gamma$ is an $(m-1)$-dimensional topological manifold and $\rho: \Gamma \rightarrow \PGL_{d}(\Rb)$ is an irreducible $P_{1,m}$-Anosov representation. If $\xi_\rho^1(x) + \xi_\rho^1(y) +  \xi_\rho^{d-m}(z)$ is a direct sum for all pairwise distinct $x,y,z \in \partial_\infty \Gamma$, then
\begin{align*}
\Es_m(\rho) \leq \sup\left\{ \alpha \in (1,2) : \xi_\rho^1(\partial_\infty \Gamma) \text{ is a }C^{\alpha}\text{-submanifold} \right\}
\end{align*}
with equality if 
\begin{align*}
\xi_\rho^1(\partial_\infty \Gamma) \cap \left(\xi_\rho^1(x) + \xi_\rho^1(y) +  \xi_\rho^{d-m}(z)\right)
\end{align*}
spans $\xi_\rho^1(x) + \xi_\rho^1(y) +  \xi_\rho^{d-m}(z)$ for all pairwise distinct $x,y,z \in \partial_\infty \Gamma$.
\end{theorem}

\begin{remark} \label{rem:stablility} \
\begin{enumerate} 
\item In Theorem \ref{thm:regularity}, when $\xi_\rho^1(\partial_\infty \Gamma)$ has either dimension one or co-dimension one, the extra hypothesis for equality is automatically satisfied. Indeed, if the dimension is one (i.e. $m=2$), then 
\begin{align*}
\xi_\rho^1(x) + \xi_\rho^1(y) +  \xi_\rho^{d-m}(z)= \xi_\rho^1(x) + \xi_\rho^1(y) +  \xi_\rho^{d-2}(z) = \Rb^d,
\end{align*}
so the extra hypothesis follows from the irreducibility of $\rho$. If the co-dimension is one (i.e. $m=d-1$), then 
\begin{align*}
\xi_\rho^1(x) + \xi_\rho^1(y) +  \xi_\rho^{d-m}(z)= \xi_\rho^1(x) + \xi_\rho^1(y) +  \xi_\rho^1(z)
\end{align*}
is spanned by $\xi_\rho^1(x), \xi_\rho^1(y), \xi_\rho^1(z)$, so the extra hypothesis always holds. 
\item In general, the extra hypothesis for equality is an open condition, see Section~\ref{sec:stability}.
\item The irreducibility of $\rho$ is necessary in Theorem~\ref{thm:regularity}. For instance, if $\overline{\tau}_d : \GL_2(\Rb) \rightarrow \GL_d(\Rb)$ is the standard irreducible representation (see Section~\ref{sec:rhoirred}), $\Gamma \leq \SL_2(\Rb)$ is a co-compact lattice, and $p_d:\GL_d(\Rb)\to\PGL_d(\Rb)$ is the obvious projection, then 
\[\rho = p_7\circ (\tau_5 \oplus \tau_2) |_{\Gamma}:\Gamma\to\PGL_7(\Rb)\] 
is $P_1$-Anosov and its $P_1$-limit set is a $1$-dimensional $C^\infty$-submanifold of $\Pb(\Rb^7)$. At the same time, for any infinite order $\gamma\in\Gamma$,
\[\log\frac{\lambda_1}{\lambda_{3}}(\rho(\gamma))\Bigg/\log\frac{\lambda_1}{\lambda_{2}}(\rho(\gamma))= 3/2.\]
\item Notice that the quantity $\Es_m(\rho)$ is invariant under passing to finite index subgroups. Indeed, if $\Gamma_0 \leq \Gamma$ is a finite index subgroup and $\gamma \in \Gamma$, then there exists some $k \in \Nb$ such that $\gamma^k \in \Gamma_0$. Further, 
\begin{align*}
\log\frac{\lambda_1}{\lambda_{m+1}}(\rho(\gamma^k))\Bigg/\log\frac{\lambda_1}{\lambda_{m}}(\rho(\gamma^k)) = \log\frac{\lambda_1}{\lambda_{m+1}}(\rho(\gamma))\Bigg/\log\frac{\lambda_1}{\lambda_{m}}(\rho(\gamma)).
\end{align*}
Hence $\Es_m(\rho|_{\Gamma_0}) = \Es_m(\rho)$. 
 
\end{enumerate}
\end{remark}

In Section~\ref{sec:regularity}, we establish a generalization of Theorem~\ref{thm:regularity} which holds for $\rho$-controlled subsets. One example of such a subset is the boundary of a properly convex domain $\Omega\subset\Pb(\Rb^d)$ that admits a $\Gamma$-action induced by a $P_1$-Anosov representation $\rho$. In this case, Theorem~\ref{thm:regularity_body} implies the following.

\begin{theorem}(Theorem \ref{thm:regularity_body} and Corollary \ref{cor: d-1 bound})\label{thm:regularity2} Suppose $\Gamma$ is a hyperbolic group and $\rho: \Gamma \rightarrow \PGL_{d}(\Rb)$ is an irreducible $P_1$-Anosov representation. Also, suppose $\Omega\subset\Pb(\Rb^d)$ is a $\rho(\Gamma)$-invariant properly convex domain such that $\xi_\rho^{d-1}(x)\cap \partial\Omega=\xi_\rho^1(x)$ for all $x\in \partial_\infty \Gamma$. If 
\begin{enumerate}
\item [($\star$)]$p_1+p_2+\xi_\rho^1(y)$ is a direct sum for all pairwise distinct $p_1,p_2,\xi_\rho^1(y)\in\partial\Omega$,
\end{enumerate}
then 
\begin{enumerate}
\item [($\star\star$)] $\partial\Omega$ is $C^\alpha$ along $\xi_\rho^1(\partial_\infty\Gamma)$ for some $\alpha>1$.
\end{enumerate}
Moreover, $T_{\xi_\rho^1(x)} \partial\Omega = \xi_\rho^{d-1}(x)$ for any $x \in \partial_\infty \Gamma$, and
\begin{align*}
\Es_{d-1}(\rho) = \sup\left\{ \alpha \in (1,2) : \partial\Omega \text{ is } C^{\alpha} \text{ along }\xi_\rho^1(\partial_\infty \Gamma) \right\}.
\end{align*}
\end{theorem}

In the case when the $\Gamma$-action on $\Omega$ is co-compact, Theorem \ref{thm:regularity2} was previously proven by Guichard~\cite{G2005} using different techniques. In that case, Guichard further proves that 
\[\Es_{d-1}(\rho) = \max\left\{ \alpha \in (1,2) : \partial\Omega \text{ is } C^{\alpha} \text{ along }\xi_\rho^1(\partial_\infty \Gamma) \right\}.\] 
In this more general setting however, our techniques do not tell us if the supremum in Theorem \ref{thm:regularity2} is attained.

\subsection*{Acknowledgements} We thank the referees for their careful reading and numerous comments which greatly improved the paper.

\section{Anosov representations}\label{sec:Anosov_repn}
In this section, we define Anosov representations from a (word) hyperbolic group to $\PGL_d(\Rb)$, and discuss some of their properties. 

\subsection{A definition of Anosov representations}
Since they were introduced by Labourie~\cite{L2006}, several characterizations of Anosov representations have been given by Kapovich, Leeb, and Porti ~\cite{KLP2017,KLP2013,KLP2014b}, Gu{\'e}ritaud, Guichard, Kassel, and Wienhard ~\cite{GGKW2015}, and Bochi, Potri, and Sambarino~\cite{BPS2016}. The definition we give below comes from~\cite[Theorem 1.7]{GGKW2015}.

To define Anosov representations, we first need appropriate definitions of ``well-behaved" limit maps associated to representations. Let $\Gamma$ be a hyperbolic group, and $\partial_\infty\Gamma$ its Gromov boundary.

\begin{definition} Let $\rho: \Gamma \rightarrow \PGL_{d}(\Rb)$ be a representation.  If $1 \leq k \leq d-1$, then a pair of maps $\xi_\rho^k: \partial_\infty \Gamma \rightarrow \Gr_k(\Rb^d)$ and $\xi_\rho^{d-k}: \partial_\infty \Gamma \rightarrow \Gr_{d-k}(\Rb^d)$ are called:
\begin{itemize}
\item \emph{$\rho$-equivariant} if $\xi_\rho^k(\gamma\cdot x) = \rho(\gamma)\cdot\xi_\rho^k(x)$ and $\xi_\rho^{d-k}(\gamma\cdot x) = \rho(\gamma)\cdot\xi_\rho^{d-k}(x)$ for all $x \in \partial_\infty \Gamma$ and $\gamma \in \Gamma$,
\item \emph{dynamics preserving} if for every $\gamma \in \Gamma$ of infinite order with attracting fixed point $\gamma^+ \in \partial_\infty \Gamma$, the points $\xi_\rho^k(\gamma^+) \in \Gr_k(\Rb^d)$ and $\xi_\rho^{d-k}(\gamma^+) \in \Gr_{d-k}(\Rb^d)$ are attracting fixed points of the action of $\rho(\gamma)$ on $\Gr_k(\Rb^d)$ and $\Gr_{d-k}(\Rb^d)$ respectively, and
\item \emph{transverse} if for every distinct $x, y \in \partial_\infty \Gamma$ we have 
\begin{align*}
\xi_\rho^k(x) + \xi_\rho^{d-k}(y) = \Rb^{d}.
\end{align*}
\end{itemize}
\end{definition}

With these, we can now define Anosov representations.
 
\begin{definition} 
A representation $\rho: \Gamma \rightarrow \PGL_d(\Rb)$ is \emph{$P_k$-Anosov} if 
\begin{itemize}
\item there exist continuous, $\rho$-equivariant, dynamics preserving, and transverse maps $\xi_\rho^k: \partial_\infty \Gamma \rightarrow \Gr_k(\Rb^d)$, $\xi_\rho^{d-k}: \partial_\infty \Gamma \rightarrow \Gr_{d-k}(\Rb^d)$, and
\item for any infinite sequence $\{\gamma_i\}_{i=1}^\infty\subset\Gamma$, we have
\[\lim_{i\to+\infty}\log\frac{\lambda_k(\rho(\gamma_i))}{\lambda_{k+1}(\rho(\gamma_i))}=\infty.\]
\end{itemize}
\end{definition}
It follows from the definition that a representation $\rho:\Gamma\to\PGL_d(\Rb)$ is $P_k$-Anosov if and only if it is $P_{d-k}$-Anosov. The dynamics preserving property implies that the maps $\xi_\rho^k$ and $\xi_\rho^{d-k}$ are unique (if they exist), and that if $k\le d-k$, then $\xi_\rho^k(x)\subset\xi_\rho^{d-k}(x)$ for all $x\in\partial_\infty\Gamma$. As such, we refer to $\xi_\rho^k$ as the \emph{$P_k$-limit map} of $\rho$, and $\xi_\rho^k(\partial_\infty\Gamma)\subset\Gr_k(\Rb^d)$ as the \emph{$P_k$-limit set} of $\rho$. Also, if $\rho$ is $P_k$-Anosov for all $k \in \{ k_1,\dots, k_j\}$ we say that $\rho$ is \emph{$P_{k_1,\dots, k_j}$-Anosov}.
 
 

\subsection{Singular values and Anosov representations} \label{sec:properties0}
Next, we give a description of Anosov representations using singular values.

\begin{definition}\label{def: singular} Let $|\cdot|$ and $\norm{\cdot}$ be norms on $\Rb^d$ induced by inner products, and let $L:(\Rb^d,|\cdot|)\to(\Rb^d,\norm{\cdot})$ be a linear map.
\begin{itemize}
\item For any non-zero $X\in(\Rb^d,|\cdot|)$, the \emph{stretch factor} of $X$ under $L$ is the quantity 
\[\sigma_X(L):=\frac{\norm{L(X)}}{|X|}.\]
\item For $i=1,\dots,d$, the $i$-th \emph{singular value} of $L$ is the quantity
\begin{align*}
\sigma_i(L)  :=\max_{W\subset\Rb^d,\dim W=i} & \min_{X\in W}\sigma_X(L) = \min_{W \subset \Rb^d, \dim W=d-i+1}  \max_{X \in W} \sigma_X(L). 
\end{align*}
\end{itemize}
\end{definition}

Observe that for all $i=1,\dots,d-1$, $\sigma_i(L)\geq\sigma_{i+1}(L)$, and if $L$ is invertible, then $\sigma_{d-i+1}(L^{-1})=\frac{1}{\sigma_i(L)}$. 

When $L\in\GL_d(\Rb)$, and $\norm{\cdot}=|\cdot|$ and is the standard $\ell^2$-norm $\norm{\cdot}_{2}$ on $\Rb^d$, we denote $\sigma_i(L)$ by $\mu_i(L)$. In that case, the singular values $\mu_1(L)\geq \dots\geq\mu_d(L)>0$ of $L$ are the square roots of the eigenvalues of $^tLL$. Then for any $g\in\PGL_d(\Rb)$, let
\[\mu_1(g)\ge\dots\ge\mu_d(g)>0\]
denote the singular values of a linear representative of $g$ with unit determinant.

Let $S$ be a finite symmetric generating set of $\Gamma$, and $d_S$ the induced word metric on $\Gamma$. The following characterization of Anosov representations was due to Kapovich, Leeb and Porti ~\cite{KLP2017,KLP2014b}, but was also later proven by Bochi, Potrie and Sambarino \cite{BPS2016} using different techniques.

\begin{theorem}\label{thm:SV_char_of_Anosov}
Suppose  $\Gamma$ is a finitely generated group, $S$ is a finite symmetric generating set, and $\rho:\Gamma\to\PGL_d(\Rb)$ is a representation. Then $\Gamma$ is a hyperbolic group and $\rho$ is a $P_k$-Anosov representation if and only if there are constants $B,C>0$ such that
\begin{align*}
\log  \frac{\mu_k(\rho(\gamma))}{\mu_{k+1}(\rho(\gamma))} \geq C d_S(\gamma,\id)-B
 \end{align*}
for all $\gamma \in \Gamma$.
\end{theorem}

Define respectively the \emph{Cartan} and \emph{Jordan projection} $\mu,\lambda:\PGL_d(\Rb)\to\Rb^d$ by 
\[
\mu(g):= \left ( \log \mu_1(g), \dots, \log \mu_d(g) \right) \quad \text{and} \quad \lambda(g) = \left ( \log \lambda_1(g), \dots, \log \lambda_d(g) \right).
\]
Observe that while the Jordan projection is invariant under conjugation in $\PGL_d(\Rb)$, the Cartan projection is not (although it is invariant under the left and right action of $\PO(d)$ on $\PGL_d(\Rb)$). These two projections can be interpreted geometrically in the following way. 

Associated to the Lie group $\PGL_d(\Rb)$ is the Riemannian symmetric space $X$, on which $\PGL_d(\Rb)$ acts transitively  and by isometries. As a $\PGL_d(\Rb)$-space, $X=\PGL_d(\Rb)/\PO(d)$. Furthermore, (after possibly scaling) the distance $d_X$ on $X$ induced by its Riemannian metric can be computed from the Cartan projection by the formula
\[d_X(g_1\cdot\PO(d),g_2\cdot\PO(d))=\norm{\mu\left(g_1^{-1}g_2\right)}_2,\]
where $\norm{\cdot}_2$ is the standard $\ell^2$-norm on $\Rb^d$. On the other hand, if $g\in\PGL_d(\Rb)$, then 
\[\inf_{p\in X}d_X(p,g\cdot p)=\norm{\lambda(g)}_2.\]

As an immediate consequence of Theorem~\ref{thm:SV_char_of_Anosov} an Anosov representation coarsely preserve the metric $d_S$ on $\Gamma$.

\begin{corollary}\label{thm:QI_Anosov} 
Let $\rho:\Gamma\to\PGL_d(\Rb)$ be $P_k$-Anosov for some $k$. Then the map $\Gamma\to X$ defined by $\gamma\mapsto \rho(\gamma)\cdot\PO(d)$ is a quasi-isometric embedding. In other words, there are constants $A\geq 1$ and $B\geq 0$ such that for all $\gamma_1,\gamma_2\in\Gamma$,
\[\frac{1}{A}\norm{\mu\left(\rho(\gamma_1^{-1}\gamma_2)\right)}_2-B\leq d_S(\gamma_1,\gamma_2)\leq A\norm{\mu\left(\rho(\gamma_1^{-1}\gamma_2)\right)}_2+B.\]
 \end{corollary} 

 \subsection{Properties of Anosov representations} \label{sec:properties}
Next, we recall some important properties of Anosov representations. 

First, there are strong restrictions on the Zariski closures of irreducible Anosov representations to $\PGL_d(\Rb)$. See \cite[Lemma 2.19]{BCLS2015} for a proof. 
 
\begin{proposition}\label{prop:Zclosure} Let $\rho: \Gamma \rightarrow \PGL_{d}(\Rb)$ be a $P_1$-Anosov representation. If $\rho$ is irreducible, then the Zariski closure of $\rho(\Gamma)$ is a semisimple Lie group without compact factors. 
\end{proposition}

Also, for $P_1$-Anosov representations, irreducibility implies strong irreducibility. See \cite[Lemma 5.12]{GW2012}.

\begin{proposition}\label{prop:strongly_irreducible} Let $\rho: \Gamma \rightarrow \PGL_{d}(\Rb)$ be an irreducible $P_1$-Anosov representation. If $\Gamma_0 \leq \Gamma$ is a finite index subgroup, then $\rho|_{\Gamma_0}$ is also irreducible. 
\end{proposition}
 
In many places, it will be more convenient to work with representations into $\SL_d(\Rb)$ instead of $\PGL_d(\Rb)$. The next observation allows us to make this reduction. Let $p_d:\GL_d(\Rb)\to\PGL_d(\Rb)$ denote the obvious projection.

\begin{observation}\label{obs:lift} For any representation $\rho: \Gamma \rightarrow \PGL_{d}(\Rb)$, there exists a subgroup $\Gamma_\rho \leq \SL_d(\Rb)$ such that $p_d|_{\Gamma_\rho}:\Gamma_\rho\to\PGL_d(\Rb)$ is a representation whose kernel is a subgroup of $\Zb/2\Zb$, and whose image is a subgroup of $\rho(\Gamma)$ with index at most two.
\end{observation}

\begin{proof}  Define $\Gamma_\rho' := \{ g \in \SL^{\pm}_d(\Rb) : [g] \in \rho(\Gamma) \}$, and let $\Gamma_\rho := \Gamma_\rho' \cap \SL_d(\Rb)$. Then $p_d(\Gamma_\rho')\subset\PGL_d(\Rb)$ coincides with $\rho(\Gamma)$, and $\Gamma_\rho$ has index at most two in $\Gamma_\rho'$.
\end{proof}

In particular, if $\rho:\Gamma\to\PGL_d(\Rb)$ is an Anosov representation, then $\Gamma_\rho$ is a hyperbolic group, and there are canonical identifications $\partial_\infty\Gamma=\partial_\infty\rho(\Gamma)=\partial_\infty p_d(\Gamma_\rho)=\partial_\infty\Gamma_\rho$. Furthermore, the following proposition is an immediate consequence of \cite[Corollary 1.3]{GW2012} .

\begin{proposition}
Let $\rho:\Gamma\to\PGL_d(\Rb)$ be a representation. The representation 
\[\rho':=p_d|_{\Gamma_\rho}:\Gamma_\rho\to\PGL_d(\Rb)\] 
is $P_k$-Anosov if and only if $\rho$ is $P_k$-Anosov. If so, the $P_k$-limit maps of $\rho$ and $\rho'$ agree.
\end{proposition}

\begin{remark}\label{rem:lift}
To prove any property about the $P_k$-limit sets of $\rho$, it is now sufficient to show that this property holds for the $P_k$-limit sets of $\rho'$. The advantage of working with $\rho'$ in place of $\rho$ is that $\rho':\Gamma_\rho\to\PGL_d(\Rb)$ admits a lift to a representation from $\Gamma_\rho$ to $\SL_d(\Rb)$. {\bf With this, we can henceforth assume that $\rho:\Gamma\to\PGL_d(\Rb)$ admits a lift to a representation $\overline{\rho}:\Gamma\to\SL_d(\Rb)$.} 
\end{remark}

\subsection{Gromov geodesic flow space}\label{sec:flowspace} 
In their proof of Theorem \ref{thm:SV_char_of_Anosov}, Bochi, Potrie, and Sambarino \cite{BPS2016} gave a characterization of Anosov representations using dominated splittings, which we now describe. To do so, we recall the definition of the flow space of a hyperbolic group, and state some of their well-known properties. For more details, see for instance~\cite{Gromov1987},~\cite{C1994}, or~\cite{M1991}. 

As a topological space, the \emph{flow space} for a hyperbolic group $\Gamma$, denoted $\Usf(\Gamma)$, is homeomorphic to $\partial_\infty\Gamma^{(2)}\times\Rb$, where $\partial_\infty\Gamma^{(2)}:=\{(x,y)\in\partial_\infty\Gamma^{2}:x\neq y\}$. This flow space admits a natural $\Rb$-action by translation in the $\Rb$-factor called the \emph{geodesic flow on $\Usf(\Gamma)$}. We will use the notation $v=(v^+,v^-,s)\in \Usf(\Gamma)$, and denote the geodesic flow on $\Usf(\Gamma)$ by $\varphi_t$, i.e. 
\[\varphi_t(v)=(v^+,v^-,s+t). \]

There is a proper, co-compact  $\Gamma$-action on $\Usf(\Gamma)$ that commutes with $\varphi_t$, and satisfies $\gamma\cdot(v^+,v^-,\Rb)=(\gamma\cdot v^+,\gamma\cdot v^-,\Rb)$. There is also a natural action of the group $\{\pm1\}$ on $\Usf(\Gamma)$ which satisfies 
\begin{align*}
-1 \cdot (x,y,\Rb) = (y,x,\Rb).
\end{align*}
This action commutes with the $\Gamma$ action, but anti-commutes with the $\varphi_t$ action, i.e.
\begin{align*}
(-1)\circ \varphi_t \circ (-1) = \varphi_{-t}.
\end{align*}
So the actions of $\Gamma$, $\varphi_t$, and $\{\pm 1\}$ combine to yield an action of $\Gamma \times (\Rb \rtimes_{\psi}\{\pm 1\})$ on $\Usf(\Gamma)$ where $\psi : \{\pm 1\} \rightarrow \Aut(\Rb)$ is given by $\psi(-1)(t) = -t$.

Since the $\Gamma$ action commutes with $\varphi_t$, the geodesic flow on $\Usf(\Gamma)$ descends to a flow on the compact space $\wh\Usf(\Gamma):=\Usf(\Gamma)/\Gamma$, which we refer to as the \emph{geodesic flow on $\wh\Usf(\Gamma)$}, and denote it by $\wh{\varphi}_t$. This also implies that if $v^+=\gamma^+$ and $v^-=\gamma^-$ are the attracting and repelling fixed points of some infinite order $\gamma\in\Gamma$, then the orbit $(\gamma^+,\gamma^-,\Rb)\subset\Usf(\Gamma)$ of $\varphi_t$ descends to a closed orbit of $\wh{\varphi}_t$ in $\wh\Usf(\Gamma)$. We will denote the period of this closed orbit by $T_\gamma\in\Rb$, and refer to $T_\gamma$ as the \emph{period of $\gamma$}. 

Further, $\Usf(\Gamma)$ admits a $\Gamma \times \{\pm1\}$-invariant proper Gromov hyperbolic metric with the following properties: 
\begin{enumerate}
\item Every orbit $(v^+,v^-,\Rb)$ of $\varphi_t$ is a quasi-geodesic.
\item For every $v \in \Usf(\Gamma)$, the orbit map $\Gamma\to \Usf(\Gamma)$ given by $\gamma\mapsto\gamma \cdot v$ is a quasi-isometry and this orbit map extends to a $\Gamma$-equivariant homeomorphism $\partial_\infty \Gamma \rightarrow \partial_\infty \Usf(\Gamma)$.
\item $v^+$ and $v^-$ in $\partial_\infty\Usf(\Gamma) \cong \partial_\infty \Gamma$ are respectively the forward and backward endpoints of the $\Rb$-orbit $(v^+,v^-,\Rb)\subset\Usf(\Gamma)$.
\end{enumerate}

\begin{remark}
In the case when $\Gamma$ is the fundamental group of a compact Riemannian manifold $X$ with negative sectional curvature,  this geodesic flow space is what one would expect. More precisely, let $T^1 X$ denote the unit tangent bundle of $X$, let $\wt{X}$ denote the universal cover of $X$, and let $T^1\wt{X}$ denote the unit tangent bundle of $\wt{X}$. Then we may take $\Usf(\Gamma)$ to be $T^1\widetilde X$ so that $\wh\Usf(\Gamma)$ is $T^1X$. The geodesic flow on both $T^1X$ and $T^1\widetilde{X}$ is the usual geodesic flow associated to the Riemannian metrics on $\widetilde{X}$ and $X$, and the Sasaki metric can be used to define a $\Gamma$-invariant proper distance on $T^1 X$. Further, the $\{\pm1\}$ action is given by $v \mapsto -v$. 
\end{remark}

Gromov proved that the geodesic flow space $\Usf(\Gamma)$ is unique up to homeomorphism.

\begin{theorem}\cite[Theorem 8.3.C]{Gromov1987}\label{thm:Gromov_uniqueness} Suppose  $\Gc$ is a proper Gromov hyperbolic metric space such that
\begin{enumerate}
\item $\Gamma \times (\Rb \rtimes_{\psi} \{\pm1\})$ acts on $\Gc$,
\item the actions of $\Gamma$ and $\{\pm1\}$ are isometric, 
\item for every $v \in \Gc$, the orbit map $\Gamma\to\Gc$ given by $\gamma\mapsto\gamma \cdot v$ is a quasi-isometry and this orbit map extends to a $\Gamma$-equivariant homeomorphism $\partial_\infty \Gamma \rightarrow \partial_\infty \Gc$ (note that this extension is independent of the orbit),
\item the $\Rb$ action is free and every $\Rb$-orbit is a quasi-geodesic in $\Gc$,
\item the induced map $\Gc / \Rb \rightarrow \partial_\infty \Gc^{(2)}$ is a homeomorphism.
\end{enumerate}
Then there exists a $\Gamma \times \{\pm1\}$-equivariant homeomorphism $T : \Gc \rightarrow \Usf(\Gamma)$ that maps $\Rb$-orbits to $\Rb$-orbits. 
\end{theorem}

\begin{remark} \ 
\begin{enumerate}
\item We note that the space $\Gc$ in Theorem~\ref{thm:Gromov_uniqueness} is not assumed to be geodesic and so the Gromov hyperbolicity of $\Gc$ is defined in terms of the four-point condition on the Gromov product (see~\cite[Chapter III.H, Definition 1.20]{BridsonHaefliger}). 
\item Despite $\Gc$ being non-geodesic, quasi-geodesics in $\Gc$ do have well defined limits in the Gromov boundary and so there is a well defined map $\Gc / \Rb \rightarrow \partial_\infty \Gc^{(2)}$ in part (5). This follows from part (3) and the fact that any quasi-geodesic in $\Gamma$ has a well defined limit in $\partial_\infty \Gamma$. 
\end{enumerate}
\end{remark}


\subsection{Dominated Splittings} \label{sec: dominated splittings}
Next, we describe an alternate characterization of Anosov representations in $\SL_d(\Rb)$ using dominated splittings due to Bochi, Potrie, and Sambarino~\cite{BPS2016}.

Let $\Gamma$ be a hyperbolic group and $\rho:\Gamma\to\SL_d(\Rb)$ be a representation. Let $E:=\Usf(\Gamma)\times\Rb^d$ be the trivial bundle over $\Usf(\Gamma)$, and define the vector bundle $E_{\rho}:=E/\Gamma$ over $\wh\Usf(\Gamma)$, where the $\Gamma$ action on $E$ is given by $\gamma\cdot(v,X)=(\gamma\cdot v,\rho(\gamma)\cdot X)$. Since $E_{\rho}$ is a vector bundle over $\wh\Usf(\Gamma)$, there is a continuous family of norms on the fibers of $E_\rho$, and the compactness of $\wh\Usf(\Gamma)$ ensures that any two such continuous families are bi-Lipschitz. For any continuous family of norms on $E_{\rho}$, let $\norm{\cdot}$ denote its lift to $E$.

We say that a representation $\rho:\Gamma\to\SL_d(\Rb)$ is \emph{$P_k$-Anosov} if $p_d\circ\rho:\Gamma\to\PGL_d(\Rb)$ is $P_k$-Anosov, where $p_d:\GL_d(\Rb)\to\PGL_d(\Rb)$ is the obvious projection. With this, we can state the following theorem due to Bochi, Potrie, and Sambarino (see Theorem 2.2, Proposition 4.5, and Proposition 4.9 in~\cite{BPS2016}).

\begin{theorem}\label{prop:dom_split} A representation $\rho:\Gamma\to\SL_d(\Rb)$ is $P_k$-Anosov if and only if there exist
\begin{itemize}
\item continuous, $\varphi_t$-invariant, $\rho$-equivariant maps 
\[
F_1:\Usf(\Gamma)\to \Gr_k(\Rb^d)\quad\text{and} \quad F_2:\Usf(\Gamma)\to\Gr_{d-k}(\Rb^d)
\]
such that $F_1(v)+F_2(v)=\Rb^d$ for all $v\in\Usf(\Gamma)$, and
\item constants $C \ge 1$, $\beta > 0$ such that 
\begin{align*}
\frac{\norm{X_1}_{\varphi_tv}}{\norm{X_2}_{\varphi_tv}} \leq C e^{-\beta t} \frac{\norm{X_1}_{v}}{\norm{X_2}_{v}}
\end{align*}
for all $v \in \Usf(\Gamma)$, $X_i \in F_i(v)$ non-zero, and $t \geq 0$.  
\end{itemize}
\end{theorem}

Here, we may think of $F_1$ and $F_2$ as $\Gamma$-invariant sub-bundles of $E$. The maps $F_1$ and $F_2$ are related to the limit maps $\xi_\rho^k$ and $\xi_\rho^{d-k}$ by
\[F_1(v) = \xi_\rho^k(v^+)\quad\text{and}\quad F_2(v) = \xi_\rho^{d-k}(v^-)\]
for all $v = (v^+,v^-,s)\in\Usf(\Gamma)$. 

\section{$\rho$-controlled sets}\label{sec:rho_controlled_sets}

In this section we introduce the notion of $\rho$-controlled sets, and describe some of their properties. 

\begin{definition}\label{def:controlled}Suppose  $\rho : \Gamma \rightarrow \PGL_d(\Rb)$ is a $P_1$-Anosov representation. A subset $M\subset\Pb(\Rb^d)$ is \emph{$\rho$-controlled} if it is non-empty, closed, $\rho(\Gamma)$-invariant, and
\[M\cap\xi_\rho^{d-1}(x)=\xi_\rho^1(x)\] 
for every $x\in\partial_\infty\Gamma$. If $\rho$ also happens to be $P_m$-Anosov for some $m=2,\dots,d-1$, then a $\rho$-controlled subset $M\subset\Pb(\Rb^d)$ is \emph{$m$-hyperconvex} if 
\begin{align*}
a_1+a_2+\xi_\rho^{d-m}(x)
\end{align*}
is a direct sum for all $a_1,a_2 \in M$ and $x \in \partial_\infty \Gamma$ with $a_1, a_2, \xi_\rho^1(x)$ pairwise distinct.
\end{definition}

\begin{remark} We will typically consider the case when $M$ is a topological $(m-1)$-dimensional manifold and then require that $M$ is $m$-hyperconvex. 
\end{remark} 

The three main examples of $\rho$-controlled subsets $M\subset\mathbb{P}(\Rb^d)$ that we will be concerned with are the following. 

\begin{example}\label{eg:limitset} When $\rho : \Gamma \rightarrow \PGL_d(\Rb)$ is a $P_1$-Anosov representation, the $P_1$-limit set for $\rho$ is $\rho$-controlled. Furthermore, if $\rho$ is $P_m$-Anosov for some $m=2,\dots,d-1$, then $\xi_\rho^1(\partial_\infty\Gamma)$ is $m$-hyperconvex if and only if 
\begin{align*}
\xi_\rho^1(x)+\xi_\rho^1(z)+\xi_\rho^{d-m}(y)
\end{align*}
 is a direct sum for all pairwise distinct $x,y,z\in\partial_\infty\Gamma$.
\end{example}

\begin{example}\label{eg:convex} Suppose  $\rho:\Gamma\to\PGL_d(\Rb)$ is a $P_1$-Anosov representation and $\rho(\Gamma)$ preserves a properly convex domain $\Omega\subset\Pb(\Rb^d)$ (see Section \ref{sec:properly_convex}) such that $\xi_\rho^{d-1}(x)\cap\partial\Omega=\xi_\rho^1(x)$ for all $x\in\partial_\infty\Gamma$. Then $M:=\partial\Omega$ is $\rho$-controlled. Notice that in this case, the requirement that $M$ is $(d-1)$-hyperconvex is simply that 
\begin{align*}
a_1+a_2+\xi_\rho^1(y)
\end{align*}
is a direct sum for all pairwise distinct $a_1,a_2, \xi_\rho^1(y)\in\partial\Omega$. This is satisfied if and only if $\xi_\rho^1(\partial_\infty \Gamma)$ does not intersect any proper line segments in $\partial \Omega$.
\end{example}

\begin{example}\label{eg:limitset_subgroup} Suppose $\rho : \Gamma \rightarrow \PGL_d(\Rb)$ is a $P_1$-Anosov representation and $\Gamma_1 \leq \Gamma$ is a quasi-convex subgroup. Then $\rho_1:=\rho|_{\Gamma_1}: \Gamma_1 \rightarrow \PGL_d(\Rb)$ is also $P_1$-Anosov, and the $P_1$-limit set of $\rho$ is $\rho_1$-controlled. Furthermore, if $\rho_1$ is $P_m$-Anosov for some $m=2,\dots,d-1$, then $P_1$-limit set of $\rho$ is $m$-hyperconvex if and only if 
\begin{align*}
\xi_\rho^1(x)+\xi_\rho^1(z)+\xi_\rho^{d-m}(y)
\end{align*}
 is a direct sum for all pairwise distinct $x,y,z \in\partial_\infty\Gamma$ with $y \in \partial_\infty \Gamma_1$. 
\end{example}

For any $(x,y)\in\partial_\infty\Gamma^{(2)}$, let $L_{x,y}$ denote the orbit $(x,y,\Rb)\subset\Usf(\Gamma)$ of $\varphi_t$. The following proposition is one of the key tools we use to investigate regularity properties of $\rho$-controlled subsets.

\begin{proposition}\label{prop:rho_controlled_sets_projections} Suppose  $\rho : \Gamma \rightarrow \PGL_d(\Rb)$ is a $P_1$-Anosov representation and $M \subset \Pb(\Rb^d)$ is $\rho$-controlled. Then there exists a family of maps 
\begin{align*}
\pi_{x,y} : M - \left\{ \xi_\rho^1(x), \xi_\rho^1(y) \right\} \rightarrow L_{x,y}
\end{align*}
indexed by $(x,y) \in \partial_\infty \Gamma^{(2)}$ such that 
\[
\gamma\circ\pi_{x,y} =\pi_{\gamma\cdot x,\gamma\cdot y}\circ \rho(\gamma), \quad \pi_{x,y}(a)= \lim_{(u,v,b)\to (x,y,a)} \pi_{u,v}(b),
\]
\[
x  = \lim_{(u,v,b)\to (x,y,\xi_\rho^1(x))}  \pi_{u,v}(b),\quad\text{and}\quad y  = \lim_{(u,v,b)\to (x,y,\xi_\rho^1(y))} \pi_{x,y}(b).\]
for all $(x,y) \in \partial_\infty \Gamma^{(2)}$, $\gamma\in\Gamma$, and $a\in M - \left\{ \xi_\rho^1(x), \xi_\rho^1(y) \right\}$.
\end{proposition}

Delaying the proof of Proposition~\ref{prop:rho_controlled_sets_projections} until Section \ref{sec:controlled_proof}, we describe the main application. Suppose that $\rho : \Gamma \rightarrow \PGL_d(\Rb)$ is a $P_1$-Anosov representation and $M \subset \Pb(\Rb^d)$ is $\rho$-controlled. Let $\{ \pi_{x,y} : (x,y)\in \partial_\infty \Gamma^{(2)}\}$ be a family of maps satisfying Proposition~\ref{prop:rho_controlled_sets_projections}. Then define 
\begin{align}\label{eqn:P(M)}
P(M) : = \left\{ (v,a) \in \Usf(\Gamma) \times M : a \in \pi_{v^-,v^+}^{-1}(v) \right\}.
\end{align}
Notice that there is a natural $\Gamma$ action on $P(M)$ given by 
\begin{align*}
\gamma \cdot (v,a) = (\gamma \cdot v, \rho(\gamma)\cdot a).
\end{align*}

 In the next observation, let  $d_{\Pb}$ denotes the distance function of some Riemannian metric on $\Pb(\Rb^d)$.

\begin{observation}\label{obs:compact} With the notation above, 
\begin{enumerate}
\item $\Gamma$ acts co-compactly on $P(M)$, 
\item for any $v \in \Usf(\Gamma)$ and $a \in M - \{ \xi_\rho^1(v^+), \xi_\rho^1(v^-)\}$, there exists $t \in \Rb$ such that $(\varphi_t(v), a) \in P(M)$,
\item for any compact set $K \subset \Usf(\Gamma)$ there exists $\delta > 0$ such that: If $v \in K$ and $a \in M- \{ \xi_\rho^1(v^+)\}$ satisfy $d_{\Pb}\left(\xi_\rho^1(v^+), a\right) \leq \delta$, then $(\varphi_t (v), a) \in P(M)$ for some $t > 0$. 
\end{enumerate}
\end{observation}

\begin{proof}
(1): Since the $\Gamma$-action on $\Usf(\Gamma)$ is co-compact, there exists a compact set $K \subset \Usf(\Gamma)$ such that $\Gamma \cdot K = \Usf(\Gamma)$. Then the set 
\begin{align*}
 \wh{K} : = \left\{(v,a) \in K \times M:a\in \pi_{v^-,v^+}^{-1}(v)  \right\}
\end{align*}
is compact. Indeed, if $\{(v_n,a_n)\}$ is a sequence in $\wh K$, then because $K\times M$ is compact, it converges, up to taking subsequence, to some $(v_\infty,a_\infty)\in K\times M$. By Proposition~\ref{prop:rho_controlled_sets_projections}, if $a_\infty=\xi_\rho^1(v_\infty^\pm)$, then $v_n\to v_\infty^\pm$, which contradicts the assumption that $v_n\to v_\infty$. Thus, $a_\infty\neq \xi_\rho^1(v_\infty^\pm)$ and so $\pi_{v_\infty^-,v_\infty^+}(a_\infty)$ is well defined. By Proposition~\ref{prop:rho_controlled_sets_projections},
\[\pi_{v_\infty^-,v_\infty^+}(a_\infty)=\lim_{n\to+\infty} \pi_{v_n^-,v_n^+}(a_n)=\lim_{n\to+\infty}v_n=v_\infty,\]
so $(v_\infty,a_\infty)\in\wh K$.

Since $\wh K$ is compact, and $\Gamma\cdot \wh{K} = P(M)$ by definition, (1) holds.

(2): Follows directly from the definition. 

(3): Fix a compact set $K \subset \Usf(\Gamma)$. If such a $\delta > 0$ does not exist, then for every $n \geq 1$ there exist $v_n \in K$, $a_n \in M- \{ \xi_\rho^1(v_n^+)\}$, and $t_n \leq 0$ such that  
\begin{align*}
d_{\Pb}\left(\xi_\rho^1(v_n^+), a_n\right) \leq 1/n
\end{align*}
and $\varphi_{t_n}(v_n)=\pi_{v_n^-,v_n^+}(a_n) $. By passing to a subsequence we can suppose that $v_n \rightarrow v_\infty \in K$. But then $a_n \rightarrow \xi_\rho^1(v_\infty^+)$ as $n\to+\infty$, so by Proposition~\ref{prop:rho_controlled_sets_projections},
\begin{align*}
v_\infty^+ = \lim_{n \rightarrow +\infty} \pi_{v_n^-,v_n^+}(a_n) = \lim_{n \rightarrow +\infty} \varphi_{t_n}(v_n) \in L_{v_\infty^+, v_\infty^-} \cup \{ v_\infty^-\}
\end{align*}
which is a contradiction.
\end{proof}

\begin{remark} The set $P(M)$ is designed to be a generalization of the following construction: Suppose $X$ a compact negatively curved Riemannian manifold, $\wt{X}$ is the universal cover of $X$,  $\pi: T^1 \wt{X}\to\wt X$ is the usual projection, and $\varphi_t$ is the geodesic flow on $T^1 \wt{X}$. Then define \begin{align*}
P(\partial_\infty\wt X) \subset T^1 \wt{X} \times \partial_\infty \wt X
\end{align*}
to be the set of pairs $(v,a)$ such that there exists $w \in T^1_{\pi(v)} \wt{X}$ with $w \bot v$ and $\lim_{t \rightarrow +\infty} \pi(\varphi_tw) = a$.
\end{remark}

\subsection{Properly convex domains}\label{sec:properly_convex}
We now describe properly convex domains and some of their relevant properties. These will be used to prove Proposition~\ref{prop:rho_controlled_sets_projections}.

An open set $\Omega\subset\Pb(\Rb^d)$ is a \emph{properly convex domain} if its closure lies in an affine chart in $\Pb(\Rb^d)$, and it is convex, i.e. for distinct $x,y\in\Omega$, there is a projective line segment in $\Omega$ whose endpoints are $x$ and $y$. Given a properly convex domain $\Omega\subset \Pb(\Rb^d)$, we denote \begin{align*}
\Aut(\Omega) := \{ g \in \PGL_d(\Rb) : g\cdot \Omega = \Omega\}.
\end{align*}

Every properly convex domain $\Omega\subset\Pb(\Rb^d)$ admits a natural metric which is defined as follows. For any pair of points $x,y\in\Omega$, let $l$ be a projective line through $x$ and $y$, and let $a$ and $b$ be the two points of intersection of $l$ with $\partial\Omega$, ordered such that $a<x\leq y<b$ lie along $l$. Then define
\[H_\Omega(x,y):=\log C(a,x,y,b),\]
where, $C$ is the cross ratio along the projective line $l$. Since $a<x\leq y<b$ along $l$, 
\[C(a,x,y,b)=\frac{\norm{a-y}\norm{b-x}}{\norm{a-x}\norm{b-y}},\]
where $\norm{\cdot}$ is a norm on some (equiv. any) affine chart of $\Pb(\Rb^d)$ containing the closure of $\Omega$. One can verify from properties of the cross ratio that the map $H_\Omega:\Omega\times\Omega\to\Rb^+\cup\{0\}$ is a continuous distance function. This is commonly known as the \emph{Hilbert metric} on $\Omega$. 

Let $T\Omega$ denote the tangent bundle of $\Omega$. The Hilbert metric $H_\Omega$ on $\Omega$ is a Finsler metric, i.e. it is infinitesimally given by a continuous family of norms on the fibers of $T^1\Omega$. This norm, denoted
\[h_\Omega:T\Omega\to\Rb,\] 
can be described as followed. First, fix an affine chart $\Ab \subset \Pb(\Rb^d)$ which contains $\Omega$ and make an affine identification $\Ab = \Rb^{d-1}$. Then we can identify 
$$
T\Omega = \Omega \times \Ab =\Omega \times \Rb^{d-1}.
$$
Let $\pi : T\Omega \rightarrow \Omega$ denote the natural projection, and let $Z\subset T\Omega$ denote the image of the zero section. For any $v\in T\Omega-Z$, let $l_v$ denote the oriented projective line segment in $\Omega$ through $\pi(v)$ in the direction given by $v$, and with forward and backward endpoints $v^+,v^-\in \partial\Omega$ respectively. Then $h_\Omega$ is defined by 
\[h_\Omega(v)= \left\{\begin{array}{ll}
0&\text{if }v\in Z,\\
\norm{v}_2\left(\frac{1}{\norm{\pi(v)-v^+}_2}+\frac{1}{\norm{\pi(v)-v^-}_2}\right)&\text{if }v\in T\Omega-Z.
\end{array}\right.\]
With this, define the \emph{unit tangent bundle} of $\Omega$ to be
\begin{align*}
T^1 \Omega := \{ v \in T\Omega : h_\Omega(v) = 1\}.
\end{align*}

We recall the definition of a convex co-compact action on $\Omega$.

\begin{definition}\label{defn:cc} Suppose  $\Omega$ is a properly convex domain. A discrete subgroup $\Gamma \leq \Aut(\Omega)$ \emph{acts convex co-compactly} on $\Omega$ if there exists a closed non-empty $\Gamma$-invariant convex subset $\Cc \subset \Omega$ such that the quotient $\Gamma \backslash \Cc$ is compact.
 \end{definition} 
 
 \begin{remark} This is not the definition of convex co-compactness used in~\cite{DCG17}, instead they say groups satisfying Definition~\ref{defn:cc} act \emph{naive convex co-compactly}. \end{remark}
 
We now use work of Danciger-Gu\'eritaud-Kassel~\cite{DCG17} and the second author~\cite{Zimmer17} to construct a convex co-compact action.

 \begin{theorem}\label{thm:cc_action}\cite[Theorem 1.4]{DCG17}, \cite[Theorem 1.27]{Zimmer17} Suppose  $\rho : \Gamma \rightarrow \PGL_d(\Rb)$ is a $P_1$-Anosov representation and there exists a properly convex domain $\Omega_0 \subset \Pb(\Rb^d)$ with $\rho(\Gamma) \leq \Aut(\Omega_0)$. Then there is a properly convex domain $\Omega\subset\Pb(\Rb^d)$ such that
 \begin{enumerate}
 \item $\xi_\rho^1(\partial_\infty \Gamma) \subset \partial \Omega$, 
 \item the convex hull $\overline{\Cc}$ of $\xi_\rho^1(\partial_\infty \Gamma)$ in $\overline{\Omega}$ satisfies 
 \[\overline{\Cc} \cap \partial \Omega = \xi_\rho^1(\partial_\infty \Gamma),\]
 \item  if we set $\Cc:=\overline{\Cc}-\xi_\rho^1(\partial_\infty \Gamma)$, then $\Gamma\backslash \Cc$ is compact and in particular $\rho(\Gamma)$ acts convex co-compactly on $\Omega$,
 \item for every $x,y \in \partial_\infty \Gamma$ distinct, $\Omega_0$ and $\Omega$ are contained in the same connected component of 
\begin{align*}
\Pb(\Rb^d) - \left( \xi_\rho^{d-1}(x) \cup \xi_\rho^{d-1}(y) \right).
\end{align*}
 \end{enumerate}
 \end{theorem}
 
 \begin{remark} \
 \begin{enumerate}
 \item In~\cite[Theorem 1.27]{Zimmer17} it is assumed that $\rho$ is irreducible. 
 \item To be precise, the various assertions in Theorem~\ref{thm:cc_action} follow (essentially) from the proofs of \cite[Theorem~1.4]{DCG17} and \cite[Theorem~1.27]{Zimmer17}. 
 \begin{enumerate}
 \item For part (1), see Proposition 3.10 part (1) in~\cite{DCG17} (notice that the proximal limit set coincides with the image of $\xi_\rho^1$). 
 \item For part (2), see Lemmas 8.4 and 8.6 in~\cite{DCG17}.
  \item For part (3), see Lemma 8.7 in~\cite{DCG17}.
 \item For Part (4), see Proposition 3.10 part (3) in~\cite{DCG17}. 
 \end{enumerate} 
 \item In general, if $\rho(\Gamma)\le\Aut(\Omega_0)$, then $\xi_\rho^1(\partial_\infty\Gamma)$ lies in $\partial\Omega_0$, but the convex hull in $\overline{\Omega}_0$ of $\xi_\rho^1(\partial_\infty\Gamma)$ might intersect $\partial\Omega_0$ along a much larger set than $\xi_\rho^1(\partial_\infty\Gamma)$ (for example, if $\Omega_0$ is the set $\mathcal C$ in Theorem \ref{thm:cc_action} part (3)). The properly convex domain $\Omega$ given by Theorem \ref{thm:cc_action} can be thought of as an ``enlargement" of $\Omega_0$ to remove this problem.
 \end{enumerate}
 \end{remark}

\subsection{The proof of Proposition~\ref{prop:rho_controlled_sets_projections}}\label{sec:controlled_proof}

We will prove Proposition~\ref{prop:rho_controlled_sets_projections} by constructing a projective model of the geodesic flow space $\Usf(\Gamma)$. This construction has several steps: first we post compose to obtain a new $P_1$-Anosov representation that preserves a properly convex domain. Theorem \ref{thm:cc_action} then gives us a convex co-compact action, which we then use to construct a projective model of the geodesic flow space. Finally we use this projective model to construct the maps $\pi_{x,y}$. 

\subsubsection{Constructing an invariant properly convex domain}\label{subsec:constructing_properly_convex_domain}

In general, a $P_1$-Anosov representation might not preserve a properly convex domain. For example, Danciger-Gu\'eritaud-Kassel \cite[Proposition 1.7]{DCG17} proved that if $d$ is even and $\rho : \pi_1(S) \rightarrow \PGL_d(\Rb)$ is Hitchin (see Definition~\ref{defn:hitchin_reps}), then $\rho(\pi_1(S))$ does not preserve a properly convex domain. However, it is well-known that after post composing with another representation we can always find an invariant properly convex domain, see Proposition \ref{prop :S_composition} for a precise statement.  We recall the proof of this well-known fact for the convenience of the reader.

Denote the vector space of symmetric 2-tensors by $\Sym_2(\Rb^d)$ and let 
\[D :=\dim \,\Sym_2(\Rb^d)=\frac{d(d+1)}{2}.\] 
Then let $\overline{S} : \GL_d(\Rb) \rightarrow \GL(\Sym_2(\Rb^d))\cong\GL_D(\Rb)$ be the representation determined by
\begin{align*}
\overline{S}(g)(v \otimes v) = gv \otimes gv,
\end{align*}
and note that $\overline{S}$ descends to a representation $S:\PGL_d(\Rb)\to\PGL_D(\Rb)$. 
Associated to $S$ are smooth embeddings 
\[\Phi : \Pb(\Rb^d) \rightarrow \Pb(\Sym_2(\Rb^d))\quad \text{and}\quad\Phi^* : \Gr_{d-1}(\Rb^d) \rightarrow \Gr_{D-1}(\Sym_2(\Rb^d))\] 
defined by
\begin{align*}
\Phi([v]) = [v \otimes v]
\end{align*}
and
\begin{align*}
\Phi^*(W) = \Span\left\{ v \otimes w + w \otimes v : w \in W, v \in \Rb^d \right\}.
\end{align*}
Equivalently, $\Phi^*(W)=\Phi(W^\perp)^\perp$ with respect to the standard inner product on $\Rb^d$ and the induced inner product on $\Sym_2(\Rb^d)=\Rb^D$. Notice that $\Phi$ and $\Phi^*$ are both $S$-equivariant. 

Next, we define a properly convex domain in $\Pb(\Sym_2(\Rb^d))$ which is invariant under the action of $S(\PGL_d(\Rb))$. 
Given $X \in \Sym_2(\Rb^d)$ we say that $X$ is \emph{positive definite}, and write $X > 0$, if $(f\otimes f)(X) > 0$ for every $f \in (\Rb^{d})^*-\{0\}$. Also, we say that $X$ is \emph{positive semidefinite}, and write $X \geq 0$, if $(f\otimes f)(X) \geq 0$ for every $f \in (\Rb^{d})^*$. Then define 
\begin{align*}
\Pc^+ := \left\{ [X] : X \in \Sym_2(\Rb^d), X > 0\right\},
\end{align*}
in which case
\begin{align*}
\overline{\Pc^+} = \left\{ [X] : X \in \Sym_2(\Rb^d), X \ge 0\right\}.
\end{align*}

\begin{observation}\label{obs:PD_matrices} In the notation above,
 \begin{enumerate}
\item $\Pc^+$ is a properly convex domain in $\Pb(\Sym_2(\Rb^d))$,
\item $S(\PGL_d(\Rb)) \le \Aut(\Pc^+)$,
\item $\Phi(\Pb(\Rb^d)) \subset \overline{\Pc^+}$. 
\end{enumerate}
\end{observation}

\begin{proof} 
(1): Clearly $C := \{ X : X\in \Sym_2(\Rb^d), X > 0\}$ is a convex open cone in $\Sym_2(\Rb^d)$. Since 
\begin{align*}
(f\otimes f)(X+tY) = (f\otimes f)(X)+t(f\otimes f)(Y)
\end{align*}
for all $X,Y\in\Sym_2(\Rb^d)$, all $t\in\Rb$ and all $f \in (\Rb^{d})^*$, $\overline{C}$ does not contain any real affine lines. Thus $C$ is a sharp convex cone. Since $C$ projectivizes to $\Pc^+$ we see that $\Pc^+$ is a properly convex domain.

(2): Notice that for all $f \in \Rb^{d*}$, $g \in \GL_d(\Rb)$, and $X \in\Sym_2(\Rb^d)$
\begin{align*}
(f\otimes f)(\overline{S}(g)X) = \Big((f \circ g) \otimes (f \circ g) \Big)(X).
\end{align*}
This implies that $S(\PGL_d(\Rb)) \leq \Aut(\Pc^+)$.

(3): Suppose that $[v] \in \Pb(\Rb^d)$. Then $\Phi([v]) = [v \otimes v]$ and 
\begin{align*}
(f\otimes f)(v \otimes v) = f(v)f(v) \geq 0
\end{align*}
for all $f \in \Rb^{d*}$. Thus $\Phi([v]) \subset \overline{\Pc^+}$. 
\end{proof}

The next proposition tells us that up to post composing any $P_1$-Anosov representation with $S$, we may assume that it leaves invariant a properly convex domain.

\begin{proposition}\label{prop :S_composition} If $\rho : \Gamma \rightarrow \PGL_d(\Rb)$ is $P_1$-Anosov, then $S\circ\rho$ is $P_1$-Anosov with boundary maps $\xi_{S\circ\rho}^1=\Phi \circ \xi_\rho^1$ and $\xi_{S\circ\rho}^{D-1}=\Phi^* \circ \xi_\rho^{d-1}$. Moreover,  $(S \circ \rho)(\Gamma)\leq \Aut(\Pc^+)$. 
\end{proposition}

\begin{proof} The maps $\Phi \circ \xi_\rho^1$ and $\Phi^* \circ \xi_\rho^{d-1}$ are clearly $S\circ\rho$-equivariant and transverse. If $\gamma\in\Gamma$, then $\lambda_1(S\circ\rho(\gamma)) = \lambda_1(\rho(\gamma))^2$ and $\lambda_2(S\circ\rho(\gamma)) = \lambda_1(\rho(\gamma)) \lambda_2(\rho(\gamma))$, so 
\begin{align*}
\frac{\lambda_1(S\circ\rho(\gamma))}{\lambda_2(S\circ\rho(\gamma))} = \frac{\lambda_1(\rho(\gamma))}{\lambda_2(\rho(\gamma))}.
\end{align*}
The assumption that $\xi_\rho^1$ and $\xi_\rho^{d-1}$ are dynamics preserving now imply that $\Phi \circ \xi_\rho^1$ and $\Phi^* \circ \xi_\rho^{d-1}$ are both dynamics preserving. Thus, $S\circ\rho$ is also $P_1$-Anosov. The fact that the image of $S\circ\rho$ lies in $\Aut(\Pc^+)$ follows from Observation~\ref{obs:PD_matrices}.
\end{proof}

 \subsubsection{Constructing a projective geodesic flow}\label{sec:proj_geod_flow}
Suppose that $\rho : \Gamma \rightarrow \PGL_d(\Rb)$ is a $P_1$-Anosov representation and $\rho(\Gamma)\subset\Aut(\Omega_0)$ for some properly convex domain $\Omega_0\subset\Pb(\Rb^d)$. 
Let $\Omega\subset\Pb(\Rb^d)$ be the properly convex domain given by Theorem~\ref{thm:cc_action}. 

Since every projective line segment in $\Omega$ is a geodesic in the Hilbert metric, $T^1 \Omega$ has a natural geodesic flow, denoted by $\psi_t$, obtained by flowing along the projective line segments at unit speed. Using this flow we can construct a model of the flow space $\Usf(\Gamma)$, which we call the \emph{projective geodesic flow}.

For distinct $x,y \in \partial_\infty \Gamma$, let $\ell_{x,y} \subset T^1 \Omega$ be the set of unit tangent vectors that have base points in the line segment joining $\xi_\rho^1(x)$ to $\xi_\rho^1(y)$ and that point towards $\xi_\rho^1(y)$. Note that $\ell_{x,y}$ exists because $\Omega$ satisfies property (2) of Theorem~\ref{thm:cc_action}. Then the set
\begin{align*}
\Gc : = \bigcup_{(x,y) \in \partial_\infty \Gamma^{(2)}}\ell_{x,y} 
\end{align*}
is invariant under the action of $\rho(\Gamma)$, the flow $\psi_t$, and the $\{\pm1\}$ action on $T^1 \Omega$ given by $(-1)\cdot v \mapsto -v$. Notice that 
$$
\psi_t( (-1) \cdot v) = (-1) \cdot \psi_{-t}(v)
$$
for all $v \in T^1 \Omega$. So the actions of $\Gamma$, $\psi_t$, and $\{\pm 1\}$ combine to yield an action of $\Gamma \times (\Rb \rtimes_{\psi}\{\pm 1\})$ on $\Gc$, and as above, $\psi : \{\pm 1\} \rightarrow \Aut(\Rb)$ is given by $\psi(-1)(t) = -t$.

Furthermore, the projection $\pi:T^1\Omega\to\Omega$ restricts to a map $\pi|_{\Gc}:\Gc\to\Cc$ that is equivariant with respect to the $\rho(\Gamma)$-actions on $\Gc$ and $\Cc$. (Recall that $\Cc=\overline{\Cc}-\xi_\rho^1(\partial_\infty\Gamma)$, where $\overline{\Cc}$ is the convex hull of $\xi_\rho^1(\partial_\infty\Gamma)$ in $\overline\Omega$.)

Using Theorem~\ref{thm:Gromov_uniqueness} we will deduce the existence of a homeomorphism $\Gc \rightarrow \Usf(\Gamma)$.

\begin{corollary}\label{cor:Gromov_model} With the notation above, there exists a homeomorphism $T : \Gc \rightarrow \Usf(\Gamma)$ with the following properties:
\begin{enumerate}
\item $T$ is equivariant relative to the $\Gamma$ and $\{\pm1\}$ actions,
\item for every $(x,y) \in \partial_\infty \Gamma^{(2)}$, $T$ maps the flow line $\ell_{x,y}$ in $\Gc$ to the flow line $L_{x,y}$ in $\Usf(\Gamma)$.
\end{enumerate}
\end{corollary}

\begin{proof} As observed above, $\Gamma \times (\Rb \rtimes_{\psi} \{\pm1\})$ acts on $\Gc$ and  further the $\Rb$ action is free.  By Theorem~\ref{thm:Gromov_uniqueness}, it suffices to construct a proper hyperbolic metric $d$ on $\Gc$ with the following properties: 
\begin{enumerate}[(a)]
\item the actions of $\Gamma$ and $\{\pm 1\}$ are isometric, 
\item for every $v \in \Gc$, the orbit map $\Gamma\to\Gc$ given by $\gamma\mapsto\gamma \cdot v$ is a quasi-isometry and this orbit map extends to a $\Gamma$-equivariant homeomorphism $\partial_\infty \Gamma \rightarrow \partial_\infty \Gc$,
\item every $\Rb$-orbit is a geodesic in $\Gc$, and
\item the induced map $\Gc / \Rb \rightarrow \partial_\infty \Gc^{(2)}$ is a homeomorphism.
\end{enumerate}

Let $\mu$ be a probability measure on $\Rb$ with finite first moment, i.e. 
\[\int_{\Rb}|t|\ \mathrm{d}\mu(t) <\infty.\] 
(For example, we may take $\mu$ given by $\mathrm{d}\mu=\frac{1}{\sqrt{\pi}}e^{-t^2}\ \mathrm{d} t$.) Note that
\[0\le H_\Omega(\pi(\psi_t(v)),\pi(\psi_t(w)))\le H_\Omega(\pi(v),\pi(w))+2|t|\]
for all $v,w\in \Gc$ and all $t\in\Rb$, so we may define $d:\mathcal{G}\times\mathcal{G}\to\Rb$ by
\begin{align*}
d(v,w) := \int_{\Rb} H_\Omega(\pi(\psi_t(v)),\pi(\psi_t(w)))\ \mathrm{d}\mu(t).
\end{align*}
It is straight forward to check that $d$ is a proper metric. Also, Conditions (a) and (c) follow immediately from the definition of $d$. It remains to show that $(\Gc,d)$ is hyperbolic and that conditions (b) and (d) hold. 

Since $(\Cc,H_\Omega)$ is a geodesic metric space, we may apply the Milnor-\v{S}varc Lemma to deduce that for every $p\in\Cc$, the orbit map $\Gamma\to\Cc$ given by $\gamma\mapsto\gamma\cdot p$ is a quasi-isometry. Since $\Gamma$ is a hyperbolic group, this implies that $(\Cc,H_\Omega)$ is a hyperbolic metric space. Further, the Gromov boundary $\partial_\infty \Cc$ of $(\Cc,H_\Omega)$  naturally identifies with $\xi^1_\rho(\partial_\infty \Gamma) = \overline{\Cc}-\Cc$. 

Notice that 
\begin{align*}
\abs{d(v,w)-H_\Omega(\pi(v), \pi(w))}&=\abs{\int_{\Rb} H_\Omega(\pi(\psi_t(v)),\pi(\psi_t(w)))-H_\Omega(\pi(v), \pi(w))\ \mathrm{d}\mu(t)}\\
&\le 2\int_{\Rb}|t|\ \mathrm{d}\mu(t)=:C
\end{align*}
for all $v,w\in\Gc$. So $\pi|_{\Gc} : (\Gc, d) \rightarrow (\Cc, H_\Omega)$ is a $(1,C)$-quasi-isometry. It is straightforward to verify from the four-point condition definition of Gromov hyperbolicity, see~\cite[Chapter III.H Definition 1.20]{BridsonHaefliger}, that $(1,C)$-quasi-isometries preserve Gromov hyperbolicity and so $(\Gc,d)$ is Gromov hyperbolic. Further, $\pi|_{\Gc}$ extends to a $\Gamma$-equivariant boundary map 
\begin{equation}
\label{eqn:identification of boundaries}
\partial_\infty \Gc \rightarrow \partial_\infty \Cc=\overline{\Cc}-\Cc
\end{equation}
(this is clear from the Gromov product definition of the Gromov boundary).

Since $\pi|_{\Gc}$ is $\rho(\Gamma)$-equivariant and orbits maps into $\Cc$ are quasi-isometries, we see that condition (b) holds. Further, by the construction of $\Gc$,  the map 
$$
v \in \Gc \mapsto (v^+, v^-) \in \partial_\infty \Cc^{(2)}
$$
descends to a homeomorphism $\Gc / \Rb \rightarrow \partial_\infty \Cc^{(2)}$. So by Equation~\eqref{eqn:identification of boundaries}, we see that condition (d) holds. 
\end{proof}

\subsubsection{Finishing the proof of Proposition~\ref{prop:rho_controlled_sets_projections}}

Let $\rho : \Gamma \rightarrow \PGL_d(\Rb)$ be a $P_1$-Anosov representation, and let $M \subset \Pb(\Rb^d)$ be $\rho$-controlled. By Proposition~\ref{prop :S_composition}, the representation 
\[S\circ\rho : \Gamma \rightarrow \PGL(\Sym_2(\Rb^d))\cong\PGL_D(\Rb)\] 
is $P_1$-Anosov with $\xi_{S\circ\rho}^1 : = \Phi \circ \xi_\rho^1$ and $\xi_{S\circ\rho}^{D-1} : = \Phi^* \circ \xi_\rho^{d-1}$, and
\begin{align*}
S\circ\rho(\Gamma) \subset \Aut(\Pc^+).
\end{align*}
Then by Theorem~\ref{thm:cc_action} there exists a properly convex domain $\Omega \subset \Pb(\Sym_2(\Rb^d))$ where $S\circ\rho(\Gamma)$ acts convex co-compactly on $\Omega$, and
\begin{align*}
\overline{\Cc}_{S\circ\rho} \cap \partial \Omega = \xi_{S\circ\rho}^1(\partial_\infty \Gamma).
 \end{align*}
 
 
For any distinct points $x,y \in \partial_\infty \Gamma$, define a projection (see Figure \ref{fig:projection})
\begin{align*}
p_{x,y} : \Pb\left(\Sym_2(\Rb^d)\right) - \left( \xi_{S\circ\rho}^{D-1}(x) \cap \xi_{S\circ\rho}^{D-1}(y) \right) \rightarrow \xi_{S\circ \rho}^1(x) + \xi_{S\circ\rho}^1(y)
\end{align*}
by 
\begin{align*}
p_{x,y}(b) = \Big( \xi_{S\circ \rho}^1(x) + \xi_{S\circ\rho}^1(y)\Big) \cap \Big( b + \xi_{S\circ\rho}^{D-1}(x) \cap \xi_{S\circ\rho}^{D-1}(y) \Big). 
\end{align*}

\begin{figure}[ht]
\centering
\includegraphics[scale=0.7]{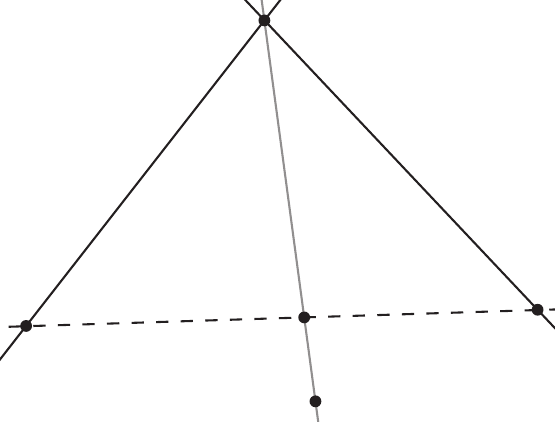}
\small
\put (-50, 87){\tiny$\xi_{S\circ\rho}^{D-1}(y)$}
\put (-7, 42){\tiny$\xi_{S\circ\rho}^1(y)$}
\put (-165, 87){\tiny$\xi_{S\circ\rho}^{D-1}(x)$}
\put (-205, 35){\tiny$\xi_{S\circ\rho}^1(x)$}
\put (-165, 134){\tiny $\xi_{S\circ\rho}^{D-1}(x)\cap \xi_{S\circ\rho}^{D-1}(y)$}
\put (-78, 5){\tiny$b$}
\put (-83, 39){\tiny$p_{x,y}(b)$}
\caption{The projection $p_{x,y}$.}\label{fig:projection}
\end{figure}

Since $\Phi$ and $\Phi^*$ are $S$-equivariant,
\begin{align}\label{eqn: PhiS}
p_{\gamma(x),\gamma(y)}\circ (S\circ\rho)(\gamma)=(S\circ\rho)(\gamma)\circ p_{x,y}.
\end{align}
Also, observe that for all distinct points $x,y\in\partial_\infty\Gamma$, we have
\begin{align}\label{eqn: PhiS1}
\begin{split}
\lim_{(u,v,b)\to(x,y,\xi_{S\circ\rho}^1(x))} & p_{u,v}(b)=\xi_{S\circ\rho}^1(x)\quad\text{and} \\
 \lim_{(u,v,b)\to(x,y,\xi_{S\circ\rho}^1(y))} & p_{u,v}(b)=\xi_{S\circ\rho}^1(y).
 \end{split}
\end{align}

\begin{observation}\label{obs: which segment} If $a \in M - \{\xi_\rho^1(x),\xi_\rho^1(y) \}$, then $p_{x,y}(\Phi(a))$ lies in the line segment joining $\xi_{S\circ\rho}^1(x)$ to $\xi_{S\circ\rho}^1(y)$ in $\Omega$. \end{observation}

\begin{proof} Since $M$ is $\rho$-controlled, 
\begin{align*}
a \notin \xi_\rho^{d-1}(x) \cup \xi_\rho^{d-1}(y).
\end{align*}
Hence 
\begin{align*}
\Phi(a) \notin \xi_{S\circ\rho}^{D-1}(x) \cup \xi_{S\circ\rho}^{D-1}(y).
\end{align*}
Observation~\ref{obs:PD_matrices} implies that $\Phi(a) \in \overline{\Pc^+}$ and part (3) of Theorem~\ref{thm:cc_action} says that $\Pc^+$ and $\Omega$ are in the same connected component of 
\begin{align*}
\Pb\left(\Sym_2(\Rb^d)\right) - \left( \xi_{S\circ\rho}^{D-1}(x) \cup \xi_{S\circ\rho}^{D-1}(y) \right).
\end{align*}
Hence $p_{x,y}(\Phi(a))$ lies in the line segment joining $\xi_{S\circ\rho}^1(x)$ to $\xi_{S\circ\rho}^1(y)$ in $\Omega$.
\end{proof}

Recall that $\Gc \subset T^1 \Omega$ is the projective model of the geodesic flow constructed in Section~\ref{sec:proj_geod_flow} and $T : \Gc \rightarrow \Usf(\Gamma)$ is the homeomorphism in Corollary~\ref{cor:Gromov_model}. By Observation \ref{obs: which segment}, for any distinct $x,y \in \partial_\infty \Gamma$, we may define a map
\begin{align*}
\wh{p}_{x,y} : M - \{\xi_\rho^1(x),\xi_\rho^1(y) \} \rightarrow \ell_{x,y}\subset\Gc
\end{align*}
by setting $\wh{p}_{x,y}(m)$ to be the unit vector at $p_{x,y}(\Phi(m))$ pointing towards $\xi_{S\circ\rho}^1(y)$. Then define
\begin{align*}
\pi_{x,y} : M - \left\{ \xi_\rho^1(x), \xi_\rho^1(y) \right\} \rightarrow L_{x,y}\subset\Usf(\Gamma)
\end{align*}
by $\pi_{x,y} = T \circ \wh{p}_{x,y}$.

By part (1) of Corollary~\ref{cor:Gromov_model}, $T$ is equivariant relative to the $\Gamma$ actions on $\Gc$ and $\Usf(\Gamma)$. Since $\Phi$ is also $S$-equivariant, it follows from \eqref{eqn: PhiS} that for all distinct $x,y\in \partial_\infty \Gamma$ and $\gamma\in\Gamma$
\begin{align*}
\gamma\circ \pi_{x,y}= \pi_{\gamma\cdot x,\gamma\cdot y}\circ \rho(\gamma)
\end{align*}
as maps from $M$ to $\Usf(\Gamma)$. The continuity of $T$, $\xi_{S\circ\rho}^1$, and $\xi_{S\circ\rho}^{D-1}$ imply that
\[\pi_{x,y}(a) = \lim_{(u,v,b)\to (x,y,a)} \pi_{u,v}(b)\]
for all $(x,y) \in \partial_\infty \Gamma^{(2)}$ and $a\in M - \left\{ \xi_\rho^1(x), \xi_\rho^1(y) \right\}$. Further, it follows from \eqref{eqn: PhiS1} and part (2) of Corollary~\ref{cor:Gromov_model} that
\begin{align*}
\lim_{(u,v,p)\to (x,y,\xi_\rho^1(x))} & \pi_{u,v}(p)=x \quad  \text{and} \quad \lim_{(u,v,p)\to (x,y,\xi_\rho^1(y))} & \pi_{x,y}(p)=y
\end{align*}
for all $(x,y) \in \partial_\infty \Gamma^{(2)}$. This proves Proposition~\ref{prop:rho_controlled_sets_projections}.

\section{Sufficient conditions for differentiability of $\rho$-controlled subsets}\label{sec:suff_cond_diff}

The goal of this section is to prove a common generalization (see Theorem \ref{thm:main_body}) of Theorem~\ref{thm:main} and the first part of Theorem~\ref{thm:regularity2}. For this generalization, we introduce a quantity  $\Fs_m(\rho)$  and prove that it is a lower bound for the supremum of the H\"older constants. This will later be used in the proofs of Theorem \ref{thm:regularity} and the second part of Theorem \ref{thm:regularity2}. Unlike $\Es_m(\rho)$ defined in the introduction, the quantity $\Fs_m(\rho)$ is not defined using the spectral data of $\rho$, but the contraction properties of the bundle $E_\rho$ defined in Section \ref{sec: dominated splittings}. However, we will later show that $\Es_m(\rho)=\Fs_m(\rho)$ if $\rho$ is $P_{1,m}$-Anosov and if the Zariski closure of its image is semisimple (see Theorem \ref{thm:alphas}).

\subsection{The relative exponential growth rate}\label{sec:optimal1}

Let $\rho:\Gamma\to\PGL_d(\Rb)$ be a  $P_{1,m}$-Anosov representation for some $m=2,\dots,d-1$. We will first describe the quantity $\Fs_m(\rho)$. Let $\overline{\rho}:\Gamma\to\SL_d(\Rb)$ be a lift of $\rho$ (see Remark~\ref{rem:lift}), and let $\norm{\cdot}_{v\in\Usf(\Gamma)}$ be a continuous family of norms on $\Rb^d$ parameterized by $\Usf(\Gamma)$, such that 
\[\norm{\overline{\rho}(\gamma)\cdot X}_{\gamma\cdot v}=\norm{X}_v\] 
for all $\gamma\in\Gamma$, $v\in\Usf(\Gamma)$, and $X\in \Rb^d$. For any $v=(v^+,v^-,s) \in \Usf(\Gamma)$, let 
\begin{align}
E_1(v) & := \xi_\rho^1(v^+),\nonumber \\
E_2(v) & := \xi_\rho^{d-1}(v^-) \cap \xi_\rho^m(v^+),\label{eqn:E}\\
E_3(v) & := \xi_\rho^{d-m}(v^-),\nonumber
\end{align}
and define $f:\Usf(\Gamma)\times\Rb\to\Rb$ by
\begin{align}\label{eqn:f}
f(v,t):=\inf_{X_i\in S_i(v)}\left\{\log\frac{\norm{X_3}_{\varphi_t(v)}}{\norm{X_1}_{\varphi_t(v)}}\Bigg/\log\frac{\norm{X_2}_{\varphi_t(v)}}{\norm{X_1}_{\varphi_t(v)}}\right\},
\end{align}
where $S_i(v):=\{X\in E_i(v):\norm{X}_v=1\}$ for all $i=1,2,3$ (note that $S_1(v)$ consists of two points). Then define the \emph{relative exponential growth rate}
\begin{equation}\label{eqn:alpham1}
\Fs_m(\rho):=\liminf_{t\to+\infty}\inf_{v\in \Usf(\Gamma)}f(v,t).
\end{equation}
Note that $\Fs_m(\rho)$ does not depend on the choice of $\overline{\rho}$. Furthermore, since $\rho$ is $P_{1,m}$-Anosov, it follows from Theorem \ref{prop:dom_split} that $\Fs_m(\rho)$ does not depend on $\norm{\cdot}_{v\in\Usf(\Gamma)}$.

We now prove that $1<\Fs_m(\rho)<\infty$. To do so, we use the following observation.

\begin{observation} \label{lem:weak flow}
There exist $C_1 \geq1$ and $\beta_1 \geq0$ such that 
\begin{align*}
\frac{1}{C_1} e^{-\beta_1 t} \norm{X}_{v} \leq \norm{X}_{\varphi_t(v)} \leq C_1 e^{\beta_1 t} \norm{X}_v
\end{align*}
for all $v \in \Usf(\Gamma)$, $t > 0$, and $X\in \Rb^d$.
\end{observation}

\begin{proof} Since $\Gamma$ acts co-compactly on $\Usf(\Gamma)$ there exists $\beta_1 \geq 0$ such that 
\begin{align*}
e^{-\beta_1} \norm{X}_{v} \leq \norm{X}_{\varphi_t(v)} \leq e^{\beta_1} \norm{X}_v
\end{align*}
for all $v \in \Usf(\Gamma)$, $t \in [0,1]$, and $X\in \Rb^d$. Thus, if we let $C_1:=e^{\beta_1}$, then
\[
\frac{1}{C_1}e^{-t\beta_1} \norm{X}_{v} \leq \norm{X}_{\varphi_t(v)} \leq C_1e^{t\beta_1} \norm{X}_v
\]
for all $v \in \Usf(\Gamma)$, $t > 0$, and $X\in \Rb^d$.
\end{proof}


By Theorem \ref{prop:dom_split}, the assumption that $\rho$ is $P_{1,m}$-Anosov ensures that there are constants $C_2,C_3\geq 1$ and $\beta_2,\beta_3>0$ such that for all $v \in \Usf(\Gamma)$, $X_i \in S_i(v)$, and $t \geq 0$, we have
\begin{align*}
\frac{\norm{X_2}_{\varphi_t(v)}}{\norm{X_1}_{\varphi_t(v)}} \geq \frac{1}{C_2} e^{\beta_2 t}\quad\text{and}\quad\frac{\norm{X_3}_{\varphi_t(v)}}{\norm{X_2}_{\varphi_t(v)}} \geq \frac{1}{C_3} e^{\beta_3 t}.
\end{align*}
This, together with Observation \ref{lem:weak flow}, then implies that for sufficiently large $t$,
\begin{align*}
\log\frac{\norm{X_3}_{\varphi_t(v)}}{\norm{X_1}_{\varphi_t(v)}}\Bigg/\log\frac{\norm{X_2}_{\varphi_t(v)}}{\norm{X_1}_{\varphi_t(v)}}&=1+\log\frac{\norm{X_3}_{\varphi_t(v)}}{\norm{X_2}_{\varphi_t(v)}}\Bigg/\log\frac{\norm{X_2}_{\varphi_t(v)}}{\norm{X_1}_{\varphi_t(v)}}\\
&\geq 1+\frac{\beta_3 t-\log C_3}{2\beta_1 t+2\log C_1}
\end{align*}
and
\begin{align*}
\log\frac{\norm{X_3}_{\varphi_t(v)}}{\norm{X_1}_{\varphi_t(v)}}\Bigg/\log\frac{\norm{X_2}_{\varphi_t(v)}}{\norm{X_1}_{\varphi_t(v)}}&\leq1+\frac{2\beta_1 t+2\log C_1}{\beta_2 t-\log C_2}.
\end{align*}
It follows that
\[1<1+\frac{\beta_3}{2\beta_1} \le \Fs_m(\rho) \le 1+\frac{2\beta_1}{\beta_2}<\infty.\]

\subsection{Optimal bounds on H\"older constants}\label{sec:optimal1}
A topological $(m-1)$-dimensional manifold $M \subset \Pb(\Rb^d)$ is \emph{$C^{\alpha}$ along a compact subset $N\subset M$} for some $\alpha >1$ if the following holds: For every $a \in N$ there exist local smooth coordinates around $a$ and a continuous map $f:U\to\Rb^{d-m}$ for some open $U\subset \Rb^{m-1}$, such that
\begin{enumerate}
\item $M$ coincides with the graph of $f$ near $a$.
\item For every $(u,f(u)) \in N$, $f$ is differentiable at $u$.
\item There are constants $C,\delta>0$ such that 
\[
\norm{f(u+h)-f(u)-df_u(h)}_2 \leq C\norm{h}_2^\alpha
\] 
for all $u\in U$ with $(u,f(u)) \in N$ and for all $h\in\Rb^{m-1}$ with $\norm{h}_2<\delta$.
\end{enumerate}
Observe that in the case when $N=M$, $M$ is $C^\alpha$ along $M$ if and only if $M\subset\Pb(\Rb^d)$ is a $C^\alpha$-submanifold. In Appendix~\ref{sec:regularity appendix} we will record some geometric consequences of the estimate in (3).

With this, we state the main theorem of this section.

\begin{theorem}\label{thm:main_body} Let $\rho:\Gamma\to\PGL_d(\Rb)$ be $P_1$-Anosov. Suppose  $M\subset\Pb(\Rb^d)$ is a $\rho$-controlled subset that is also a topological $(m-1)$-dimensional manifold for some $m=2,\dots, d-1$. If
\begin{enumerate}
\item[($\dagger$)] $\rho$ is $P_m$-Anosov and $M$ is $m$-hyperconvex, 
\end{enumerate}
then 
\begin{enumerate}
\item [($\ddagger$)] $M$ is $C^\alpha$ along $\xi_\rho^1(\partial_\infty \Gamma)$ for all $\alpha$ such that $1<\alpha<\Fs_m(\rho)$,
\end{enumerate}
where $\Fs_m(\rho)$ is the quantity defined by \eqref{eqn:alpham1} in Section \ref{sec:optimal1}. Moreover, for all $x\in\partial_\infty\Gamma$, the tangent space to $M$ at $\xi_\rho^1(x)$ is $\xi_\rho^m(x)$. 
\end{theorem}

As mentioned in the introduction, in the special case when $M = \xi_\rho^1(\partial_\infty \Gamma)$, a version of Theorem \ref{thm:main_body} was independently proven by Pozzetti-Sambarino-Wienhard \cite{PSW18}, where they prove that $\xi_\rho^1(\partial_\infty \Gamma)$ is a $C^1$-submanifold of $\Pb(\Rb^d)$. It is clear that Theorem~\ref{thm:main} and the first part of Theorem~\ref{thm:regularity2} follow immediately from Theorem \ref{thm:main_body}, see Examples \ref{eg:limitset} and \ref{eg:convex}.

In general there is no upper bound for $\Fs_m(\rho)$. However, as a consequence of Theorem \ref{thm:main_body}, we obtain an upper bound for $\Fs_m(\rho)$ when $\rho$ satisfies the hypothesis of Theorem \ref{thm:main_body} with $M=\xi_\rho^1(\partial_\infty \Gamma)$.

\begin{corollary}\label{cor:main_body1} Let $m=2,\dots,d-1$. Suppose  $\Gamma$ is a hyperbolic group such that $\partial_\infty\Gamma$ is a $(m-1)$-dimensional topological manifold, and $\rho:\Gamma\to\PGL_d(\Rb)$ is a $P_{1,m}$-Anosov representation. If $\xi_\rho^1(\partial_\infty\Gamma)$ is $m$-hyperconvex, then $\Fs_m(\rho)\le 2$. 
\end{corollary}

\begin{proof}
Suppose for the purpose of contradiction that $\Fs_m(\rho)>2$. By Theorem~\ref{thm:main_body} there is some $\alpha>2$ such that $\xi_\rho^1(\partial_\infty\Gamma)$ is $C^\alpha$-submanifold of $\Pb(\Rb^d)$, and the tangent space to $M$ at $\xi_\rho^1(x)$ is $\xi_\rho^m(x)$ for all $x\in\partial_\infty\Gamma$. 

 If $\alpha-1>1$, then the assigment $M\to \Gr_m(\Rb^d)$ given by $a\mapsto T_aM$ is constant (see Observation \ref{obs: appendix3 estimate 2}), so $\xi_\rho^m(x)=\xi_\rho^m(y)$ for all $y\in\partial_\infty\Gamma$ that is sufficiently close to $x$. Since this is true for all $x\in\partial_\infty\Gamma$, it follows that $\xi_\rho^m$ is a constant map. However, this violates the transverality of $\xi_\rho^m$ and $\xi_\rho^{d-m}$. 
\end{proof}

The remainder of this section is the proof of Theorem \ref{thm:main_body}.

\subsection{The key inequality }\label{sec:optimal2}
Suppose  $\rho:\Gamma\to\PGL_d(\Rb)$ is $P_{1,m}$-Anosov for some $m=2,\dots,d-1$ (if $m=d-1$, then $\rho$ is just $P_1$-Anosov). Fix a Riemannian metric on $\Pb(\Rb^d)$, and let $d_{\Pb}$ be the induced distance function on $\Pb(\Rb^d)$. The following lemma is the key inequality needed to prove Theorem \ref{thm:main_body}.

\begin{lemma}\label{thm:optimal_contraction} 
Suppose  $M\subset\Pb(\Rb^d)$ is $\rho$-controlled and $m$-hyperconvex. Then for all $\alpha$ satisfying $0<\alpha<\Fs_m(\rho)$, there exists $D >0$ with the following property: for every $x \in \partial_\infty \Gamma$ and $a\in M$, we have
\begin{align}
\label{eq:inequality_main}
d_{\Pb}\left(a, \xi_\rho^m(x) \right) \leq D d_{\Pb}\left(a, \xi_\rho^1(x) \right)^{\alpha}.
\end{align}
\end{lemma}

We prove Lemma \ref{thm:optimal_contraction} via a series of observations. 

\begin{observation}\label{obs:B0}
If $0<\alpha<\Fs_m(\rho)$, then there is a constant $B\geq 1$ such that 
\begin{equation}\label{eqn:C1}
\frac{\norm{X_1}_{\varphi_t(v)}}{\norm{X_3}_{\varphi_t(v)}} \leq B \left( \frac{\norm{X_1}_{\varphi_t(v)}}{\norm{X_2}_{\varphi_t(v)}} \right)^{\alpha}
\end{equation}
for all $v\in\Usf(\Gamma)$, $t\geq 0$, and $X_i\in S_i(v)$
\end{observation}

\begin{proof}
Since $0<\alpha <\Fs_m(\rho)$, there exists $T > 0$ such that $\alpha < f(v,t)$ for all $t \geq T$ and $v \in \Usf(\Gamma)$. So
\begin{align}\label{eqn:noconstant}
\frac{\norm{X_1}_{\varphi_t(v)}}{\norm{X_3}_{\varphi_t(v)}} < \left( \frac{\norm{X_1}_{\varphi_t(v)}}{\norm{X_2}_{\varphi_t(v)}} \right)^{\alpha}
\end{align}
for all $t \geq T$, $v \in \Usf(\Gamma)$, and $X_i \in S_i(v)$. On the other hand, if $0\leq t\leq T$, then the $\Gamma$-invariance of $\norm{\cdot}$ implies that both sides of the inequality \eqref{eqn:noconstant} descend to continuous positive functions on $[0,T]\times S$, where $\widetilde{S}\subset E:=\Usf(\Gamma)\times\Rb^d$ is the fiber bundle over $\Usf(\Gamma)$ whose fiber over $v\in \Usf(\Gamma)$ is $S_1(v)\times S_2(v)\times S_3(v)$, and $S:=\widetilde{S}/\Gamma$. Since $S$ is compact, there exists some $B \geq 1$ such that 
\eqref{eqn:C1} holds. 
\end{proof}

For $i=1,2,3$ and $v \in \Usf(\Gamma)$, let $P_{i,v} : \Rb^d \rightarrow E_i(v)$ be the projection with kernel $E_{i-1}(v)+E_{i+1}(v)$, where arithmetic in the subscripts is done modulo $3$. Also let
\begin{align*}
\pi_{x,y} : M - \left\{ \xi_\rho^1(x), \xi_\rho^1(y) \right\} \rightarrow L_{x,y}
\end{align*}
be a family of maps  which satisfy Proposition~\ref{prop:rho_controlled_sets_projections}. Then, as in Section~\ref{sec:rho_controlled_sets}, define
\begin{align*}
P(M):=\left\{ (v,a) \in \Usf(\Gamma) \times M : a\in \pi_{v^+,v^-}^{-1}(v) \right\}.
\end{align*}
Choose a compact set $K\subset\Usf(\Gamma)$ such that $\Gamma\cdot K=\Usf(\Gamma)$. By enlarging $K$ if necessary, we can ensure that $\{v^+:v\in K\}=\partial_\infty\Gamma$. Let $\norm{\cdot}_2$ denote the standard $\ell^2$-norm on $\Rb^d$. 

\begin{observation}\label{obs:C}
There is a constant $C\geq 1$ such that
 \begin{equation}\label{eqn:C2}
\frac{1}{C} \leq  \frac{ \norm{ P_{i,v}\left( X \right) }_v}{\norm{P_{j,v}\left( X \right)}_v} \leq C
\end{equation}
for all $(v,[X]) \in P(M)$ and all $i,j \in \{1,2,3\}$, and
\begin{equation}\label{eqn:C3}
\frac{1}{C}  \leq \frac{\norm{X}_v }{\norm{X}_{2}}\leq C 
\end{equation}
for all $v \in K$ and non-zero $X\in\Rb^d$.
\end{observation}

\begin{proof}
Since $M$ is $\rho$-controlled and $m$-hyperconvex, it follows that if $v \in \Usf(\Gamma)$ and $a\in M - \{\xi_\rho^1(v^+), \xi_\rho^1(v^-)\}$, then  $P_{i,v}(X) \neq 0$ for all $i=1,2,3$ and all non-zero $X\in a$. Since $\norm{\cdot}_{v\in\Usf(\Gamma)}$ and the sub-bundles $E_1$, $E_2$, and $E_3$ of $E$ are $\Gamma$-invariant, we may define the function $P(M)/\Gamma\to\Rb$ given by
\begin{align*}
[(v,a)]\mapsto\frac{ \norm{ P_{i,v}\left( X \right) }_v}{\norm{P_{j,v}\left( X \right)}_v}
\end{align*}
where $X\in a$ is any non-zero vector. Note that this is a continuous, positive function on $P(M)/\Gamma$. By (1) of Observation \ref{obs:compact}, $P(M)/\Gamma$ is compact, so there exists $C\geq 1$, such that \eqref{eqn:C2} holds. Also, the function $K\times\Pb(\Rb^d)\to\Rb$ defined by 
\[(v,a)\mapsto \frac{\norm{X}_v}{\norm{X}_{2}}\] 
where $X\in a$ is any non-zero vector, is also continuous, and positive, so by further enlarging $C$ if necessary, we may assume that \eqref{eqn:C3} holds.
\end{proof}

Using the fact that $d_{\Pb}$ is induced by a Riemannian metric on $\Pb(\Rb^d)$, we have the following estimates. 

\begin{observation} \label{obs:delta1}
For any sufficiently small $\delta>0$, there exists $A \geq 1$ such that: for all $v \in K$, $a\in M$ such that $d_{\Pb}\left(\xi_\rho^1(v^+), a\right) \leq \delta$ and $X\in a$ non-zero, we have
\begin{equation}\label{eqn:obs1}
\frac{1}{A} \frac{\norm{P_{3,v}(X)}_{2}}{\norm{P_{1,v}(X)}_{2}}\leq d_{\Pb}\left( a, \xi_\rho^m(v^+) \right) \leq  A \frac{\norm{P_{3,v}(X)}_{2}}{\norm{P_{1,v}(X)}_{2}}
\end{equation}
and 
\begin{equation}\label{eqn:obs2}
 \frac{1}{A} \frac{\norm{P_{2,v}(X)}_{2}}{\norm{P_{1,v}(X)}_{2}}\leq d_{\Pb}\left(a, \xi_\rho^1(v^+)\right) \leq A \frac{\norm{P_{2,v}(X)}_{2}+\norm{P_{3,v}(X)}_{2}}{\norm{P_{1,v}(X)}_{2}}.
\end{equation}
\end{observation}

\begin{proof}
For each $v\in K$, let $\Ab_v$ denote the affine chart $\Pb(\Rb^d)-\xi_\rho^{d-1}(v^-)$, and note that $\xi_\rho^1(v^+)\in\Ab_v$. Let $\langle\cdot,\cdot\rangle_{v\in K}$ be a continuous family of inner products on $\Rb^d$ such that $P_{i,v}$ is an orthogonal projection for $i=1,2,3$, and let $\abs{\cdot}_v$ denote the induced norm on $\Rb^d$. By identifying 
\[\Ab_v\cong\{X\in\Rb^d:\abs{P_{1,v}(X)}_v=1\}/\pm 1,\]
the metric on $\Rb^d$ induced by the inner product $\langle\cdot,\cdot\rangle_v$ induces an Euclidean metric $d_{\Ab_v}$ on $\Ab_v$ where $\xi_\rho^m(v^+)\cap\Ab_v$ and $(\xi_\rho^1(v^+)+\xi_\rho^{d-m}(v^-))\cap\Ab_v$ are orthogonal.

Since $K$ is compact, there is some $\delta>0$ such that the $d_{\Pb}$-ball $B_v\subset\Pb(\Rb^d)$ of radius $\delta$ centered at $\xi_\rho^1(v^+)$ lies in $\Ab_v$. Also, since $d_{\Pb}$ is induced by a Riemannian metric on $\Pb(\Rb^d)$, it is bi-Lipschitz to $d_{\Ab_v}$ on $B_v$. Furthermore, since $K$ is compact, the bi-Lipschitz constants can be chosen to be uniform over all $v\in K$. Finally, since $K$ is compact, $\abs{\cdot}_v$ is uniformly bi-Lipschitz to $\norm{\cdot}_2$ over all $v\in K$. Thus, it suffices to prove 
\[d_{\Ab_v}\left( a, \xi_\rho^m(v^+)\cap \Ab_v \right) = \abs{P_{3,v}(X)}_v\]
and
\[  \abs{P_{2,v}(X)}_v\leq d_{\Ab_v}\left(a, \xi_\rho^1(v^+)\right) \leq  \abs{P_{2,v}(X)}_v+\abs{P_{3,v}(X)}_v.\]
for all $a\in B_v$ and $X\in a$ such that $\abs{P_{1,v}(X)}_v=1$. It is easy to see from basic Euclidean geometry that the former holds, and
\[d_{\Ab_v}\left(a, \xi_\rho^1(v^+)\right)=\sqrt{\abs{P_{2,v}(X)}_v^2+\abs{P_{3,v}(X)}_v^2}.\]
The latter then follows from the observation that $a\le\sqrt{a^2+b^2}\le a+b$ for all $a,b\ge 0$.
\end{proof}

Using Observations \ref{obs:B0}, \ref{obs:C} and \ref{obs:delta1}, we will now prove Lemma \ref{thm:optimal_contraction}.

\begin{proof}[Proof of Lemma \ref{thm:optimal_contraction} ] 
Let $\delta>0$ be sufficiently small so that Observation \ref{obs:delta1} holds. Using (3) of Observation~\ref{obs:compact} and possibly decreasing $\delta > 0$ we may also assume that for all $v \in K$ and $a \in M - \{\xi_\rho^1(v^+)\}$ satisfying $d_{\Pb}\left(\xi_\rho^1(v^+), a\right) \leq \delta$, there is some $t>0$ such that $(\varphi_t (v), a) \in P(M)$. 

A compactness argument implies that it is sufficient to prove Lemma \ref{thm:optimal_contraction} for all $x \in \partial_\infty \Gamma$ and $a\in M - \{\xi_\rho^1(x)\}$ such that $d_{\Pb}\left(\xi_\rho^1(x), a\right) \leq \delta$. By the assumptions on $K$, there exists some $v \in K$ such that $v^+ = x$. Further, by our choice of $\delta$, there exists $t > 0$ such that $(\varphi_t(v),a) \in P(M)$.

For any non-zero $X \in a$ and for $i=1,2,3$, let
\begin{align*}
X_i := \frac{P_{i,v}(X)}{\norm{P_{i,v}(X)}_v}\in S_i(v).
\end{align*} 
By \eqref{eqn:C2}, \eqref{eqn:C3}, and \eqref{eqn:obs1},
\begin{align}
d_{\Pb}\left( a, \xi_\rho^m(x) \right) 
&\leq  A \frac{\norm{P_{3,v}(X)}_{2}}{\norm{P_{1,v}(X)}_{2}} \nonumber \\
&\leq AC^3 \frac{\norm{P_{3,v}(X)}_{v}}{\norm{P_{1,v}(X)}_{v}}  \frac{\norm{P_{1,v}(X)}_{\varphi_t(v)}}{\norm{P_{3,v}(X)}_{\varphi_t(v)}} \label{eqn:1AC3}\\
&= AC^3 \frac{\norm{X_1}_{\varphi_t(v)}}{\norm{X_3}_{\varphi_t(v)}}\nonumber
\end{align}
Using a similar argument, but with \eqref{eqn:obs2} in place of \eqref{eqn:obs1}, proves
\begin{align}
 d_{\Pb}\left( a, \xi_\rho^1(x)\right)  \geq \frac{1}{AC^3} \frac{\norm{X_1}_{\varphi_t(v)}}{\norm{X_2}_{\varphi_t(v)}}.\label{eqn:1AC3-}
\end{align}

Finally, since $0<\alpha<\Fs_m(\rho)$, Observation \ref{obs:B0} and \eqref{eqn:1AC3} give
\begin{align*}
d_{\Pb}\left( a, \xi_\rho^m(x) \right) \leq A BC^3 \left( \frac{\norm{X_1}_{\varphi_t(v)}}{\norm{X_2}_{\varphi_t(v)}}\right)^{\alpha}.
\end{align*}
Combining this with \eqref{eqn:1AC3-} yields
\begin{align*}
d_{\Pb}\left( a, \xi_\rho^m(x) \right)  \leq D d_{\Pb}\left( a, \xi_\rho^1(x)\right)^{\alpha}
\end{align*}
where $D := A^{1+\alpha}B C^{3+3\alpha}$. 
\end{proof}

\subsection{The proof of Theorem \ref{thm:main_body}}\label{sec:proof_of_main_thm}
Fix $\Vc\in\Gr_{d-m}(\Rb^d)$, and $\Uc\in\Gr_m(\Rb^d)$ such that $\Uc+\Vc=\Rb^d$. Note that the projection $\Rb^d\to\Uc$ with kernel $\Vc$ is nowhere zero on $\Rb^d-\Vc$, and therefore projectivizes to a map
\[\Pi_{\Uc,\Vc}:\Pb(\Rb^d)-\Pb(\Vc)\to \Pb(\Uc).\] 
Then for any $\Hc\in\Gr_{d-1}(\Rb^d)$ such that $\Vc\subset\Hc$, $\Pi_{\Uc,\Vc}$ restricted to the affine chart $\Ab_{\Hc}:=\Pb(\Rb^d)-\Pb(\Hc)$ is an affine projection onto $\Ab_{\Hc}\cap\Pb(\Uc)$. Further, observe that the fibers of $\Pi_{\Uc,\Vc}|_{\Ab_{\Hc}}$ do not depend on $\Uc$, i.e. if $\Uc'\in\Gr_m(\Rb^d)$ is transverse to $\Vc$, then the fibers of $\Pi_{\Uc,\Vc}|_{\Ab_{\Hc}}$ and the fibers of $\Pi_{\Uc',\Vc}|_{\Ab_{\Hc}}$  are both the set of translates of $\mathbb{A}_{\mathcal H}\cap\Pb(\mathcal V)$ in $\mathbb{A}_{\mathcal H}$.

Now, fix $y\in\partial_\infty\Gamma$. We will specialize the observation in the previous paragraph to the case where $\Hc=\xi_\rho^{d-1}(y)$ and $\Vc=\xi_\rho^{d-m}(y)$. This yields the following statement, which we record as an observation.

Consider the affine chart 
\[
\Ab_y:=\Pb(\Rb^d) - \xi_\rho^{d-1}(y).
\]

\begin{observation}\label{obs:parallel}  If $x\in\partial_\infty\Gamma-\{y\}$, then
\begin{align*}
\Pi_{x,y}:=\Pi_{\xi_\rho^m(x),\xi_\rho^{d-m}(y)} : \Ab_y\rightarrow \xi_\rho^m(x)\cap\Ab_y
\end{align*}
is a projection whose fibers do not depend on $x$.
\end{observation}

Since $M$ is $\rho$-controlled, $M-\{\xi_\rho^1(y)\}\subset\Ab_y$, so we may define 
\[F_{x,y}:=\Pi_{x,y}|_{M-\{\xi_\rho^1(y)\}}.\]

\begin{lemma}\label{lem:homeo} If $x\in\partial_\infty\Gamma-\{y\}$, then the map
\begin{align*}
F_{x,y}:M-\{\xi_\rho^1(y)\} \rightarrow \xi_\rho^m(x) \cap \Ab_y.
\end{align*}
is a homeomorphism onto an open set $I(x,y)\subset \xi_\rho^m(x) \cap \Ab_y$.
\end{lemma}

\begin{proof}

If $p_1, p_2 \in M - \{\xi_\rho^1(y)\}$ and $F_{x,y}(p_1) = F_{x,y}(p_2)$, then 
\begin{align*}
p_1 + \xi_\rho^{d-m}(y) = p_2 + \xi_\rho^{d-m}(y),
\end{align*}
so $p_1 + p_2 + \xi_\rho^{d-m}(y)$ is not direct. The assumption that $M$ is $m$-hypercovex implies that $p_1=p_2$. Thus, $F_{x,y}$ is injective. 

Since $F_{x,y}$ is continuous, we can now apply the invariance of domain theorem to deduce that $F_{x,y}$ is a homeomorphism onto an open set $I(x,y)$ in $\xi_\rho^m(x)\cap \Ab_y$ (recall that $M$ is a topological $(m-1)$-manifold). 
\end{proof}

Fix $x \in \partial_\infty \Gamma-\{y\}$ and choose an affine identification $\Ab_y\cong\Rb^{d-1}$ such that
\begin{itemize}
\item $\xi_\rho^1(x) = 0$,
\item $\xi_\rho^m(x)\cap\Ab_y = \Rb^{m-1} \times \{0\}$,
\item $\left(\xi_\rho^1(x)+ \xi_\rho^{d-m}(y) \right)\cap\Ab_y  = \{0\} \times \Rb^{d-m}$.
\end{itemize}
Fix a Riemannian metric on $\Pb(\Rb^d)$ that agrees with the standard Euclidean metric on $\Ab_y\cong \Rb^{d-1}$ when restricted to a neighborhood $U$ of $\xi_\rho^1(x)$. Let $d_{\Pb}$ be the induced distance function of this Riemannian metric.

For any $z\in\partial_\infty\Gamma-\{y\}$, let
\begin{align*}
A_z : \Rb^{m-1} \times \{0\} \rightarrow \{0\} \times \Rb^{d-m}
\end{align*}
be the map whose graph is $H_z:=\xi_\rho^m(z)\cap\Ab_y $, i.e.
\begin{align*}
H_z = \left\{ u+A_z(u) : u \in \Rb^{m-1} \times \{0\} \right\},
\end{align*}
see Figure \ref{fig:graph}. Note that $A_z$ is an affine map, and Observation \ref{obs:parallel} implies that 
$$
I(z,y)=\{u + A_z(u) : u \in I(x,y) \}.
$$

\begin{figure}[ht]
\centering
\includegraphics[scale=0.7]{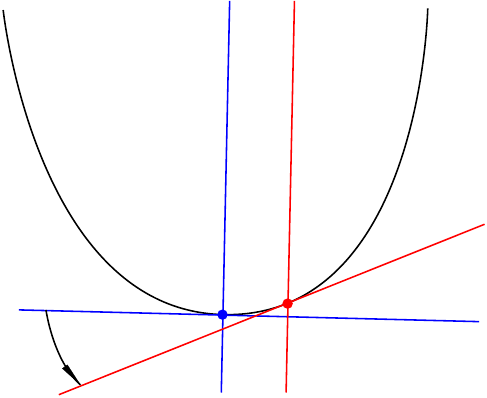}
\small
\put (-20, 29){\tiny\color{blue}$\xi_\rho^m(x)\cap\Ab_y$}
\put (-20, 57){\tiny\color{red}$H_z$}
\put (-160, 80){\tiny$M$}
\put (-70, 137){\tiny\color{red}$(\xi_\rho^1(z)+ \xi_\rho^{d-m}(y))\cap\Ab_y$}
\put (-165, 137){\tiny\color{blue}$(\xi_\rho^1(x)+ \xi_\rho^{d-m}(y))\cap\Ab_y$}
\put (-155, 15){\tiny$A_z$}
\caption{$M$ in the affine chart $\Ab_y$.}\label{fig:graph}
\end{figure}

Let 
\[L_z: \Rb^{m-1}\times\{0\} \rightarrow \{0\}\times\Rb^{d-m}\] 
denote the linear part of $A_z$ (in our choice of affine coordinates). Note also that the maps $z \mapsto A_z$ and $z \mapsto L_z$ are continuous.  Furthermore, by Lemma~\ref{lem:homeo}, there exists a map 
\begin{align*}
f_z : I(z,y) \rightarrow \{0\} \times \Rb^{d-m}
\end{align*}
whose graph is $M-\xi_\rho^1(y )$, i.e.
\begin{align*}
M-\xi_\rho^1(y )= \{ u + f_z(u) : u \in I(z,y)\}.
\end{align*}
Notice that for any $z \in \partial_\infty \Gamma$ and $u \in I(x,y)\subset\Rb^{m-1} \times \{0\}$, we have
\begin{align*}
u + f_{x}(u) = \Big( u+ A_z(u) \Big) + f_z\Big( u+ A_z(u) \Big).
\end{align*}

Fix $\alpha$ satisfying $1\leq\alpha <\Fs_m(\rho)$ and fix $x \in \partial_\infty \Gamma -\{y\}$. Then fix a neighborhood $U$ of $\xi^1_\rho(x)$ in $I(x,y)$ and $\delta > 0$ such that: if $u \in U$, $h \in \Rb^{m-1} \times \{0\}$, and $\norm{h}_2 < \delta$, then $u+h \in I(x,y)$. By possibly decreasing $\delta$ and using Lemma~\ref{thm:optimal_contraction}, there exists $C > 0$ such that for all $z \in F_{x,y}^{-1}(U) \cap \xi^1(\partial_\infty \Gamma)$ and $h\in\Rb^{m-1} \times \{0\}$ with $\norm{h}_2 < \delta$, we have
\begin{align}\label{eqn: littleo1}
\norm{f_z\Big(\xi_\rho^1(z)+h+L_z(h) \Big)} \leq C\norm{h}_2^{\alpha}.
\end{align}

Fix $u \in U$ with $u+f_x(u) \in \xi^1_\rho(\partial_\infty \Gamma)$. Then let $z \in \partial_\infty \Gamma$ be the unique point with $u+f_x(u) = \xi_\rho^1(z)$. Then by definition, $f_x(u) = A_z(u)$. Now, if $h \in \Rb^{m-1} \times \{0\}$ and $\norm{h}_2 < \delta$, then
\begin{align*}
f_{x}(u+h) 
&= A_z(u+h) + f_z\Big( u+h+ A_z(u+h) \Big) \\
& = A_z(u)+L_z(h) + f_z\Big( u+h+ A_z(u)+L_z(h) \Big) \\
&= f_{x}(u) + L_z(h) + f_z\Big(\xi_\rho^1(z)+h+L_z(h) \Big).
\end{align*}
So Equation~\eqref{eqn: littleo1} implies that $f_x$ is $C^1$ at $u$, $d(f_x)_u = L_z$, and 
$$
\norm{ f_x(u+h) - f_x(u) - d(f_x)_u(h)}_2 \leq C\norm{h}_2^{\alpha}. 
$$
for all $h \in \Rb^{m-1} \times \{0\}$ with $\norm{h}_2 < \delta$.

Since $y \in \partial_\infty \Gamma$ and $x \in \partial_\infty \Gamma$ were arbitrary, this implies that $M$ is $C^\alpha$ along $\xi_\rho^1(\partial_\infty \Gamma)$. Since  $1\leq\alpha <\Fs_m(\rho)$ was arbitrary, this completes the proof of Theorem~\ref{thm:main_body}.

\section{Eigenvalue description of $\Fs_m(\rho)$}\label{sec:optimal}
Recall that in the introduction, we defined 
\begin{align*}
\Es_m(\rho) := \inf_{\gamma\in\Gamma}\left\{\log\frac{\lambda_1(\rho(\gamma))}{\lambda_{m+1}(\rho(\gamma))}\Bigg/\log\frac{\lambda_1(\rho(\gamma))}{\lambda_{m}(\rho(\gamma))}: \frac{\lambda_1(\rho(\gamma))}{\lambda_{m}(\rho(\gamma))} \neq 1 \right\}.
\end{align*}
The main result of this section is the following theorem.

\begin{theorem}\label{thm:alphas} If $\rho:\Gamma\to\PGL(d,\Rb)$ is $P_{1,m}$-Anosov and the Zariski closure of $\rho(\Gamma)$ in $\PGL(d,\Rb)$ is semisimple, then
\begin{align*}
 \Fs_m(\rho)=\Es_m(\rho),
\end{align*}
where $\Fs_m(\rho)$ is the quantity defined by \eqref{eqn:alpham1}.
\end{theorem}

\begin{remark} Recall, from Proposition~\ref{prop:Zclosure}, that if $\rho$ is irreducible, then the Zariski closure of $\rho(\Gamma)$ in $\PGL(d,\Rb)$ is semisimple.
\end{remark} 


Theorem \ref{thm:alphas} has the following consequences.

\begin{corollary}\label{cor: d-1 bound}
If $\rho:\Gamma\to\PGL(d,\Rb)$ is $P_{1,m}$-Anosov and the Zariski closure of $\rho(\Gamma)$ in $\PGL(d,\Rb)$ is semisimple, then $\Fs_{d-1}(\rho)=\Es_{d-1}(\rho)\le 2$.
\end{corollary}

\begin{proof}
By Theorem \ref{thm:alphas}, we need to show that $\Es_{d-1}(\rho)\le 2$. For any $\gamma\in\Gamma$, set 
\[k(\gamma):=\log\frac{\lambda_1(\rho(\gamma))}{\lambda_d(\rho(\gamma))}\Bigg/\log\frac{\lambda_1(\rho(\gamma))}{\lambda_{d-1}(\rho(\gamma))}\]
and note that
\[k(\gamma^{-1})=\log\frac{\lambda_1(\rho(\gamma))}{\lambda_d(\rho(\gamma))}\Bigg/\log\frac{\lambda_2(\rho(\gamma))}{\lambda_d(\rho(\gamma))}.\]
Thus, 
\[\frac{1}{k(\gamma)}+\frac{1}{k(\gamma^{-1})}=\left(\log\frac{\lambda_1(\rho(\gamma))}{\lambda_{d-1}(\rho(\gamma))}+\log\frac{\lambda_2(\rho(\gamma))}{\lambda_d(\rho(\gamma))}\right)\Bigg/\log\frac{\lambda_1(\rho(\gamma))}{\lambda_d(\rho(\gamma))}\ge 1.\]
In particular, either $k(\gamma)\le 2$ or $k(\gamma^{-1})\le 2$, which implies that $\Es_{d-1}(\rho)\le 2$. 
\end{proof}

\begin{corollary}\label{cor:main_body} Let $m=2,\dots,d-1$. Suppose $\Gamma$ is a hyperbolic group such that $\partial_\infty\Gamma$ is a $(m-1)$-dimensional topological manifold, $\rho:\Gamma\to\PGL_d(\Rb)$ is a $P_{1,m}$-Anosov representation, and the Zariski closure of $\rho(\Gamma)$ in $\PGL(d,\Rb)$ is semisimple. If $\xi_\rho^1(\partial_\infty\Gamma)$ is $m$-hyperconvex, then $\Es_m(\rho)\le 2$. 
\end{corollary}

\begin{proof}
This follows immediately from Corollary \ref{cor:main_body1} and Theorem \ref{thm:alphas}.
\end{proof}

The remainder of this section is devoted to the proof of Theorem \ref{thm:alphas}. First, we make several reductions. Since $\Es_m(\rho)$ and $\Fs_m(\rho)$ are invariant under passing to a finite index subgroup (see Remark~\ref{rem:stablility}), we can assume that $\rho$ admits a lift $\overline{\rho}:\Gamma\to\SL_d(\Rb)$ (see Remark \ref{rem:lift}). Then, by passing to another finite index subgroup, we may also assume that the Zariski closure of $\rho(\Gamma)$ is connected (and semisimple by assumption).


We will prove that $\Fs_m(\rho)\le\Es_m(\rho)$ and $\Fs_m(\rho)\ge\Es_m(\rho)$ separately.

\subsection{The proof of $\Fs_m(\rho)\le\Es_m(\rho)$}

Let $E:=\Usf(\Gamma)\times\Rb^d$, and for $i=1,2,3$, let $E_i$ be the $\Gamma$-invariant sub-bundle of $E$ defined by \eqref{eqn:E}. Also, choose a $\rho$-equivariant family of inner products $\langle\cdot,\cdot\rangle_{v\in\Usf(\Gamma)}$ on $\Rb^d$ such that for all $v\in\Usf(\Gamma)$, $(E_1)_v\oplus (E_2)_v\oplus (E_3)_v$ is an orthogonal splitting of $\Rb^d$, and let $\norm{\cdot}_v=\sqrt{\langle\cdot,\cdot\rangle_v}$.

For any $(v,t)\in \Usf(\Gamma)\times\Rb$, let $\sigma_i(v,t)$ denote the $i$-th singular value of 
\[\id=\id_{v,t}:(\Rb^d,\norm{\cdot}_v)\to(\Rb^d,\norm{\cdot}_{\varphi_t(v)}).\] 
Then define the function $h:\Usf(\Gamma)\times\Rb\to\Rb$ by 
\[h(v,t):=\log\frac{\sigma_{d-m}(v,t)}{\sigma_d(v,t)}\Bigg/\log\frac{\sigma_{d-m+1}(v,t)}{\sigma_{d}(v,t)}.\]
The functions $h$ and $f$ (recall that $f$ is defined by \eqref{eqn:f}) are related by the following lemma.

\begin{lemma} \label{lem:f and g} For all $v\in \Usf(\Gamma)$ and for sufficiently large $t$, we have
\[f(v,t)=h(v,t).\]
In particular, $\displaystyle\Fs_m(\rho)=\liminf_{t\to+\infty}\inf_{v\in \Usf(\Gamma)}h(v,t)$.
\end{lemma}

\begin{proof}
Recall that $S_i(v):=\{X\in E_i(v):\norm{X}_v=1\}$ for $i=1,2,3$. Since $E=E_1\oplus E_2\oplus E_3$ is an orthogonal splitting, Theorem \ref{prop:dom_split} implies that for all $v\in \Usf(\Gamma)$ and for sufficiently large $t$, 
\begin{itemize}
\item $\sigma_d(v,t)=\norm{X}_{\varphi_t(v)}$ for all $X\in S_1(v)$,
\item $\sigma_{d-m+1}(v,t)=\sup_{X\in S_2(v)} \norm{X}_{\varphi_t(v)}$,
\item $\sigma_{d-m}(v,t)=\inf_{X\in S_3(v)} \norm{X}_{\varphi_t(v)}$.
\end{itemize}
Thus,
\begin{eqnarray*}
h(v,t)&=&\log\frac{\sigma_{d-m}(v,t)}{\sigma_d(v,t)}\Bigg/\log\frac{\sigma_{d-m+1}(v,t)}{\sigma_{d}(v,t)}\\
&=&\log\frac{\inf_{X\in S_3(v)} \norm{X}_{\varphi_t(v)}}{\sup_{X\in S_1(v)}\norm{X}_{\varphi_t(v)}}\Bigg/\log\frac{\sup_{X\in S_2(v)} \norm{X}_{\varphi_t(v)}}{\inf_{X\in S_1(v)}\norm{X}_{\varphi_t(v)}}\\
&=&\inf_{X_i\in S_i(v)}\left\{\log\frac{\norm{X_3}_{\varphi_t(v)}}{\norm{X_1}_{\varphi_t(v)}}\Bigg/\log\frac{\norm{X_2}_{\varphi_t(v)}}{\norm{X_1}_{\varphi_t(v)}}\right\}\\
&=&f(v,t).
\end{eqnarray*}
Notice that in the third equality, we used the fact that $\dim E_1(v)=1$.
\end{proof}

The following observation is a straightforward consequence of the min-max definition of singular values (see Definition \ref{def: singular}).

\begin{observation}\label{prop:easy comp}
Suppose  for $i=1,\dots,4$, $\norm{\cdot}_{(i)}$ are norms on $\Rb^d$ induced by inner products such that for all $X\in\Rb^d$, $\frac{1}{A}\leq\frac{\norm{X}_{(1)}}{\norm{X}_{(2)}}\leq A$ and $\frac{1}{A'}\leq\frac{\norm{X}_{(3)}}{\norm{X}_{(4)}}\leq A'$ for some $A,A'>1$. Let $L:\left(\Rb^d,\norm{\cdot}_{(1)}\right)\to\left(\Rb^d,\norm{\cdot}_{(3)}\right)$ and $L':\left(\Rb^d,\norm{\cdot}_{(2)}\right)\to\left(\Rb^d,\norm{\cdot}_{(4)}\right)$ denote the identity maps. Then
\[\frac{1}{AA'}\leq\frac{\sigma_i(L)}{\sigma_i(L')}\leq AA'.\]
\end{observation}

The next lemma relates the function $h$ to the eigenvalues of $\rho(\gamma)$.

\begin{lemma}\label{lem:g and eig}
Let $\gamma\in\Gamma$ be an infinite order element, and let $v\in \Usf(\Gamma)$ such that $v^\pm=\gamma^\pm$. Then
\begin{align}\label{eqn: evalue description}
\lim_{t\to+\infty}h(v,t)=\log\frac{\lambda_1(\rho(\gamma))}{\lambda_{m+1}(\rho(\gamma))}\Bigg/\log\frac{\lambda_1(\rho(\gamma))}{\lambda_{m}(\rho(\gamma))}.
\end{align}
\end{lemma}

\begin{proof}
Let $T:=T_\gamma$ denote the period of $\gamma$ (see Section \ref{sec:flowspace}). Then 
\[\norm{X}_{\varphi_{kT}(v)}=\norm{X}_{\gamma^k\cdot v}=\norm{\overline{\rho}(\gamma^{-k})\cdot X}_v\]
for all $k \in \Zb^+$ and $X \in \Rb^d$. Hence, for all $i=1,\dots,d$, the $i$-th singular value of
\[\id:(\Rb^d,\norm{\cdot}_v)\to(\Rb^d,\norm{\cdot}_{\varphi_{kT}(v)})\]
agrees with $\mu_i(\rho(\gamma^{-k}))$.

It is well-known that 
\begin{align}\label{eqn:Ben}
\lim_{k\to+\infty}\frac{1}{k}\log\mu_i(g^k)=\log\lambda_i(g).
\end{align}
for any $g\in\PGL_d(\Rb)$. Thus, we can deduce that
\begin{equation}\label{eqn:Benoist}
\lim_{k\to+\infty}\sigma_i(v,kT)^{\frac{1}{k}}=\lim_{k\to+\infty}\sigma_i(\rho(\gamma^{-k}))^{\frac{1}{k}}=|\lambda_i(\rho(\gamma^{-1}))|=\frac{1}{|\lambda_{d+1-i}(\rho(\gamma))|},\end{equation}
which implies that
\begin{eqnarray}\label{eqn:period contraction}
\lim_{k\to+\infty}h(v,kT)&=&\lim_{k\to+\infty}\left(\log\frac{\sigma_{d-m}(v,kT)}{\sigma_d(v,kT)}\Bigg/\log\frac{\sigma_{d-m+1}(v,kT)}{\sigma_{d}(v,kT)}\right)\\
&=&\log\frac{\lambda_1(\rho(\gamma))}{\lambda_{m+1}(\rho(\gamma))}\Bigg/\log\frac{\lambda_1(\rho(\gamma))}{\lambda_{m}(\rho(\gamma))}.\nonumber
\end{eqnarray}

For any $t>0$, let $k\in\Zb^+$ such that $t\in[kT,(k+1)T)$. Then Observation \ref{lem:weak flow} implies that there are constants $C\geq 1$ and $\beta\geq 0$ such that 
\[\frac{1}{C}e^{-\beta T}\leq\frac{\norm{X}_{\varphi_tv}}{\norm{X}_{\varphi_{kT}v}}\leq Ce^{\beta T}\]
for all $t \in \Rb$ and $X\in\Rb^d$. This, together with Observation \ref{prop:easy comp}, implies that 
\[\frac{1}{C}e^{-\beta T}\leq\frac{\sigma_i(v,kT)}{\sigma_i(v,t)}\leq Ce^{\beta T}\] 
for all $i=1,\dots,d$. Also, since $\rho$ is $P_{1,m}$-Anosov, we know that
\[\lim_{k\to+\infty}\log\frac{\sigma_{d-m}(v,kT)}{\sigma_d(v,kT)}=\infty=\lim_{k\to+\infty}\log\frac{\sigma_{d-m+1}(v,kT)}{\sigma_{d}(v,kT)}.\]
Hence,
\begin{eqnarray*}
\limsup_{t\to+\infty}h(v,t)&=&\limsup_{t\to+\infty}\log\frac{\sigma_{d-m}(v,t)}{\sigma_d(v,t)}\Bigg/\log\frac{\sigma_{d-m+1}(v,t)}{\sigma_{d}(v,t)}\\
&\leq&\limsup_{k\to+\infty}\frac{\displaystyle2\log C+2\beta T+\log\frac{\sigma_{d-m}(v,kT)}{\sigma_d(v,kT)}}{\displaystyle-2\log C-2\beta T+\log\frac{\sigma_{d-m+1}(v,kT)}{\sigma_{d}(v,kT)}}\\
&=&\lim_{k\to+\infty}h(v,kT).
\end{eqnarray*}
By a similar argument, $\displaystyle\liminf_{t\to+\infty}h(v,t)\geq\lim_{k\to+\infty}h(v,kT)$. Thus, 
\[\lim_{t\to+\infty}h(v,t)=\lim_{k\to+\infty}h(v,kT).\] 
This, together with \eqref{eqn:period contraction} implies \eqref{eqn: evalue description}. 
\end{proof}

It follows from Lemma \ref{lem:f and g} and Lemma \ref{lem:g and eig}  that 
\[\Fs_m(\rho)\le\inf_{v\in\Usf(\Gamma)}\lim_{t\to+\infty}h(v,t)\le\Es_m(\rho).\qedhere\]

\begin{remark}
Note that this part of the proof does not use the assumption that the Zariski closure of $\rho(\Gamma)$ in $\PGL_d(\Rb)$ is semisimple.
\end{remark}

\subsection{The proof of $\Fs_m(\rho)\ge\Es_m(\rho)$}\label{sec:cones}

Recall that $\lambda,\mu:\PGL_d(\Rb)\to\Rb^d$ respectively denote the Jordan and Cartan projections defined in Section \ref{sec:properties}. For any subgroup $G \leq \PGL_d(\Rb)$, let $\Cc_\lambda(G) \subset \Rb^d$ denote the smallest closed cone containing $\lambda(G)$, that is
\begin{align*}
\Cc_\lambda(G) := \overline{\bigcup_{g \in G} \Rb_{>0} \cdot \lambda(g)}.
\end{align*}
Also, let $\Cc_\mu(G)$ denote the \emph{asymptotic cone of $\mu(G)$}, that is
\begin{align*}
\Cc_\mu(G) := \{ x \in \Rb^d : \exists g_n \in G, \exists t_n \searrow 0, \text{ with } \lim_{n \rightarrow +\infty} t_n \mu(g_n) =x\}.
\end{align*}

We will use the following result of  Benoist~\cite{B1997}.

\begin{theorem}\label{thm:cones}\cite{B1997}.
If $G \leq \PGL_d(\Rb)$ is a connected semisimple real algebraic subgroup and $\Gamma \leq G$ is a Zariski dense subgroup, then 
\begin{align*}
\Cc_\mu(\Gamma) = \Cc_\lambda(\Gamma).
\end{align*}
\end{theorem}

\begin{remark} \ 
\begin{enumerate}
\item To be precise, the main result in~\cite{B1997} is stated slightly differently and in Appendix~\ref{sec:Benoists appendix} we will explain how the above result follows from Benoist's theorem. 
\item Notice that for any subgroup $\Gamma\leq \PGL_d(\Rb)$, the fact that $\Cc_\lambda(\Gamma) \subset \Cc_\mu(\Gamma)$ is a consequence of \eqref{eqn:Ben}.
\end{enumerate}
\end{remark}

Using Theorem \ref{thm:cones}, we prove the following lemma.

\begin{lemma} \label{lem:preliminary-inequality}
For any $\epsilon > 0$ there exists $R > 0$ such that 
\begin{align*}
\Es_m(\rho)-\epsilon < \log\frac{\mu_1(\rho(\gamma))}{\mu_{m+1}(\rho(\gamma))}\Bigg/\log\frac{\mu_1(\rho(\gamma))}{\mu_{m}(\rho(\gamma))}
\end{align*}
for all $\gamma \in \Gamma$ with $\norm{\mu(\rho(\gamma))}_2\geq R$. 
\end{lemma}

\begin{proof} 
By Theorem~\ref{thm:cones}, $\Cc_\mu(\rho(\Gamma)) = \Cc_\lambda(\rho(\Gamma))$. Fix $\epsilon > 0$ and suppose for the purpose of contradiction that there exists a sequence $\{\gamma_n\}_{n \geq 1} \subset \Gamma$ such that for all $n$, $\norm{\mu(\rho(\gamma_n))}_2\geq n$ and
\begin{align*}
\Es_m(\rho)-\epsilon \geq \log\frac{\mu_1(\rho(\gamma_n))}{\mu_{m+1}(\rho(\gamma_n))}\Bigg/\log\frac{\mu_1(\rho(\gamma_n))}{\mu_{m}(\rho(\gamma_n))}.
\end{align*}
By passing to a subsequence we can suppose that 
\begin{align*}
\frac{1}{\norm{\mu(\rho(\gamma_n))}_2} \mu(\rho(\gamma_n)) \rightarrow x =(x_1,\dots,x_d)\in \Cc_\mu(\rho(\Gamma)) = \Cc_\lambda(\rho(\Gamma)).
\end{align*}
It follows that $\Es_m(\rho)-\epsilon \geq\frac{x_1-x_{m+1}}{x_1-x_m}$. On the other hand, the definition of $\Es_m(\rho)$ and $\Cc_\lambda(\rho(\Gamma))$ implies that
\[\Es_m(\rho) \leq \frac{x_1-x_{m+1}}{x_1-x_m},\] 
which is a contradiction. 
 \end{proof}

We now turn to the proof of $\Fs_m(\rho)\ge\Es_m(\rho)$. Fix a compact set $K \subset \Usf(\Gamma)$ such that $\Gamma \cdot K = \Usf(\Gamma)$. Since $h(v,t)=h(\gamma\cdot v,t)$ for all $\gamma\in\Gamma$ and all $v\in\Usf(\Gamma)$, by Lemma \ref{lem:f and g}, it is enough to show that 
\begin{align*}
\Es_m(\rho) \leq \liminf_{t\to+\infty}\inf_{v\in K}h(v,t).
\end{align*}

Fix $C > 1$ such that $\frac{1}{C} \norm{X}_{2} \leq \norm{X}_{v} \leq C \norm{X}_{2}$ for all $v\in K$ and $X \in \Rb^d$. By Lemma \ref{lem:preliminary-inequality},  for every $\epsilon>0$, there exists a positive number $R' > 0$ such that
\begin{align*}
\Es_m(\rho)-\epsilon < \log\frac{\mu_1(\rho(\gamma))}{\mu_{m+1}(\rho(\gamma))}\Bigg/\log\frac{\mu_1(\rho(\gamma))}{\mu_{m}(\rho(\gamma))}
\end{align*}
for all $\gamma \in \Gamma$ with $\norm{\mu(\rho(\gamma))}_2\geq R'$. Since $\rho$ is $P_1$-Anosov and
\begin{align*}
\log\frac{\mu_1(\rho(\gamma))}{\mu_{2}(\rho(\gamma))} \leq \log\frac{\mu_1(\rho(\gamma))}{\mu_{k}(\rho(\gamma))}
\end{align*}
for $k > 1$, Theorem~\ref{thm:SV_char_of_Anosov} and Corollary~\ref{thm:QI_Anosov} together imply that 
\[\log\frac{\mu_1(\rho(\gamma))}{\mu_{m+1}(\rho(\gamma))}\,,\,\log\frac{\mu_1(\rho(\gamma))}{\mu_{m}(\rho(\gamma))}\geq \frac{1}{A''}\norm{\mu(\rho(\gamma))}_2-B''\]
for some $A''\geq 1$ and $B''\geq 0$. Hence, there exists $R \geq R'$ such that 
 \begin{align}\label{eqn: 2eps}
\Es_m(\rho)-2\epsilon < \left(\log\frac{\mu_1(\rho(\gamma))}{\mu_{m+1}(\rho(\gamma))}-4\log C\right)\bigg/\left(\log\frac{\mu_1(\rho(\gamma))}{\mu_{m}(\rho(\gamma))}+4\log C\right)
\end{align} 
for all $\gamma \in \Gamma$ with $\norm{\mu(\rho(\gamma))}_2\geq R$.

Let $d:=d_{\Usf(\Gamma)}$ denote the $\Gamma$-invariant metric on $\Usf(\Gamma)$ specified in Section \ref{sec:flowspace}, and let $D$ be the diameter of $K$. By Corollary~\ref{thm:QI_Anosov} and the fact that $\Gamma$ acts on $\Usf(\Gamma)$ cocompactly and by isometries, there exists $A \geq 1$ and $B \geq 0$ such that 
\begin{align*}
\frac{1}{A} \norm{\mu(\rho(\gamma))}_2 - B \leq d(v, \gamma \cdot v) \leq A \norm{\mu(\rho(\gamma))}_2 + B
\end{align*}
for all $v \in \Usf(\Gamma)$ and all $\gamma\in\Gamma$. Also, since every $\varphi_t$-orbit in $\Usf(\Gamma)$ is a quasi-isometric embedding, there exist $A'\geq1$ and $B'\geq0$ such that ,
\begin{align*}
\frac{1}{A'} |t|- B' \leq d(v,\varphi_t(v))\leq A' |t| + B'
\end{align*}
for all $t\in\Rb$ and $v \in \Usf(\Gamma)$.

Let $t > A'(B'+D + AR + B)$, $v \in K$, and $\gamma \in \Gamma$ such that $\gamma^{-1}\cdot\varphi_t (v) \in K$. Then 
\[d(\gamma\cdot v,v)\geq d(v,\varphi_t(v))-d( \varphi_t (v), \gamma\cdot v) \geq \frac{1}{A'}t-B'-D,\] 
which implies that
\begin{align}\label{eqn: larger than R}
\norm{\mu(\rho(\gamma))}_2 \geq \frac{1}{A} \left(d(\gamma\cdot v, v) - B\right) \geq \frac{\frac{1}{A'}t-B'-D-B}{A}  \geq R.
\end{align}

By the definition of $C$, we see that for any $X\in\Rb^d$,
\[\frac{1}{C}\leq\frac{\norm{X}_v}{\norm{X}_{2}}\leq C\quad \text{and}\quad \frac{1}{C}\leq\frac{\norm{\overline{\rho}(\gamma)^{-1}X}_{\gamma^{-1}\cdot \varphi_t(v)}}{\norm{\overline{\rho}(\gamma)^{-1}X}_{2}} \leq C.\]
Note that $X\mapsto \norm{\overline{\rho}(\gamma)^{-1}X}_{\gamma^{-1}\cdot \varphi_t (v)}$ and  $X\mapsto\norm{\overline{\rho}(\gamma)^{-1}X}_{2}$ are both norms on $\Rb^d$, $\sigma_i(v,t)$ is the $i$-th singular value of the map
\[\id:(\Rb^d,\norm{\cdot}_v)\to(\Rb^d,\norm{\cdot}_{\varphi_t (v)})=\left(\Rb^d,\norm{\overline{\rho}(\gamma)^{-1}\cdot}_{\gamma^{-1}\cdot \varphi_t (v)}\right),\]
and $\mu_i(\rho(\gamma)^{-1})$ is the $i$-th singular value of the map
\[\id:(\Rb^d,\norm{\cdot}_2)\to\left(\Rb^d,\norm{\overline{\rho}(\gamma)^{-1}\cdot}_2\right)\]
It thus follows from Observation \ref{prop:easy comp} that 
\begin{align}\label{eqn: Csquare}
\frac{1}{C^2}\frac{1}{\mu_{d+1-i}(\rho(\gamma))}=\frac{1}{C^2}\mu_{i}(\rho(\gamma)^{-1})\leq\sigma_i(v,t)\leq  C^2\frac{1}{\mu_{d+1-i}(\rho(\gamma))}.\end{align}
Hence, by \eqref{eqn: 2eps}, \eqref{eqn: larger than R}, and \eqref{eqn: Csquare},
\begin{align*}
\Es_m(\rho)-2\epsilon &<   \left(\log\frac{\mu_1(\rho(\gamma))}{\mu_{m+1}(\rho(\gamma))}-4\log C\right)\bigg/\left(\log\frac{\mu_1(\rho(\gamma))}{\mu_{m}(\rho(\gamma))}+4\log C\right)\\
& \leq \log\frac{\sigma_{d-m}(v,t)}{\sigma_d(v,t)}\Bigg/\log\frac{\sigma_{d-m+1}(v,t)}{\sigma_{d}(v,t)} = h(v,t).
\end{align*}
Since $v \in K$ and $t > A'(B'+D + AR+ B)$ were arbitrary, 
\begin{align*}
\Es_m(\rho)-2\epsilon \leq \liminf_{t\to+\infty}\inf_{v\in K}h(v,t).
\end{align*}
Then since $\epsilon > 0$ was also arbitrary,
\begin{equation*}
\Es_m(\rho) \leq \liminf_{t\to+\infty}\inf_{v\in K}h(v,t). \qedhere
\end{equation*}

 \section{Optimal regularity}\label{sec:regularity}
   
 In this section we prove Theorem~\ref{thm:regularity} and the second part of Theorem~\ref{thm:regularity2}. In light of Corollaries \ref{cor:main_body} and \ref{cor: d-1 bound} (see Examples \ref{eg:limitset} and \ref{eg:convex}), it is sufficient to prove the following theorem.
 
 \begin{theorem} \label{thm:regularity_body}Suppose  $\rho: \Gamma \rightarrow \PGL_{d}(\Rb)$ is an irreducible, $P_{1,m}$-Anosov representation for some $m=2,\dots,d-1$, and suppose  $M\subset\Pb(\Rb^d)$ is a $\rho$-controlled, $m$-hyperconvex, topological $(m-1)$-dimensional submanifold. Then 
\begin{align*}
\Es_m(\rho)\leq \sup\left\{ \alpha\in (1,+\infty) : M \text{ is } C^{\alpha} \text{ along }\xi_\rho^1(\partial_\infty \Gamma) \right\}
\end{align*} 
with equality if 
\begin{itemize}
\item[($\ast$)] $M \cap \left(a_1 + a_2 +  \xi_\rho^{d-m}(y)\right)$ spans $a_1 + a_2 +  \xi_\rho^{d-m}(y)$ for all pairwise distinct $a_1,a_2,\xi_\rho^1(y)\in M$.
\end{itemize}
\end{theorem}

As mentioned in the introduction (see (2) of Remark \ref{rem:stablility}), the condition ($\ast$) is trivial when $m=2$ and $m=d-1$. In Section \ref{sec:stability}, we show that when $M=\xi_\rho^1(\partial_\infty\Gamma)$, ($\ast$) is an open condition in the space of $P_1$-Anosov representations in $\Hom(\Gamma,\PSL_d(\Rb))$. Then, in Section \ref{sec:proof_regularity_body}, we prove Theorem \ref{thm:regularity_body}.

\subsection{Stability of hypotheses}\label{sec:stability}

To show that ($\ast$) is an open condition when $M=\xi_\rho^1(\partial_\infty\Gamma)$, we  use two well-known results. The first is a standard fact about hyperbolic groups, see \cite[Proposition 1.13]{B2011} for a proof.

\begin{proposition}\label{prop:3_point_action}  The $\Gamma$-action on $\partial_\infty \Gamma^{(3)}: = \{ (x,y,z)  \in \partial_\infty \Gamma^3 : x,y,z \text{ distinct} \}$ is properly discontinuous and co-compact.
\end{proposition}

The second is a result about Anosov representations due to Guichard-Wienhard. In the case when $\Gamma$ is the fundamental group of a negatively curved Riemannian manifold, this was established by Labourie~\cite[Proposition 2.1]{L2006}. 

\begin{theorem}\label{thm:gwstable}
\cite[Theorem 5.13]{GW2012}\label{thm:continuous_limit_curve} Let
\begin{align*}
\Oc_k : = \{ \rho \in \Hom(\Gamma, \PGL_d(\Rb)) : \rho \text{ is $P_k$-Anosov} \}.
\end{align*}
Then $\Oc_k$ is open, and the map 
\begin{align*}
\rho \in \Oc_k \mapsto \xi_\rho^k \in C\left( \partial_\infty \Gamma, \Gr_k(\Rb^d)\right)
\end{align*}
is continuous. 
\end{theorem}

Using these, we prove the required stability results.

\begin{corollary} Suppose $\partial_\infty \Gamma$ is a topological $(m-1)$-manifold, and $\rho_0: \Gamma \rightarrow \PGL_{d}(\Rb)$ is a $P_{1,m}$-Anosov representation. If $\xi_{\rho_0}^1(x) + \xi_{\rho_0}^1(z) +  \xi_{\rho_0}^{d-m}(y)$ is a direct sum and
\begin{align*}
\xi_{\rho_0}^1(\partial_\infty \Gamma) \cap \left(\xi_{\rho_0}^1(x) + \xi_{\rho_0}^1(z) +  \xi_{\rho_0}^{d-m}(y)\right)
\end{align*}
spans $\xi_{\rho_0}^1(x) + \xi_{\rho_0}^1(z) +  \xi_{\rho_0}^{d-m}(y)$ for all $x,y,z \in \partial_\infty \Gamma$ distinct, then any sufficiently small deformation of $\rho_0$ also has these properties. 
\end{corollary}

\begin{proof} 
Observe that the required properties are invariant under the $\Gamma$-action on $\partial_\infty\Gamma^{(3)}$, which is cocompact by Proposition~\ref{prop:3_point_action}. It thus suffices to fix $(x_0,y_0,z_0) \in \partial_\infty \Gamma^{(3)}$ and prove that there exist a neighborhood $U\subset\partial_\infty \Gamma^{(3)}$ of $(x_0,y_0,z_0)$ and a neighborhood $\Oc\subset\Oc_1\cap\Oc_m$ of $\rho_0$ such that $\xi_\rho^1(x) + \xi_\rho^1(z) +  \xi_\rho^{d-m}(y)$ is a direct sum and
\begin{align*}
\xi_\rho^{1}(\partial_\infty \Gamma) \cap \left(\xi_\rho^{1}(x) + \xi_\rho^{1}(z) +  \xi_\rho^{d-m}(y)\right)
\end{align*}
spans $\xi_\rho^{1}(x) + \xi_\rho^{1}(z) +  \xi_\rho^{d-m}(y)$ for all $(x,y,z) \in U$ and any $\rho\in \Oc$.

It is a consequence of Theorem \ref{thm:gwstable} that the $\rho$-equivariant map 
\[\Xi:(\Oc_1\cap\Oc_m)\times\partial_\infty\Gamma^{(3)}\to\Gr_1(\Rb^d)\times\Gr_{d-m}(\Rb^d)\times\Gr_1(\Rb^d)\]
given by $\Xi:(\rho,(x,y,z))\mapsto(\xi_\rho^1(x),\xi_\rho^{d-m}(y),\xi_\rho^1(z))$ is continuous.
Let $V$ denote the open subset of triples $(F^1,G^{d-m},H^1)$ in $\Gr_1(\Rb^d)\times\Gr_{d-m}(\Rb^d)\times\Gr_1(\Rb^d)$ such that $F^1+G^{d-m}+H^1$ is a direct sum. Then $\Xi^{-1}(V)\subset (\Oc_1\cap\Oc_m)\times\partial_\infty\Gamma^{(3)}$ is an open set that contains $(\rho_0,(x_0,y_0,z_0))$, so there is some neighborhood $U\subset\partial_\infty \Gamma^{(3)}$ of $(x_0,y_0,z_0)$ and some neighborhood $\Oc\subset\Oc_1\cap\Oc_m$ of $\rho_0$ such that $\Oc\hspace{-.07cm}\times U\subset \Xi^{-1}(V)$. It is clear that $\xi_\rho^1(x) + \xi_\rho^1(z) +  \xi_\rho^{d-m}(y)$ is a direct sum for all $(x,y,z)\in U$ and all $\rho\in\Oc$.



Let $e_1,\dots, e_d$ be a basis of $\Rb^d$ such that 
\begin{align*}
\xi_{\rho_0}^1(x_0) & = \Rb \cdot e_1, \quad \xi_{\rho_0}^m(x_0) = \Span\{e_1,\dots,e_m\}, \quad \xi_{\rho_0}^{d-m}(y_0) = \Span\{ e_{m+1},\dots, e_d\}, \\
& \xi_{\rho_0}^{d-1}(y_0) = \Span\{e_2,\dots, e_d\}, \quad \text{and}\quad\xi_{\rho_0}^1(z_0) = \Rb\cdot(e_1+e_2+e_d).
\end{align*}
Then there exists a continuous map 
\[\Oc\hspace{-.07cm}\times U\to\PGL_d(\Rb)\]
given by $(\rho, (x,y,z)) \mapsto g_{\rho,(x,y,z)}$, where  $g_{\rho_0,(x_0,y_0,z_0)} =\id$, 
 \begin{align*}
g_{\rho,(x,y,z)}\cdot\xi_\rho^1(x) & = \Rb \cdot e_1, \\
g_{\rho,(x,y,z)}\cdot\xi_\rho^m(x) & = \Span\{e_1,\dots,e_m\},\\
g_{\rho,(x,y,z)}\cdot \xi_\rho^{d-m}(y) &= \Span\{ e_{m+1},\dots, e_d\},\\
g_{\rho,(x,y,z)}\cdot \xi_\rho^{d-1}(y) &= \Span\{e_2,\dots, e_d\}, \text{ and} \\ 
 g_{\rho,(x,y,z)}\cdot\xi_\rho^1(z) &= \Rb\cdot(e_1+e_2+e_d).
 \end{align*} 
 
 
 


Consider the affine chart 
\[\mathbb{A}:=\left\{[1:v:v']\in\Pb(\Rb^d):v\in\Rb^{m-1}\text{ and }v'\in\Rb^{d-m}\right\}\] 
of $\Pb(\Rb^d)$. For all $(x,y,z)\in U$ and all $\rho\in\Oc$, the sum $\xi_\rho^1(x) + \xi_\rho^1(z) +  \xi_\rho^{d-m}(y)$ is direct. Hence the projection $\pi:\Ab\to \Ab$ given by $\pi([1:v:v'])=[1:v:0]$ restricts to an injection on $g_{\rho,(x,y,z)}\cdot\xi_\rho^1( \partial_\infty \Gamma - \{y\})$. By the invariance of domain theorem, 
\[W_{\rho,(x,y,z)}:=\pi\big(g_{\rho,(x,y,z)}\cdot\xi_\rho^1( \partial_\infty \Gamma - \{y\})\big)\]
is an open subset of $\{[1:v:0]\in\Pb(\Rb^d):v\in\Rb^{m-1}\}\cong\Rb^{m-1}$. Then define $f_{\rho, (x,y,z)} : W_{\rho,(x,y,z)} \rightarrow \Rb^{d-m}$ so that 
 \begin{align*}
  g_{\rho,(x,y,z)}\cdot\xi_\rho^1( \partial_\infty \Gamma - \{y\}) = \left\{ [1:v:f_{\rho, (x,y,z)}(v)] : v \in W_{\rho,(x,y,z)} \right\},
\end{align*}
and note that
 \begin{align*}
\xi_\rho^1( \partial_\infty & \Gamma - \{y\}) \cap \left(\xi_\rho^{1}(x) + \xi_\rho^{1}(z) +  \xi_\rho^{d-m}(y)\right) \\
& =   g_{\rho,(x,y,z)}^{-1}\cdot\left\{ [1:te_2:f_{\rho, (x,y,z)}(te_2)] : t\in\Rb\text{ and } te_2 \in W_{\rho,(x,y,z)} \right\}.
\end{align*}
So by hypothesis, there exist $t_1,\dots,t_{d-m+2}\in\Rb$ such that $t_ie_2 \in W_{\rho_0,(x_0,y_0,z_0)}$ for all $i\in\{1,\dots,d-m+2\}$, and
\begin{align*}
[1:t_1e_2:f_{\rho_0, (x_0,y_0,z_0)}(t_1e_2)], \dots, [1:t_{d-m+2}e_2:f_{\rho_0, (x_0,y_0,z_0)}(t_{d-m+2}e_2)] 
\end{align*}
spans $\xi_{\rho_0}^1(x_0) + \xi_{\rho_0}^1(y_0) +  \xi_{\rho_0}^{d-m}(z_0)$. 

By Theorem~\ref{thm:continuous_limit_curve}, the open set $W_{\rho,(x,y,z)}$ varies continuously with $(\rho,(x,y,z))\in\Oc\hspace{-.07cm}\times U$. Thus, by shrinking $\Oc$ and $U$ if necessary, we may ensure that there is some open $W\subset \Rb^{m-1}$ such that $W\subset W_{\rho,(x,y,z)}$ and $t_ie_2 \in W$ for all $i\in\{1,\dots,d-m+2\}$ and $(\rho,(x,y,z))\in\Oc\hspace{-.07cm}\times U$. Then Theorem~\ref{thm:continuous_limit_curve} implies that the map 
\[\Oc\hspace{-.07cm} \times U\to C\left(W, \Rb^{d-m}\right)\] 
given by $(\rho, (x,y,z)) \mapsto f_{\rho,(x,y,z)}|_W$ is continuous. Hence, by further shrinking $\Oc$ and $U$, we may assume that 
\begin{align*}
g_{\rho,(x,y,z)}^{-1}[1:t_1e_2:f_{\rho, (x,y,z)}(t_1e_2)], \dots, g_{\rho,(x,y,z)}^{-1}[1:t_{d-m+2}e_2:f_{\rho, (x,y,z)}(t_{d-m+2}e_2)] 
\end{align*}
spans $\xi_\rho^{1}(x) + \xi_\rho^{1}(y) +  \xi_\rho^{d-m}(z)$ for all $\rho\in\Oc$ and $(x,y,z)\in U$.
\end{proof}

 \subsection{The proof of Theorem \ref{thm:regularity_body}}\label{sec:proof_regularity_body}
From Theorem~\ref{thm:main_body} and Theorem \ref{thm:alphas}, we see that 
 \begin{align*}
\Es_m(\rho) \leq \sup\left\{ \alpha\in(1,+\infty): M \text{ is } C^{\alpha} \text{ along }\xi_\rho^1(\partial_\infty \Gamma) \right\}.
\end{align*}
To prove the equality case, we first make the following observation. Let $e_1,\dots, e_d$ denote the standard basis of $\Rb^d$, and let $g \in \PGL_{d}(\Rb)$ be a proximal element such that 
\begin{itemize}
\item $e_1$ spans the eigenspace corresponding to $\lambda_1(g)$, 
\item $e_m$ lies in the generalized eigenspace corresponding to $\lambda_m(g)$,
\item $e_{m+1}$ lies in the generalized eigenspace corresponding to $\lambda_{m+1}(g)$, and
\item $e_{m+2},\dots,e_d$ lie in the sum of the generalized eigenspaces corresponding to $\lambda_{m+2}(g),\dots,\lambda_d(g)$. 
\end{itemize} 
Then observe that
\begin{align}\label{eqn:P2stretch}
\log\lambda_1(g)&=\lim_{n \rightarrow +\infty} \frac{1}{n} \log \norm{\overline{g}^n\cdot e_1}_2,\nonumber\\
\log\lambda_m (g)&= \lim_{n \rightarrow +\infty} \frac{1}{n} \log \norm{\overline{g}^n\cdot e_m}_2,\text{ and }\\
\log\lambda_{m+1}(g) &= \lim_{n \rightarrow +\infty} \frac{1}{n} \log \norm{\overline{g}^n\cdot \sum_{j=m+1}^d v_j e_j}_2\text{ when }v_{m+1} \neq 0,\nonumber
\end{align}
where $\overline{g}\in\GL_d(\Rb)$ is a linear representative of $g$ with unit determinant and $\norm{\cdot}_2$ is the standard $\ell^2$-norm on $\Rb^d$.

Now, fix some $\gamma \in \Gamma$ with infinite order and let $\gamma^{\pm} \in \partial_\infty \Gamma$ denote the attracting and repelling fixed points of $\gamma$. We can make a change of basis and assume that $\xi_\rho^1(\gamma^+) = [e_1]$, $\xi_\rho^m(\gamma^+) = \Span\{ e_1,\dots, e_m\}$, $\xi_\rho^1(\gamma^-) = [e_d]$, $\xi_\rho^{d-m}(\gamma^-) = \Span\{ e_{m+1}, \dots, e_d\}$, and $\xi_\rho^{d-1}(\gamma^-) = \Span\{e_2,\dots, e_d\}$. 
Then
 \begin{align*}
\rho(\gamma) = \begin{bmatrix} \lambda & & \\ & U & \\ & & V \end{bmatrix}
 \end{align*}
where $\lambda \in \Rb$, $U \in \GL_{m-1}(\Rb)$, and $V \in \GL_{d-m}(\Rb)$. By a further change of basis, we can assume that $e_m$ lies in the generalized eigenspace corresponding to $\lambda_m(\rho(\gamma))$, and $e_{m+1}$ lies in the generalized eigenspace corresponding to $\lambda_{m+1}(\rho(\gamma))$. 
 
By Theorem \ref{thm:main_body}, $M$ is $C^1$ along $\xi_\rho^1(\partial_\infty \Gamma)$, and the tangent space to $M$ at $\xi_\rho^1(\gamma^+)$ is $\xi_\rho^m(\gamma^+)$. Thus, for any $\epsilon > 0$ sufficiently small there exists some $a \in M$ such that 
\begin{align*}
a = \left[e_1 + \epsilon e_m + \sum_{j=m+1}^d y_j e_j \right].
\end{align*}
Then 
\begin{align*}
\xi_\rho^1(\gamma^+) + a +  \xi_\rho^{d-m}(\gamma^-) = \Span\{e_1,e_m, e_{m+1}, \dots, e_d\}
\end{align*}
and by hypothesis ($\ast$), there exists some $b \in M$ such that 
\begin{align*}
b= \left[z_1e_1 + z_m e_m + \sum_{j=m+1}^d z_j e_j \right]
\end{align*}
and $z_{m+1} \neq 0$. Note that $\xi_\rho^1(\gamma^-)\neq b\neq\xi_\rho^1(\gamma^+)$, so by transversality and $m$-hyperconvexity, the sums 
\[
b+\xi_\rho^{d-1}(\gamma^-)\quad\text{and}\quad \xi_\rho^1(\gamma^+) + b +  \xi_\rho^{d-m}(\gamma^-)
\] 
are both direct. In particular, $z_1 \neq 0\neq z_m$. 

Next fix a distance $d_{\Pb}$ on $\Pb(\Rb^d)$ induced by a Riemannian metric on $\Pb(\Rb^d)$. Recall that for any $v=(v^+,v^-,s) \in \Usf(\Gamma)$, we denote
\begin{align}
E_1(v) & := \xi_\rho^1(v^+),\nonumber \\
E_2(v) & := \xi_\rho^{d-1}(v^-) \cap \xi_\rho^m(v^+),\label{eqn:E}\\
E_3(v) & := \xi_\rho^{d-m}(v^-),\nonumber
\end{align}
and for any $i=1,2,3$, we denote by $P_{i,v} : \Rb^d \rightarrow E_i(v)$ the projection with kernel $E_{i-1}(v)+E_{i+1}(v)$ (as before, arithmetic in the subscripts is done modulo $3$).

Let $v\in\Usf(\Gamma)$ be a vector such that $v^\pm=\gamma^\pm$. Since 
\[\lim_{n\to+\infty}\rho(\gamma^n)\cdot b= \xi_\rho^1(\gamma^+),\]  
Observation~\ref{obs:delta1} implies that if 
\[X_n:=\overline{\rho}(\gamma^n)\cdot \left(z_1e_1 + z_m e_m + \sum_{j=m+1}^d z_j e_j\right),\] 
then there is some $A\geq 1$ such that for sufficiently large $n$,
\[\frac{1}{A} \frac{\norm{P_{3,v}(X_n)}_{2}}{\norm{P_{1,v}(X_n)}_{2}}\leq d_{\Pb}\left( \rho(\gamma^n)\cdot b, \xi_\rho^m(\gamma^+) \right) \leq  A \frac{\norm{P_{3,v}(X_n)}_{2}}{\norm{P_{1,v}(X_n)}_{2}}\]
and 
\[\frac{1}{A} \frac{\norm{P_{2,v}(X_n)}_{2}}{\norm{P_{1,v}(X_n)}_{2}}\leq d_{\Pb}\left( \rho(\gamma^n)\cdot b, \xi_\rho^1(\gamma^+) \right) \leq  A \frac{\norm{P_{2,v}(X_n)}_{2}+\norm{P_{3,v}(X_n)}_{2}}{\norm{P_{1,v}(X_n)}_{2}}.\]
It then follows from \eqref{eqn:P2stretch} that
\begin{align*}
\lim_{n \rightarrow +\infty} \frac{1}{n} \log d_{\Pb}\left(\rho(\gamma^n)\cdot b, \xi_\rho^m(\gamma^+) \right)=\log \frac{\lambda_{m+1}(\rho(\gamma))}{\lambda_1(\rho(\gamma))}
\end{align*}
and 
\begin{align*}
\lim_{n \rightarrow +\infty} \frac{1}{n} \log d_{\Pb}\left(\rho(\gamma^n)\cdot b, \xi_\rho^1(\gamma^+) \right)= \log \frac{\lambda_{m}(\rho(\gamma))}{\lambda_1(\rho(\gamma))},
\end{align*}
where the second equality uses the fact that $\lambda_m(\rho(\gamma))>\lambda_{m+1}(\rho(\gamma))$, which is a consequence of the assumption that $\rho$ is $P_m$-Anosov.

Finally, if $M$ is $C^\alpha$ along $\xi_\rho^1(\partial_\infty \Gamma)$, then by Observation~\ref{obs:appendix3 estimate 1} there exists $C > 0$ such that
\begin{align*}
d_{\Pb}\left( \xi_\rho^m(\gamma^+), \rho(\gamma^n)\cdot b \right) \leq C d_{\Pb}\left( \xi_\rho^1(\gamma^+), \rho(\gamma^n)\cdot b \right)^{\alpha}
\end{align*}
for any $n$ and any $b\in M-\xi_\rho^1(\gamma^+)$. By taking the logarithm on both sides, dividing by $n$, and then taking the limit, we see that
\begin{align*}
\alpha \leq \log \frac{\lambda_{1}(\rho(\gamma))}{\lambda_{m+1}(\rho(\gamma))}\bigg/\log \frac{\lambda_{1}(\rho(\gamma))}{\lambda_m(\rho(\gamma))}.
\end{align*}
Since $\gamma \in \Gamma$ was arbitrary, we see that $\alpha \leq \Es_m(\rho)$.

\section{Necessary conditions for differentiability of $\rho$-controlled subsets}
In this section, we establish Theorem~\ref{thm:nec_general}. By Example \ref{eg:limitset}, it is sufficient to prove the following theorem.

\begin{theorem} \label{thm:nec_general_body}Suppose  $\rho: \Gamma \rightarrow \PGL_{d}(\Rb)$ is an irreducible $P_1$-Anosov representation such that $\bigwedge^m \rho: \Gamma \rightarrow \PGL\big(\bigwedge^m \Rb^d\big)$ is also irreducible. Also, suppose  $M$ is a $\rho$-controlled, $(m-1)$ dimensional topological manifold. If
\begin{enumerate}
\item[($\ddagger$)] $M$ is $C^\alpha$ along $\xi_\rho^1(\partial_\infty\Gamma)$ for some $\alpha>1$,
\end{enumerate}
then
\begin{enumerate}
\item[($\dagger$')] $\rho$ is $P_m$-Anosov and $\xi_\rho^1(x) + p + \xi_\rho^{d-m}(y)$
is a direct sum for all pairwise distinct $\xi_\rho^1(x),p,\xi_\rho^1(y) \in M$.
\end{enumerate}
\end{theorem}

\begin{remark}
Note that ($\dagger'$) in Theorem \ref{thm:nec_general_body} is a weaker than the $m$-hyperconvexity condition ($\dagger$) in Theorem \ref{thm:main_body}. However, when $M=\xi_\rho^1(\partial_\infty\Gamma)$, then the two conditions are identical.
\end{remark}

In Section \ref{sec:wedge}, we recall the definition of the representation 
\[\bigwedge^m\rho:\Gamma\to\PGL\bigg(\bigwedge^m\Rb^d\bigg),\] 
whose irreducibility appears as a hypothesis in the statements of Theorem~\ref{thm:nec_general_body}. Then, in Section \ref{sec:irredeg}, we give an example to demonstrate the necessity of the irreducibility of $\bigwedge^m\rho$ as a hypothesis of Theorem \ref{thm:nec_general_body} (and also in Theorem~\ref{thm:nec_general}). Finally, we give the proof of Theorem \ref{thm:nec_general_body} in Section \ref{thm: lkj}.

\subsection{The exterior representation}\label{sec:wedge}

Observe that for any $m\leq d-1$, there is a natural linear $\GL_d(\Rb)$-action on $\bigwedge^m\Rb^d$ given by 
\[g\cdot(u_1\wedge\dots\wedge u_m):=(g\cdot u_1)\wedge\dots\wedge (g\cdot u_m),\] 
where $u_i\in\Rb^d$ for all $i$. Furthermore, the standard basis of $\Rb^d$ induces a basis of $\bigwedge^m\Rb^d$, which gives an identification
\[\bigwedge^m\Rb^d\cong\Rb^D,\]
where $D:=\dim\big(\bigwedge^m\Rb^d\big)={d\choose m}$. Thus, this $\GL_d(\Rb)$-action defines a representation 
\[\overline{\iota_{d,m}}:\GL_d(\Rb)\to\GL\bigg(\bigwedge^m\Rb^d\bigg)\cong\GL_D\left(\Rb\right),\]
which in turn defines a representation
\[\iota_{d,m}:\PGL_d(\Rb)\to\PGL\bigg(\bigwedge^m\Rb^d\bigg)\cong\PGL_D\left(\Rb\right).\]
Using this, we may define the \emph{$m$-exterior representation} of $\overline{\rho}:\Gamma\to\GL_d(\Rb)$ (resp. $\rho:\Gamma\to\PGL_d(\Rb)$) to be
\[\bigwedge^m\overline{\rho}:=\overline{\iota_{d,m}}\circ\overline{\rho}:\Gamma\to\GL_D\left(\Rb\right)\quad\left(\text{resp. }\bigwedge^m\rho:=\iota_{d,m}\circ\rho:\Gamma\to\PGL_D\left(\Rb\right)\right).\]


Next, let $g\in\PGL_d(\Rb)$. 
One can verify from the definition of the $\GL_d(\Rb)$-action on $\bigwedge^m\Rb^d$ that for all $i$, there exists $1\le i_1 < i_2 < \dots < i_m\le d$ and  $1\le j_1 < j_2 < \dots < j_m\le d$ such that
\begin{equation}\label{eqn:wedge eignevalues and singular values}
\lambda_i(\iota_{d,m}(g))=\lambda_{i_1}(g)\dots\lambda_{i_m}(g)\quad\text{and}\quad\mu_i(\iota_{d,m}(g))=\mu_{j_1}(g)\dots\mu_{j_m}(g).
\end{equation}
This implies that
\begin{align}\label{eqn:evalue1}
\lambda_1(\iota_{d,m}(g))=\prod_{i=1}^m\lambda_{i}(g)\quad\text{and}\quad\lambda_2(\iota_{d,m}(g))=\lambda_{m+1}(g)\prod_{i=1}^{m-1}\lambda_{i}(g),
\end{align}
and
\begin{align}\label{eqn:0evalue1}
\mu_1(\iota_{d,m}(g))=\prod_{i=1}^m\mu_{i}(g)\quad\text{and}\quad\mu_2(\iota_{d,m}(g))=\mu_{m+1}(g)\prod_{i=1}^{m-1}\mu_{i}(g).
\end{align}
Hence, for any $\gamma\in\Gamma$,
\begin{align}\label{eqn:evalue2}
\frac{\lambda_1}{\lambda_2}\bigg(\bigwedge^m\rho(\gamma)\bigg)=\frac{\lambda_m(\rho(\gamma))}{\lambda_{m+1}(\rho(\gamma))}
\end{align}
and
\begin{align}\label{eqn:0evalue2}
\frac{\mu_1}{\mu_2}\bigg(\bigwedge^m\rho(\gamma)\bigg)=\frac{\mu_m(\rho(\gamma))}{\mu_{m+1}(\rho(\gamma))}.
\end{align}

\subsection{The irreducibility condition}\label{sec:irredeg}
Now, we will discuss an example to demonstrate that the irreducibility of $\bigwedge^m\rho$ is a necessary hypothesis for Theorem~\ref{thm:nec_general_body} to hold.

The identification of $\Cb^3$ with $\Rb^6$ given by 
\[(z_1,z_2,z_3)\mapsto(\Re(z_1),\Im(z_1),\Re(z_2),\Im(z_2),\Re(z_3),\Im(z_3))\]
defines an inclusion $j:\SL_3(\Cb)\to\SL_6(\Rb)$. The image of $j$ can be characterized as the subgroup of $\SL_6(\Rb)$ that commutes with the linear endomorphism $J$ on $\Rb^6$ defined by $J(x_1,y_1,x_2,y_2,x_3,y_3):=(-y_1,x_1,-y_2,x_2,-y_3,x_3)$. Notice that $J^2 = -\id$.

Let $G\subset\SL_3(\Cb)$ be the subgroup that leaves invariant the sesquilinear pairing that is represented in the standard basis of $\Cb^3$ by the matrix
\[\left(\begin{array}{ccc}
0&0&1\\
0&1&0\\
1&0&0
\end{array}\right).\]
Note that $G$ is conjugate to the group $\SU(2,1)$. Define 
\[\tau:=(\iota_{6,2}\circ j)|_{G} : G \rightarrow \SL\bigg(\bigwedge^2 \Rb^6\bigg),\] 
where $\iota_{d,m}$ was defined in Section \ref{sec:wedge}. 

Let $\bigwedge^2J$ be the linear endomorphism on $\bigwedge^2\Rb^6$ given by 
\[\bigg(\bigwedge^2J\bigg)(u_1\wedge u_2)=J(u_1)\wedge J(u_2).\]
Notice that $\big(\bigwedge^2 J\big)^2 = \id$ and so $\bigwedge^2J$ has eigenvalues $\pm 1$. Consider the $\tau$-invariant subspace 
\begin{align*}
E := \left\{ v \in \bigwedge^2 \Rb^6: \bigg(\bigwedge^2J\bigg)(v) = v\right\},
\end{align*}
and let $\tau : G \rightarrow \GL(E)$ be the representation defined by the $\tau_0$ action on $E$. Observe that if $e_1,\dots,e_6$ is the standard basis for $\Rb^6$, then 
\begin{align*}
&f_1:=e_1\wedge e_2,&&f_4:=e_3\wedge e_4,&&f_7:=e_3\wedge e_5+e_4\wedge e_6,\\ 
&f_2:=e_2\wedge e_3-e_1\wedge e_4,&&f_5:=e_2\wedge e_5-e_1\wedge e_6,&&f_8:=e_4\wedge e_5-e_3\wedge e_6,\\
&f_3:=e_1\wedge e_3+e_2\wedge e_4,&&f_6:=e_1\wedge e_5+e_2\wedge e_6,&&f_9:=e_5\wedge e_6
\end{align*}
is a basis of $E$. One can then explicitly verify that $\tau$ is irreducible. 

If $g \in G$, then there exists some $\lambda \geq 1$ such that the (complex) eigenvalues of $g \in G$ have moduli $\lambda, 1, \lambda^{-1}$. By conjugating $g$ by an appropriate element in $G$, we may also assume that the generalized eigenvectors of $g$ whose eigenvalues have moduli $\lambda,1,\lambda^{-1}$ are $(1,0,0)^T$, $(0,1,0)^T$ and $(0,0,1)^T$ respectively. This implies that the eigenvalues of $j(g) \in \SL_6(\Rb)$ have moduli $\lambda,1,\lambda^{-1}$ (each with multiplicity $2$), and the corresponding invariant subspaces are $\Span_{\Rb}\{e_1,e_2\}$, $\Span_{\Rb}\{e_3,e_4\}$, and $\Span_{\Rb}\{e_5,e_6\}$ respectively. Using the basis of $E$ described above, one can then compute that the eigenvalues of $\tau(g)$ have moduli
\begin{align}\label{eqn: evalue calc}
\lambda^2, \lambda, \lambda, 1,1,1, \lambda^{-1}, \lambda^{-1}, \lambda^{-2}.
\end{align}
In particular, the image of $\tau$ lies in $\SL(E)$.

With this set up, we can now give our example. It describes a $P_1$-Anosov, irreducible representation of a co-compact lattice $\Gamma\subset G$ to $\SL(E)$ where the $P_1$-limit set of this representation is a $3$-dimensional, $C^\infty$-submanifold of $\Pb(E)$, but $\rho$ is not $4$-Anosov.

\begin{example}\label{ex:irred_bad_example} Fix a co-compact lattice $\Gamma \leq G$, and let $\rho:=\tau|_{\Gamma}$. Since $G$ is a rank-one Lie group, it acts transitively and by isometries on the negatively curved Riemannian symmetric space $\Hb_{\Cb}^2$ (the $2$-dimensional complex hyperbolic space), whose visual boundary $\partial_\infty\Hb_{\Cb}^2$ has the structure of a $3$-dimensional smooth sphere. Thus, the inclusion of $\Gamma$ into $G$ specifies an identification of $\partial_\infty\Gamma$ with $\partial_\infty\Hb_{\Cb}^2$.

As $G$-spaces, $\partial_\infty\Hb_{\Cb}^2\simeq G/B$, where $B\subset G$ is the subgroup of upper triangular matrices. It is straightforward to check that $\tau(B)\subset\SL(E)$ lies in $P\cap Q$, where $P\subset\SL(E)$ is the subgroup that preserves the line spanned by $f_1$, and $Q\subset\SL(E)$ is the subgroup that preserves $\Span_{\Rb}(f_1,\dots,f_8)$. In particular, there are smooth, $\tau$-equivariant maps
\[\xi_\rho^1:\partial_\infty\Gamma=G/B\to\SL(E)/P=\Pb(E)\]
and
\[\xi_\rho^{8}:\partial_\infty\Gamma=G/B\to\SL(E)/Q=\Pb^*(E)=\Gr_8(E).\]
Furthermore, a result of Guichard-Wienhard \cite[Proposition 4.4]{GW2012} imply that $\rho:\Gamma\to\SL(E)$ is a $P_1$-Anosov representation whose $P_1$-limit map and $P_8$-limit map are $\xi_\rho^1$ and $\xi_\rho^{8}$ respectively. The eigenvalue calculation \eqref{eqn: evalue calc} implies that $\rho$ is not $4$-Anosov. However, the $P_1$-limit set $\xi_\rho^1(\partial_\infty \Gamma)$ is a $3$-dimensional $C^\infty$-submanifold of $\Pb(E)$. As a consequence of Theorem~\ref{thm:nec_general_body}, $\bigwedge^4E$ is not irreducible.
\end{example}

\subsection{The proof of Theorem \ref{thm:nec_general_body}}\label{thm: lkj}
Let $\rho:\Gamma\to\PGL_d(\Rb)$ be a $P_1$-Anosov representation and $M$ is a $\rho$-controlled subset which is $C^1$ along $\xi_\rho^1(\partial_\infty\Gamma)$. To prove Theorem \ref{thm:nec_general_body}, we will use the map
\begin{align}\label{eqn:Phi}
\Phi: \xi_\rho^1(\partial_\infty\Gamma) \rightarrow \Pb\bigg(\bigwedge^{m} \Rb^{d}\bigg).
\end{align}
defined as follows.
Let 
\begin{align*}
F_{d,m}:\Gr_m(\Rb^d)\to\Pb\bigg(\bigwedge^m\Rb^d\bigg)
\end{align*}
denote Pl\"ucker embedding, that is the map defined by
\begin{align}\label{eqn:Fdm}
F_{d,m}(V)=\left[v_1 \wedge \cdots \wedge v_m \right],
\end{align}
where $v_1,\dots,v_m$ is some (any) basis of $V$. Note that $F_{d,m}$ is smooth and $\iota_{d,m}$-equivariant. Since $M$ is differentiable along $\xi_\rho^1(\partial_\infty\Gamma)$, we can define 
\[\overline{\Phi}:\xi_\rho^1(\partial_\infty\Gamma)\to\Gr_m(\Rb^d)\] 
to be the map that associates to every point in $x \in \xi_\rho^1(\partial_\infty\Gamma)$ the tangent space to $M$ at $x$ (see part (3) of Remark \ref{rmk:open}). Then define $\Phi:= F_{d,m}\circ\overline{\Phi}$.

 First, we prove the following lemma, which tells us that when $\rho$ is irreducible, the conclusion of Theorem~\ref{thm:nec_general_body} can be deduced from certain eigenvalue and eigenspace assumptions. 

\begin{lemma}\label{lem:nec_general}
Suppose  $\rho$ is irreducible, that there exists $\alpha>1$ such that 
\begin{align}
\label{eqn:comparing ratios}
\frac{ \lambda_{m+1}(\rho(\gamma))}{\lambda_m(\rho(\gamma))} \leq \left(\frac{ \lambda_{2}(\rho(\gamma))}{\lambda_1(\rho(\gamma))} \right)^{\alpha-1} 
\end{align}
for all $\gamma\in\Gamma$, and that $\Phi(\xi_\rho^1(\gamma^+))\in\Pb\big(\bigwedge^m\Rb^d\big)$ is the attracting fixed point of $\left(\bigwedge^{m} \rho\right)(\gamma)$ for all $\gamma\in\Gamma$ with infinite order. Then $\rho$ is $P_m$-Anosov and $\xi_\rho^1(x) + p +  \xi_\rho^{d-m}(y)$ is a direct sum for all pairwise distinct $\xi_\rho^1(x), p,  \xi_\rho^1(y) \in M$.
\end{lemma}

\begin{remark} \
\begin{enumerate}
\item Notice that in the context of Lemma~\ref{lem:nec_general}, since $\rho$ is $P_1$-Anosov, any infinite order element has proximal image. Then Equation~\eqref{eqn:comparing ratios} implies that $\left(\bigwedge^{m} \rho\right)(\gamma)$ is also proximal and hence has an attracting fixed point in $\Pb\big(\bigwedge^m\Rb^d\big)$. 
\item After the first version of this paper was written, Kassel and Potrie proved a general result \cite[Proposition 1.2]{PK22} from which one can deduce that $\rho$ is $P_m$-Anosov from the assumptions of Lemma \ref{lem:nec_general}.
\end{enumerate}
\end{remark} 

Thus, to prove Theorem \ref{thm:nec_general_body}, one now needs to establish that the hypotheses of Lemma \ref{lem:nec_general} hold. This is an immediate consequence of the following lemma.

\begin{lemma}\label{lem:proximal} \label{prop:eigenvalue_est}
Suppose  $\bigwedge^m \rho: \Gamma \rightarrow \PGL\big(\bigwedge^m \Rb^d\big)$ is irreducible and $M$ is an $(m-1)$-dimensional topological manifold that is $C^{\alpha}$ along the $P_1$-limit set of $\rho$ for some $\alpha>1$. Then:
\begin{enumerate}
\item If $\gamma \in \Gamma$ has infinite order, then $g:=\left(\bigwedge^{m} \rho\right)(\gamma)$ is proximal and $\Phi(\xi_\rho^1(\gamma^+))\in\Pb\big(\bigwedge^m\Rb^d\big)$ is the attracting fixed point of $g$. 
\item If $\gamma \in \Gamma$, then 
\begin{align*}
\frac{ \lambda_{m+1}(\rho(\gamma))}{\lambda_m(\rho(\gamma))} \leq \left(\frac{ \lambda_{2}(\rho(\gamma))}{\lambda_1(\rho(\gamma))} \right)^{\alpha-1}. 
\end{align*}
\end{enumerate}
\end{lemma}

The remainder of this section is the proof of Lemma \ref{lem:nec_general} and Lemma \ref{lem:proximal}.

\subsubsection{The proof of Lemma \ref{lem:nec_general}}\label{sec:sgap}

First, notice that the reduction made in Remark~\ref{rem:lift} does not impact the hypotheses nor conclusion of the lemma. So we may assume that there exists a lift $\overline{\rho}:\Gamma\to\SL_d(\Rb)$ of $\rho$. Second, notice that passing to a finite index subgroup also does not impact the hypotheses or conclusion of the Lemma (see Proposition~\ref{prop:strongly_irreducible}). Hence we may also assume that the Zariski closure of $\overline{\rho}(\Gamma)$ is connected. 

We first prove the following preliminary lemma.

\begin{lemma}\label{lem:est_on_sing_values} 

If $1<\beta < \alpha$, then there exists $C > 0$ such that 
\begin{align*}
\log \frac{ \mu_{m+1}(\rho(\gamma))}{\mu_m(\rho(\gamma))} \leq (\beta-1) \log \frac{ \mu_2(\rho(\gamma))}{\mu_1(\rho(\gamma))} +C
\end{align*}
for all $\gamma \in \Gamma$. 
\end{lemma}

\begin{proof} 
Let $\Cc_\mu:=\Cc_\mu(\rho(\Gamma))$ and $\Cc_\lambda:=\Cc_\lambda(\rho(\Gamma))$ be the cones defined in Section~\ref{sec:cones}. Then $\Cc_\mu = \Cc_\lambda$ by Proposition~\ref{prop:Zclosure} and Theorem~\ref{thm:cones}. By hypothesis, if $x=(x_1, \dots, x_d) \in \Cc_\lambda$, then 
\begin{align*}
x_{m+1}-x_m \leq (\alpha-1) (x_2 - x_1).
\end{align*}
Further, since $\rho$ is $P_1$-Anosov, $x_2 - x_1 <0$ for all $x=(x_1, \dots, x_d) \in \Cc_\lambda$.

We will now prove that there exists $R > 0$ with the following property: if $\norm{\mu(\rho(\gamma))}_2 \geq R$, then 
 \begin{align*}
\log \frac{ \mu_{m+1}(\rho(\gamma))}{\mu_m(\rho(\gamma))} \leq (\beta-1) \log \frac{ \mu_2(\rho(\gamma))}{\mu_1(\rho(\gamma))}.
\end{align*}
Suppose for contradiction that there exists $\{\gamma_n\}_{n \geq 1} \subset \Gamma$ with $\norm{\mu(\rho(\gamma_n))}_2 \rightarrow +\infty$ and
 \begin{align*}
\log \frac{ \mu_{m+1}(\rho(\gamma_n))}{\mu_m(\rho(\gamma_n))} > (\beta-1) \log \frac{ \mu_2(\rho(\gamma_n))}{\mu_1(\rho(\gamma_n))}.
\end{align*}
By passing to a subsequence, we can assume that 
\begin{align*}
\frac{1}{\norm{\mu(\rho(\gamma_n))}_2} \mu(\rho(\gamma_n)) \rightarrow x=(x_1, \dots, x_d).
\end{align*}
Then $x \in \Cc_\mu  = \Cc_\lambda$ and
\begin{align*}
x_{m+1}-x_m \geq (\beta-1)( x_2 -x_1) > (\alpha-1) (x_2 - x_1)
\end{align*}
so we have a contradiction. 

The lemma then follows from the observation that since $\rho$ is $P_1$-Anosov, the set $\{ \gamma \in \Gamma : \norm{\mu(\rho(\gamma))}_2 < R\}$ is finite.
\end{proof}


Since $\rho$ is $P_1$-Anosov, Theorem~\ref{thm:SV_char_of_Anosov} implies that there exists $C_0, c_0  > 0$ such that 
\begin{align*}
\frac{ \mu_2(\rho(\gamma))}{\mu_1(\rho(\gamma))} \leq C_0 e^{-c_0 d_S(\gamma, \id)}
\end{align*}
for all $\gamma \in \Gamma$. Then by Lemma \ref{lem:est_on_sing_values}, there exists $C,c > 0$ such that 
\begin{align*}
\frac{ \mu_{m+1}(\rho(\gamma))}{\mu_m(\rho(\gamma))} \leq C e^{-c d_S(\gamma, \id)}
\end{align*}
for all $\gamma \in \Gamma$. Thus, Theorem~\ref{thm:SV_char_of_Anosov}  implies that $\rho$ is $P_m$-Anosov. 

To finish the proof, it suffices to show that $\xi_\rho^1(x) + p + \xi_\rho^{d-m}(y)$
is a direct sum for all pairwise distinct $\xi_\rho^1(x),p,\xi_\rho^1(y) \in M$. Equivalently, we need to show that if 
\[\wh{M}: = \{ (\xi_\rho^1(x),p,\xi_\rho^1(y)) \in M^3 : \xi_\rho^1(x),p,\xi_\rho^1(y) \text{ are pairwise distinct} \}\] 
and 
\[\Oc :=\left\{(\xi_\rho^1(x),p,\xi_\rho^1(y))\in \wh{M}:\xi_\rho^1(x) + p+ \xi_\rho^{d-m}(y)\text{ is a direct sum}\right\},\]
then $\wh M= \Oc$. 

Given a pair of non-zero linear subspaces $V,W\subset\Rb^d$, we denote
\[\angle_{\min}(V,W):=\min\left\{\angle(v,w):v\in V-\{0\},w\in W-\{0\}\right\}\]
and 
\[\angle_{\max}(V,W):=\max\left\{\angle(v,w):v\in V-\{0\},w\in W-\{0\}\right\}.\]

Recall that $\Gamma$ acts co-compactly on $\Usf(\Gamma)$, the flow space associated to $\Gamma$ described in Section~\ref{sec:flowspace}. Hence, there exists a compact set $K \subset \Usf(\Gamma)$ such that $\Gamma \cdot K = \Usf(\Gamma)$. Then define
\begin{align*}
0 < \epsilon := \min\{ d_{\Pb}(\xi_\rho^1(v^+), \xi_\rho^1(v^-)) : v \in K\},
\end{align*}
where $d_{\Pb}$ is a distance on $\Pb(\Rb^d)$ induced by a Riemannian metric. Since 
\[\{(x,y)\in\partial_\infty\Gamma^2:d_{\Pb}(\xi_\rho^1(x), \xi_\rho^1(y)) \geq \epsilon\}\] 
is compact and $\xi_\rho^m(x) + \xi_\rho^{d-m}(y)=\Rb^d$ when $x \neq y$, there exists $\theta_0 > 0$ with the following property: if $x,y \in \partial_\infty \Gamma$ and $d_{\Pb}(\xi_\rho^1(x), \xi_\rho^1(y)) \geq \epsilon$, then 
\begin{align*}
\angle_{\min}(\xi_\rho^m(x),\xi_\rho^{d-m}(y))\geq\theta_0.
\end{align*}

By hypothesis, if $\gamma \in \Gamma$ has infinite order and $\gamma^+\in \partial_\infty \Gamma$ is the attracting fixed point of $\gamma$, then $\Phi(\xi_\rho^1(\gamma^+))$ is the attracting fixed point of $(\bigwedge^m\rho)(\gamma)$. This implies that $T_{\xi_\rho^1(\gamma^+)} M$, when viewed as an element in $\Gr_m(\Rb^d)$, is the attracting fixed point of $\rho(\gamma)$, so 
\begin{align*}
\xi_\rho^m(\gamma^+) =T_{\xi_\rho^1(\gamma^+)} M.
\end{align*}
Then by the continuity of $\xi_\rho^m$ and the density of $\{ \gamma^+: \gamma \in \Gamma \text{ has infinite order}\}$ in $\partial_\infty \Gamma$ we see that $\xi_\rho^m(x) = T_{\xi_\rho^1(x)} M$ for all $x \in \partial_\infty \Gamma$. Thus, the compactness of $\partial_\infty \Gamma$ implies that there exists $\delta > 0$ with the following property: if $\xi_\rho^1(x)$, $p \in M$ satisfy $d_{\Pb}(\xi_\rho^1(x), p) \leq \delta$, then
\begin{align*}
\angle_{\max} \left( \xi_\rho^1(x)+p,\xi_\rho^m(x)  \right) < \theta_0/2.
\end{align*}

With this, define
\[\Uc :=\left\{(\xi_\rho^1(x),p,\xi_\rho^1(y))\in \wh{M}:d_{\Pb}(\xi_\rho^1(x),\xi_\rho^1(y)) \geq \epsilon\text{ and } d_{\Pb}(\xi_\rho^1(x), p) \leq \delta\right\},\]  
and note that $\Uc\subset \Oc$. Indeed, if $(\xi_\rho^1(x),p,\xi_\rho^1(y)) \in \Uc$, then 
\begin{align*}
\angle_{\min}\left( \xi_\rho^1(x)+p,\xi_\rho^{d-m}(y)  \right)  & >  \angle_{\min}\left( \xi_\rho^m(x),\xi_\rho^{d-m}(y)\right)-\angle_{\max}\left( \xi_\rho^1(x)+p,\xi_\rho^m(x)  \right)\\
&\ge\theta_0/2.
\end{align*}
which implies that $\xi_\rho^1(x) +p+ \xi_\rho^{d-m}(y)$ is direct. 

Next, let $P(M) \subset \Usf(\Gamma) \times M$ be the set defined by \eqref{eqn:P(M)}, and recall that $\varphi_t$ denotes the geodesic flow on $\Usf(\Gamma)$.
Notice that if $v \in \Usf(\Gamma)$, then 
$$
\lim_{t \rightarrow +\infty} \max\left\{ d_{\Pb}\left(p, \xi_\rho^1(v^+) \right): (\varphi_t(v),p) \in P(M) \right\} = 0 
$$
and the convergence is uniform on compact subsets of $\Usf(\Gamma)$. So  there exists $T \geq 0$ such that if $v \in K$, $t \geq T$, and $(\varphi_t(v),p) \in P(M)$, then $(\xi_\rho^1(v^+),p,\xi_\rho^1(v^-)) \in \Uc$. 

Now, choose any $(\xi_\rho^1(x),p,\xi_\rho^1(y)) \in \wh{M}$. From the definition of $P(M)$, there exists $v \in \Usf(\Gamma)$ such that $v^+ = x$, $v^- = y$, and $(v,p) \in P(M)$. Further, there exists $\gamma \in \Gamma$ such that $w := \gamma \cdot\varphi_{-T}(v) \in K$. Since the $\Gamma$-action on $\Usf(\Gamma)$ commutes with the geodesic flow,
\begin{align*}
(\varphi_T(w),\rho(\gamma)\cdot p) = \gamma\cdot ( v, p) \in P(M)
\end{align*}
and so $\gamma\cdot (\xi_\rho^1(x),p,\xi_\rho^1(y))=(\xi_\rho^1(w^+),\rho(\gamma)\cdot p,\xi_\rho^1(w^-))\in \Uc\subset\Oc$. Since $\Oc$ is $\Gamma$-invariant, $(\xi_\rho^1(x),p,\xi_\rho^1(y)) \in \Oc$. Thus, $\wh{M}\subset\Oc$.

\subsubsection{The proof of Lemma \ref{lem:proximal}}\label{sec:egap} 
Let $d_{\Pb}$ be a distance on $\Pb(\Rb^d)$ induced by a Riemannian metric. Suppose  $g\in\PGL_d(\Rb)$ has a fixed point $l\in\Pb(\Rb^d)$ corresponding to $\lambda_1(g)$. We first observe two different estimates for 
\[\limsup_{n \rightarrow +\infty} \frac{1}{n} \log d_{\Pb}\Big(g^{n}\cdot p, l\Big),\]
where $p\in\Pb(\Rb^d)$ is generic. The first (Observation \ref{obs:dynamics}) holds when the generalized eigenspace of $g$ corresponding to $\lambda_1(g)$ has dimension $1$, while the second (Observation \ref{obs:slow}) holds when this generalized eigenspace has dimension strictly larger than $1$ but strictly less than $d$.

\begin{observation}\label{obs:dynamics} Suppose  $g \in \PGL_{d}(\Rb)$ is proximal, and let $g^+ \in \Pb(\Rb^d)$ and $g^-\in\Gr_{d-1}(\Rb^d)$ be the attracting fixed point and repelling fixed hyperplane of $g$ respectively. If $p\in\Pb(\Rb^d)$ satisfies $p\neq g^+$ and $p\notin g^-$, then
\begin{align*}
\log\frac{\lambda_2(g)}{\lambda_1(g)} \geq \limsup_{n \rightarrow +\infty} \frac{1}{n} \log d_{\Pb}\Big(g^n\cdot p, g^+ \Big).
\end{align*}
Moreover, if $V\subset\Rb^d$ is the sum of $g^+$ and all the generalized eigenspaces of $g$ corresponding to the eigenvalues with modulus strictly less than $\lambda_2(g)$, then the above inequality holds as equality when $p$ does not lie in $V\cup g^-$.
\end{observation}

\begin{proof} 
Note that the affine chart 
$$
\Ab_{g^-} = \Pb(\Rb^d) - g^-
$$ 
contains both $p$ and $g^+$. Equip $\Ab_{g^-}$ with an Euclidean metric $d_{\Ab}$, and let $\Bb$ be the unit ball in $\Ab_{g^-}$ centered at $g^+$. Since $p\notin g^-$, $g^n\cdot p\in\Bb$ for sufficiently large $n$. On $\Bb$, $d_{\Pb}$ and $d_{\Ab}$ are bi-Lipschitz, so there is a constant $A$ such that for sufficiently large $n$,
\begin{align}\label{eqn:approx}
\frac{1}{A}\frac{\norm{P_2(\overline{g}^n\cdot X)}_{2}}{\norm{P_1(\overline{g}^n\cdot X)}_{2}}\leq d_{\Pb}(g^n\cdot p,g^+)\leq A\frac{\norm{P_2(\overline{g}^n\cdot X)}_{2}}{\norm{P_1(\overline{g}^n\cdot X)}_{2}},
\end{align}
where $X\in p$ is a non-zero vector, $P_1:\Rb^d\to g^+$ is the projection with kernel $g^-$, $P_2:\Rb^d\to g^-$ is the projection with kernel $g^+$, and $\overline{g}\in\GL_d(\Rb)$ is a linear representative of $g$. On the other hand, it is straightforward that 
\begin{align}\label{eqn:proj}
\log\frac{\lambda_2(g)}{\lambda_1(g)} \geq \limsup_{n \rightarrow +\infty} \frac{1}{n} \log \frac{\norm{P_2(\overline{g}^n\cdot X)}_{2}}{\norm{P_1(\overline{g}^n\cdot X)}_{2}},
\end{align}
thus giving the desired inequality.

Choose a basis $\{e_1,\dots,e_d\}$ for $\Rb^d$ such that $g$ is in real Jordan normal form in this basis. We may assume that $e_1$ is an eigenvector of $g$ corresponding to $\lambda_1(g)$, and there is some $m$ such that $e_2,\dots,e_m$ spans the invariant subspace corresponding to $\lambda_2(g)$. Then $V$ is the span of $e_1,e_{m+1},\dots,e_d$, and it is easy to see that the inequality \eqref{eqn:proj} holds with equality when $p$ does not lie in $V$. 
 \end{proof}

\begin{observation}\label{obs:slow}
Let $g\in\PGL_d(\Rb)$, let $V_1$ be the generalized eigenspace of $g$ corresponding to $\lambda_1(g)$, and let $V_2$ be the sum of the generalized eigenspaces of $g$ corresponding to the eigenvalues of $g$ with modulus strictly less than $\lambda_1(g)$. Suppose  $g$ has a fixed point $l \in \Pb(\Rb^d)$ with $l \subset V_1$. Then for all $p\in \Pb(\Rb^d)-(l+V_2)$, 
\begin{align}
\label{eq:W_subspace_zero_limit}
0  = \lim_{n \rightarrow +\infty} \frac{1}{n} \log d_{\Pb}\Big(g^{n}\cdot p, l\Big).
\end{align}
\end{observation}

\begin{proof}
Fix $p \in \Pb(\Rb^d)-(l+V_2)$. First, note that since $d_{\Pb}$ has bounded diameter, 
\begin{align*}
 \limsup_{n \rightarrow +\infty} \frac{1}{n} \log d_{\Pb}\Big(g^{n}\cdot p, l\Big) \leq 0. 
\end{align*}
To establish the opposite inequality, fix a sequence $\{n_k\}_{k \geq 1}$ such that 
\begin{align*}
\liminf_{n \rightarrow +\infty} \frac{1}{n} \log d_{\Pb}\Big(g^{n}\cdot p, l\Big)=\lim_{k \rightarrow +\infty} \frac{1}{n_k} \log d_{\Pb}\Big(g^{n_k}\cdot p, l\Big).
\end{align*}
If $g^{n_k}\cdot p$ does not converge to $l$ as $k\to+\infty$, then we have 
$$
\lim_{k \rightarrow +\infty} \frac{1}{n_k} \log d_{\Pb}\Big(g^{n_k}\cdot p, l\Big) = 0
$$
and the proof is complete. So we can suppose that $g^{n_k}\cdot p \rightarrow l$ as $k\to+\infty$. 

Using the real Jordan normal form of $g$, we can decompose $V_1 = \bigoplus_{j=1}^r V_{1,j}$ where
\begin{enumerate}
\item $V_{1,1} = l$,
 \item for $2 \leq j \leq r$ 
 \begin{enumerate}
 \item $V_{1,j}$ is either one or two dimensional,
 \item there exists a linear transformation $L_j : V_{1,j} \rightarrow V_{1,j}$ such that
\begin{align*}
g \cdot Y \in \Rb \cdot (L_j \cdot Y) + V_{1,j-1}
\end{align*}
for all $Y\in V_{1,j}$,
\item $\norm{L_j \cdot Y}_{2} = \lambda_1(g)\norm{Y}_{2}$ for all $Y \in V_{1,j}$. 
 \end{enumerate}

\end{enumerate}
Also, let $P_{1,j} : \Rb^d \rightarrow V_{1,j}$ and $P_2 : \Rb^d \rightarrow V_{2}$ be the projections relative to the decomposition $\Rb^d = V_{1,1} \oplus \dots \oplus V_{1,r}\oplus V_2$. 

Choose an affine chart $\Ab\subset\Pb(\Rb^d)$ containing $l$, and let $d_{\Ab}$ denote an Euclidean metric on $\Ab$. Let $\Bb\subset\Ab$ be the unit ball in the metric $d_{\Ab}$ centered at $l$. Since $d_{\Pb}$ and $d_{\Ab}$ are bi-Lipschitz when restricted to $\Bb$ and $g^{n_k}\cdot p$ converges to $l$, there exists $A \geq 1$ such that if $\overline{g}\in\GL_d(\Rb)$ is a linear representative of $g$, then
\begin{align}
\label{eq:affine_chart_estimate_non_proximal}
\frac{1}{A} \left( \frac{\sum_{j=1}^\ell \norm{P_{1,j}(\overline{g}^{n_k}\cdot X)}_{2} + \norm{P_2(\overline{g}^{n_k}\cdot X)}_{2}}{\norm{P_{1,1}(\overline{g}^{n_k}\cdot X)}_{2}} \right) \leq d_{\Pb}\Big(g^{n_k}\cdot p, l \Big)
\end{align}
for all non-zero $X\in p$ and all sufficiently large $k$. Since $X \notin l+V_2$, there exists $2 \leq j_0 \leq r$ such that $P_{1,j_0}(X) \neq 0$. By increasing $j_0$ if necessary, we can also assume that $P_{1,j}(X) = 0$ for $j_0 < j \leq r$. This implies that
\begin{align}
\label{eq:P1j0_estimate}
 \norm{P_{1,j_0}(\overline{g}^{n_k}\cdot X)}_{2} = \lambda_1(g)^{n_k} \norm{P_{1,j_0}(X)}_{2}.
 \end{align}
Further, observe that
\begin{align}
\label{eq:P11_estimate}
\norm{P_{1,1}(\overline{g}^{n_k}\cdot X)}_{2} \leq \norm{\overline{g}^{n_k}\cdot X}_{2} \leq \mu_1(g^{n_k}) \norm{X}_{2}.
\end{align}

By Equations~\eqref{eq:affine_chart_estimate_non_proximal},~\eqref{eq:P1j0_estimate}, and~\eqref{eq:P11_estimate},
\begin{align*}
\lim_{k \rightarrow +\infty} \frac{1}{n_k} \log d_{\Pb}\Big(g^{n_k}\cdot p, l\Big) &\geq \limsup_{k \rightarrow +\infty} \frac{1}{n_k} \log  \frac{\norm{P_{1,j_0}(\overline{g}^{n_k}\cdot X)}_{2}}{A\norm{P_{1,1}(\overline{g}^{n_k}\cdot X)}_{2}} \\
& \geq \log \lambda_1(g)+\limsup_{k \rightarrow +\infty} \frac{1}{n_k} \log  \frac{\norm{P_{1,j_0}(X)}_{2}}{A\mu_1(g^{n_k})\norm{X}_{2}} \\
& \geq \log \lambda_1(g)-\liminf_{k \rightarrow +\infty} \frac{1}{n_k} \log(\mu_1(g^{n_k}))=0,
\end{align*}
where the last equality is the well known fact that
\begin{align*}
\lim_{n \rightarrow +\infty} \frac{1}{n} \log \mu_1(g^{n})= \log \lambda_1(g).
\end{align*}
Thus,
\[
\liminf_{n \rightarrow +\infty} \frac{1}{n} \log d_{\Pb}\Big(g^{n}\cdot p, l\Big)= \lim_{k \rightarrow +\infty} \frac{1}{n_k} \log d_{\Pb}\Big(g^{n_k}\cdot p, l \Big) \geq 0.\qedhere
\]
\end{proof}

\begin{proof}[Proof of  part (1) of Lemma \ref{lem:proximal}]
Let $\lambda_i := \lambda_i(\rho(\gamma))$ for $i=1,\dots,d$. By \eqref{eqn:evalue1}, $\lambda_1(g)=\lambda_1\cdots\lambda_m$ (recall that $g:=\left(\bigwedge^{m} \rho\right)(\gamma))$. Thus it is equivalent to prove that $g$ is proximal and $\Phi(\xi_\rho^1(\gamma^+))$ is an eigenline of $g$ whose eigenvalue has modulus $\lambda_1 \cdots \lambda_{m}$.

We first show that $\Phi(\xi_\rho^1(\gamma^+))$ is an eigenline of $g$ whose eigenvalue has modulus $\lambda_1 \cdots \lambda_{m}$. Let $\overline{g}\in \GL\big(\bigwedge^m \Rb^d\big)$ be a linear representative of $g$, and let $\{n_k\}_{k=1}^\infty$ be an increasing sequence of integers such that 
\begin{align*}
\frac{1}{\norm{\overline{g}^{n_k}}}\overline{g}^{n_k}
\end{align*}
converges to some $T \in \End\big(\bigwedge^{m} \Rb^{d}\big)$. Also, let $\bigwedge^m \Rb^d = V_1 \oplus V_2$ be the $\overline{g}$-invariant decomposition such that every eigenvalue of $\overline{g}|_{V_1}$ has modulus $\lambda_1 \cdots \lambda_m$ and every eigenvalue of $\overline{g}|_{V_2}$ has modulus strictly less than $\lambda_1 \cdots \lambda_m$. Observe that the image of $T$ is contained in $V_1$. Since 
\[g\cdot\Phi(\xi_\rho^1(\gamma^+))=\Phi(\xi_\rho^1(\gamma\cdot \gamma^+))=\Phi(\xi_\rho^1(\gamma^+)),\] 
$\Phi(\xi_\rho^1(\gamma^+))$ is an eigenline of $g$. Thus, we only need to show that $\Phi(\xi_\rho^1(\gamma^+))$ is contained in the image of $T$. 

We claim that the image of $T$ is exactly $\Phi(\xi_\rho^1(\gamma^+))$. Notice that if $p \in \Pb\big(\bigwedge^{m} \Rb^{d}\big) - \ker T$, then 
\begin{align}
\label{eqn:limit of g n k is T}
T(p) = \lim_{k \rightarrow +\infty} g^{n_k}\cdot p.
\end{align}
Since $\bigwedge^{m} \rho : \Gamma \rightarrow \PGL\big(\bigwedge^{m} \Rb^d\big)$ is irreducible, the set $\{ \Phi(x) : x \in \xi_\rho^1(\partial_\infty\Gamma)\}$ spans $\bigwedge^{m} \Rb^d$. Thus, if we set $D := \dim \bigwedge^{m} \Rb^{d}$, then there exists $x_1, \dots, x_D \in \partial_\infty \Gamma$ such that 
\begin{align*}
\Phi(\xi_\rho^1(x_1)), \dots, \Phi(\xi_\rho^1(x_D))
\end{align*}
span $\bigwedge^{m} \Rb^d$. Since $\partial_\infty \Gamma$ is perfect, by perturbing the $x_i$, we can assume that $\gamma^- \notin \{x_1, \dots, x_D\}$. Then by relabelling the $x_i$ we can also assume that there exists $1 \leq \ell \leq D$ such that
\begin{align*}
\Phi(\xi_\rho^1(x_1)) + \dots + \Phi(\xi_\rho^1(x_\ell)) + \ker T = \bigwedge^{m} \Rb^d
\end{align*}
is a direct sum. For $1 \leq i \leq \ell$, Equation~\eqref{eqn:limit of g n k is T} implies
\begin{align*}
T ( \Phi(\xi_\rho^1(x_i)) ) = \lim_{k \rightarrow +\infty} g^{n_k} \Phi(\xi_\rho^1(x_i)) =  \lim_{k \rightarrow +\infty} \Phi( \xi_\rho( \gamma^{n_k}\cdot x_i)) = \Phi(\xi_\rho^1(\gamma^+)),
\end{align*}
so the image of $T$ is $\Phi(\xi_\rho^1(\gamma^+))$. 

We next argue that $g$ is proximal, or equivalently that $\dim V_1 = 1$. Fix distances $d_1$ on $\Pb(\Rb^{d})$ and $d_2$ on $\Pb\big(\bigwedge^{m} \Rb^{d}\big)$ which are induced by Riemannian metrics. Since $M$ is $C^\alpha$ along $\xi_\rho^1(\partial_\infty\Gamma)$ and $F_{d,m}$ is smooth, a calculation shows that there is some $C\geq 1$ such that (see Observation \ref{obs: appendix3 estimate 2})
\begin{align}\label{eqn: d1d2}
d_2(\Phi(q_1), \Phi(q_2)) \leq C d_1(q_1, q_2)^{\alpha-1}
\end{align}
for all $q_1, q_2 \in \xi_\rho^1(\partial_\infty\Gamma)$.

Now, suppose for contradiction that $\dim V_1 > 1$.  Then
\[W := V_2 + \Phi(\xi_\rho^1(\gamma^+))\subset\bigwedge^{m}\Rb^d\] 
is a proper subspace. By Observation \ref{obs:slow},
\begin{align*}
0  = \lim_{n \rightarrow +\infty} \frac{1}{n} \log d_2\Big(g^{n}\cdot p, \Phi(\xi_\rho^1(\gamma^+))\Big)
\end{align*}
when $p \in \Pb\big(\bigwedge^{m} \Rb^d\big) - W$.

Since $\left\{ \Phi(x) : x \in \xi_\rho^1(\partial_\infty\Gamma)\right\}$ spans $\bigwedge^{m} \Rb^d$, there exists $x \in \partial_\infty \Gamma$ such that $\Phi(\xi_\rho^1(x)) \notin \Pb(W)$. By perturbing $x$ (if necessary) we can assume that $x \neq \gamma^-$. Then 
\begin{align*}
\lim_{n\to+\infty}\rho(\gamma)^n\cdot \xi_\rho^1(x)=\xi_\rho^1(\gamma^+) \quad \text{and} \quad \lim_{n\to+\infty}g^{n} \cdot\Phi(\xi_\rho^1(x))=\Phi(\xi_\rho^1(\gamma^+)).
\end{align*}
So, by Observation~\ref{obs:dynamics},
\begin{align*}
0 > \log \frac{\lambda_2}{\lambda_1} & \geq \limsup_{n \rightarrow +\infty} \frac{1}{n} \log d_1\Big(\rho(\gamma)^n\cdot \xi_\rho^1(x), \xi_\rho^1(\gamma^+)\Big) \\
&\geq  \limsup_{n \rightarrow +\infty} \frac{1}{(\alpha-1) n} \log d_2\Big( \Phi(\xi_\rho^1(\gamma^n\cdot x)),\Phi( \xi_\rho^1(\gamma^+))\Big) \\
& = \limsup_{n \rightarrow +\infty} \frac{1}{(\alpha-1) n} \log d_2\Big( g^n\cdot\Phi(\xi_\rho^1(x)),\Phi( \xi_\rho^1(\gamma^+))\Big) =0,
\end{align*}
where the last inequality is \eqref{eqn: d1d2}. This is a contradiction, so $g$ is proximal.
\end{proof}

\begin{proof}[Proof of  part (2) of Lemma \ref{lem:proximal}]  Fix some $\gamma \in \Gamma$. If $\gamma$ has finite order, then 
\begin{align*}
\frac{\lambda_i(\rho(\gamma))}{\lambda_j(\rho(\gamma))} = 1
\end{align*}
for all $1 \leq i,j \leq d$ and there is nothing to prove. So suppose that $\gamma$ has infinite order and let $\gamma^+ \in \partial_\infty \Gamma$ be the attracting fixed point of $\gamma$. By part (1), $g:=\bigwedge^m \rho(\gamma)$ is proximal and $\Phi(\xi_\rho^1(\gamma^+)) = g^+$.

By \eqref{eqn:evalue2} and Observation~\ref{obs:dynamics}, there exists a proper subspace $V \subset \bigwedge^m \Rb^d$ such that:
if $p\in \Pb(\bigwedge^m \Rb^{d}) - V$ and $p$ does not belong to the repelling hyperplane of $g$, then
\begin{align*}
\log \frac{ \lambda_{m+1}}{\lambda_m}(\rho(\gamma)) =  \lim_{n \rightarrow +\infty} \frac{1}{n} \log d_2\Big( g^n\cdot p, \Phi(\xi_\rho^1(\gamma^+)) \Big).
\end{align*}
Since $\{ \Phi(q) : q \in \xi_\rho^1(\partial_\infty\Gamma)\}$ spans $\bigwedge^m \Rb^d$ we can find $x \in \partial_\infty \Gamma$ such that $\Phi(\xi_\rho^1(x)) \notin \Pb(V)$. Since $\partial_\infty \Gamma$ is perfect, by perturbing $x$, we can also assume that $x \neq \gamma^-$. Then 
\begin{align*}
\lim_{n \rightarrow +\infty} g^n\cdot \Phi(\xi_\rho^1(x)) = \Phi(\xi_\rho^1(\gamma^+)),
\end{align*}
so $\Phi(\xi_\rho^1(x))$ does not lie in the repelling hyperplane of $g$. Thus, by \eqref{eqn: d1d2} and Observation~\ref{obs:dynamics},
\begin{align*}
\log \frac{ \lambda_{m+1}}{\lambda_m}(\rho(\gamma))  &\leq (\alpha-1) \lim_{n \rightarrow +\infty} \frac{1}{n} \log d_1\Big( \rho(\gamma)^n\cdot \xi_\rho^1(x), \xi_\rho^1(\gamma^+) \Big)\\
& \leq (\alpha-1) \log \frac{ \lambda_2}{\lambda_1}(\rho(\gamma)). \qedhere
\end{align*}
\end{proof}

\section{Necessary conditions for differentiability of $1$-dimensional $\rho$-controlled subsets}\label{sec:nec_surface}

In this section we prove Theorem~\ref{thm:nec_surface}. By Example \ref{eg:limitset}, it is sufficient to prove the following theorem.

\begin{theorem} \label{thm:nec_surface_body} Suppose $\Gamma$ is a hyperbolic group and $\rho: \Gamma \rightarrow \PGL_{d}(\Rb)$ is an irreducible $P_1$-Anosov representation. Also, suppose  $M$ is a $\rho$-controlled, topological circle. If
\begin{enumerate}
\item[($\ddagger$)] $M$ is $C^\alpha$ along $\xi_\rho^1(\partial_\infty\Gamma)$ for some $\alpha>1$,
\end{enumerate}
then
\begin{enumerate}
\item[($\dagger$')] $\rho$ is $P_2$-Anosov and $\xi_\rho^1(x) + p + \xi_\rho^{d-2}(y)$
is a direct sum for all pairwise distinct $\xi_\rho^1(x),p,\xi_\rho^1(y) \in M$.
\end{enumerate}
\end{theorem}

In Section \ref{sec:rhoirred}, we give an example to illustrate that the irreducibility of $\rho$ is a necessary hypothesis in Theorem \ref{thm:nec_surface_body} (and also in Theorem~\ref{thm:nec_surface}). Then we prove Theorem \ref{thm:nec_surface_body} in Section \ref{sec: dfgklj}.

\subsection{The irreducibility condition}\label{sec:rhoirred}
For $d \in \Nb$, let $\overline{\tau}_d : \GL_2(\Rb) \rightarrow \GL_d(\Rb)$ be the standard irreducible representation, which is constructed as follows. First, identify $\Rb^d$ with the space of homogeneous degree $d-1$ polynomials in two variables with real coefficients by 
\[(a_1,\dots,a_d)\mapsto \sum_{i=1}^{d}a_iX^{d-i}Y^{i-1}.\] 
Using this, we may define a linear representation $\overline{\tau}_d : \GL_2(\Rb) \rightarrow \GL_d(\Rb)$ by 
\begin{align*}
\overline{\tau}_d(g)P = P \circ g^{-1}. 
\end{align*}

One can verify that if $\lambda, \lambda^{-1}$ are the modulus of the eigenvalues of $g \in \GL_2(\Rb)$, then 
\begin{align}
\label{eq:eigenvalues_std_repn}
\lambda^{d-1}, \lambda^{d-3}, \dots, \lambda^{1-d}
\end{align}
are the moduli of the eigenvalues of $\overline{\tau}_d(g)$. Further, if $B_k\subset\GL_k(\Rb)$ denotes the subgroup of upper triangular matrices, then $\overline{\tau}_d(B_2)\subset B_d$. In particular, $\overline{\tau}_d$ induces a smooth map 
\[\Psi_d:\Pb(\Rb^2)\simeq\GL_2(\Rb)/B_2\to\GL_d(\Rb)/B_d.\] 
Since $\GL_d(\Rb)/B_d$ is the space of complete flags in $\Rb^d$, there is an obvious smooth projection $p_m:\GL_d(\Rb)/B_d\to\Gr_m(\Rb^d)$ for each $m=1,\dots,d-1$. Using this, define the smooth map $\Psi_{d,m}:=p_m\circ\Psi_d:\Pb(\Rb^2)\to\Gr_m(\Rb^d)$. One can verify that
\begin{enumerate}
\item $\Psi_{d,m}$ is $\overline{\tau}_d$-equivariant.
\item If $x,y \in \Pb(\Rb^2)$ are distinct, then $\Phi_{d,m}(x)$ and $\Phi_{d,d-m}(y)$ are transverse.
\item If $g \in \SL^\pm_2(\Rb)$ is proximal with attracting/repelling fixed points $g^+, g^- \in \Pb(\Rb^2)$, then $\overline{\tau}_d(g)$ is proximal with attracting fixed point $\Psi_{d,1}(g^+) \in \Pb(\Rb^d)$ and repelling fixed point $\Psi_{d,d-1}(g^-) \in \Gr_{d-1}(\Rb^d)$. 
\end{enumerate} 
Notice that (1) is by definition. Further, using (1), it is enough to verify (2) when $x=[1:0]$, $y=[0:1]$ and  to verify (3) when $g$ is diagonal. 

\begin{example}\label{ex:surface_bad_example} 
Fix a co-compact lattice $\Gamma \leq \PGL_2(\Rb)$. The inclusion of $\Gamma$ into $\PGL_2(\Rb)$ induces an identification $\partial_\infty\Gamma\simeq\Pb(\Rb^2)$, and thus equips $\partial_\infty\Gamma$ with the structure of a smooth manifold. Consider the representation 
\[\rho:\Gamma\to\GL(\Rb^{d+2}\oplus\Rb^{d})\]
defined by $\rho(\gamma) = \overline{\tau}_{d+2} \oplus \overline{\tau}_d(\gamma)$. Define boundary maps $\xi_\rho^1 : \partial_\infty\Gamma \rightarrow \Pb(\Rb^{d+2} \oplus \Rb^d)$ and $\xi_\rho^{2d-1} : \partial_\infty \Gamma \rightarrow \Gr_{2d-1}(\Rb^{d+2} \oplus \Rb^d)$ by 
$$
\xi_\rho^1(x) = \Psi_{d+2,1}(x) \oplus \{0\} \quad \text{and} \quad \xi_\rho^{2d-1}(x) = \Psi_{d+2,d-1}(x) \oplus \Rb^d.
$$
By the discussion above and \eqref{eq:eigenvalues_std_repn}, this is a pair of smooth, dynamics preserving, $\rho$-equivariant, transverse maps. Thus, one deduces from \eqref{eq:eigenvalues_std_repn} that $\rho$ is $P_1$-Anosov, but it is not $P_2$-Anosov because
\begin{align*}
\frac{\lambda_2(\rho(\gamma)) }{\lambda_3(\rho(\gamma)) }=1
\end{align*}
for any $\gamma \in \Gamma$.  However, since $\Psi_{d+2,1}$ is a smooth map, the $P_1$-limit set of $\rho$ is a $1$-dimensional, $C^{\infty}$-submanifold of $\Pb(\Rb^{2d+2})$. This shows that the conclusion of Theorem \ref{thm:nec_surface_body} does not necessarily hold if we do not assume the irreducibility hypothesis of Theorem \ref{thm:nec_surface_body}.
\end{example} 

\subsection{Proof of Theorem \ref{thm:nec_surface_body}}\label{sec: dfgklj}
Let $\rho:\Gamma\to\PGL_d(\Rb)$ be a $P_1$-Anosov representation and $M$ is a $\rho$-controlled subset which is $C^\alpha$ along $\xi_\rho^1(\partial_\infty\Gamma)$ for some $\alpha > 1$. Recall that 
\[\Phi:M\to\Pb\bigg(\bigwedge^2\Rb^d\bigg)\] 
was defined in Section \ref{thm: lkj}, see \eqref{eqn:Phi}. We will use the same strategy used to prove Theorem~\ref{thm:nec_general_body}. More precisely, by Lemma \ref{lem:nec_general}, it suffices to prove the following analog of Lemma~\ref{lem:proximal}.


\begin{lemma} \label{lem:surface1}\label{lem:surface2}
Suppose  $\rho$ is irreducible and $M$ is a topological circle that is $C^{\alpha}$ along the $P_1$-limit set of $\rho$ for some $\alpha>1$.
\begin{enumerate}
\item If $\gamma \in \Gamma$ has infinite order, then $\bigwedge^{2} \rho(\gamma)$ is proximal and $\Phi\left(\xi_\rho^1(\gamma^+)\right)\in\Pb\big(\bigwedge^2\Rb^d\big)$ is the attracting fixed point of $\bigwedge^{2} \rho(\gamma)$. 
\item If $\gamma \in \Gamma$, then 
\begin{align*}
\frac{ \lambda_{3}(\rho(\gamma))}{\lambda_2(\rho(\gamma))} \leq \left(\frac{ \lambda_{2}(\rho(\gamma))}{\lambda_1(\rho(\gamma))} \right)^{\alpha-1}. 
\end{align*}
\end{enumerate}
\end{lemma}

\begin{remark} In Lemma \ref{lem:proximal}, we assumed that $\bigwedge^2\rho$ is irreducible. However, in Lemma \ref{lem:surface1}, we assumed that $\rho$ is irreducible.
\end{remark}

\begin{proof}[Proof of Lemma \ref{lem:surface1}] 
Define the map
\[\overline{\Psi}: \xi_\rho^1(\partial_\infty\Gamma)\times \xi_\rho^1(\partial_\infty\Gamma)\rightarrow \Gr_2(\Rb^d)\] 
by letting $\overline{\Psi}(p,q)$ be the projective line containing $p,q$ when $p \neq q$ and letting $\overline{\Psi}(p,p)$ be the projective line tangent to $M$ at $p$. Notice that $\overline{\Psi}$ is continuous since $M$ is  $C^1$ along the $P_1$-limit set of $\rho$. Then define 
\[\Psi:=F_{d,2}\circ\overline{\Psi}: \xi_\rho^1(\partial_\infty\Gamma)\times \xi_\rho^1(\partial_\infty\Gamma)\to\Pb\bigg(\bigwedge^2\Rb^d\bigg),\] 
where $F_{d,2}$ is defined by \eqref{eqn:Fdm}. Observe that $\Psi$ is continuous and $\Phi(p)=\Psi(p,p)$ for all $p\in \xi_\rho^1(\partial_\infty\Gamma)$.

Fix distances $d_1$ on $\Pb(\Rb^d)$ and $d_2$ on $\Pb\big(\bigwedge^2\Rb^d\big)$ that are induced by Riemannian metrics. Since $M$ is $C^{\alpha}$ along the $P_1$-limit set of $\rho$ for some $\alpha>1$, there exists $C > 0$ such that 
\begin{align}\label{eqn:m=2}
d_2\Big( \Psi(p,p), \Psi(p, q) \Big) \leq Cd_1(p, q)^{\alpha-1} 
\end{align}
for all $p,q \in \xi_\rho^1(\partial_\infty\Gamma)$ (see Observation \ref{obs: appendix3 estimate 3}). Also, since $\rho$ is irreducible, the elements of $ \xi_\rho^1(\partial_\infty\Gamma)$ span $\Rb^d$, so 
\begin{align*}
\Psi\left( \xi_\rho^1(\partial_\infty\Gamma) \times  \xi_\rho^1(\partial_\infty\Gamma)\right)
\end{align*} 
spans $\bigwedge^2 \Rb^d$. 

The rest of the proof closely follows the proof of Lemma~\ref{lem:proximal}, except that we use $\Psi(\xi_\rho^1(\gamma^+),\xi_\rho^1(\gamma^+))$ and $\Psi(\xi_\rho^1(x),\xi_\rho^1(\gamma^+))$ in place of $\Phi(\xi_\rho^1(\gamma^+))$ and $\Phi(\xi_\rho^1(x))$ respectively.
%
\end{proof}

\begin{remark}
In the case when $M$ is a topological $(m-1)$-dimensional manifold with $m>2$, it is not true that $\xi_\rho^1(x_1)+\dots+\xi_\rho^1(x_m)$ converges to $\xi_\rho^m(x)$ as $x_i\to x$, so the direct analog of \eqref{eqn:m=2} cannot hold. As such, we need the additional assumption that $\bigwedge^m\rho$ is irreducible in Theorem \ref{thm:nec_general_body}.
\end{remark}

\section{$\PGL_d(\Rb)$-Hitchin representations}\label{sec:Hitchin}

 In this section, let $\Gamma:=\pi_1(\Sigma)$, where $\Sigma$ is a closed, orientable, connected hyperbolic surface. Let $\tau_d:\PGL_2(\Rb)\to\PGL_d(\Rb)$ be the projectivization of the representation $\overline{\tau}_d : \GL_2(\Rb) \rightarrow \GL_d(\Rb)$ defined in Section \ref{sec:rhoirred}.
 
 \begin{definition}\label{defn:hitchin_reps}
 A \emph{$\PGL_d(\Rb)$-Hitchin representation} is a representation that lies in a connected component of $\Hom(\Gamma,\PGL_d(\Rb))$ that contains $\tau_d\circ j$ for some Fuchsian representation $j:\Gamma\to\PGL_2(\Rb)$.
 \end{definition}

The goal of this section is to show that if $\rho$ is a $\PGL_d(\Rb)$-Hitchin representation, then for all $k=1,\dots,d-1$, $\bigwedge^k\rho:\Gamma\to\PGL(\bigwedge^k\Rb^d)$ satisfies the hypothesis of Theorem \ref{thm:main} (see Example \ref{cor:hitchin}). The following proposition is a consequence of Labourie's deep work on the Hitchin component~\cite{L2006} and has also been observed by Pozzetti-Sambarino-Wienhard \cite{PSW18}.

\begin{proposition}\label{prop:3_hyperconvex_exterior_prod} Let $\rho$ be a $\PGL_d(\Rb)$-Hitchin representation, let $k \in \{1,\dots, d-1\}$ and let $D := \dim \big(\bigwedge^k \Rb^d\big)={d\choose k}$. Then $\bigwedge^k \rho : \Gamma \rightarrow \PGL\big(\bigwedge^k \Rb^d\big)$ is $P_{1,2}$-Anosov, and its $P_1$-limit map $\zeta^{1}$ and $P_{D-2}$-limit map $\zeta^{D-2}$ satisfy the property that
\begin{align*}
\zeta^{1}(x)+\zeta^{1}(y) + \zeta^{D-2}(z),
\end{align*}
is a direct sum for all $x,y,z \in \partial_\infty \Gamma$ distinct. 
\end{proposition}

For the rest of the section fix some $\PGL_d(\Rb)$-Hitchin representation $\rho$ and some finite symmetric generating set $S$ of $\Gamma$.

\subsection{Preliminaries}\label{sec:prelim_Hitchin}
Before proving the proposition, we recall some results of Labourie. By Theorem 4.1 and Proposition 3.2 in~\cite{L2006},
 \begin{enumerate}
 \item\label{item:hitchin1} $\rho$ is $P_k$-Anosov for every $1 \leq k \leq d$. Denote the $P_k$-limit map of $\rho$ by $\xi_\rho^k$.
 \item\label{item:hitchin2} If $x,y,z \in \partial_\infty \Gamma$ are pairwise distinct, $k_1,k_2,k_3 \geq 0$, and $k_1+k_2+k_3 =d$, then 
 \begin{align*}
 \xi_\rho^{k_1}(x) + \xi_\rho^{k_2}(y) + \xi_\rho^{k_3}(z) = \Rb^d
 \end{align*}
 is a direct sum.
  \item\label{item:hitchin3} If $x,y,z \in \partial_\infty \Gamma$ are pairwise distinct and $0\leq k < d-2$, then
 \begin{align*}
 \xi_\rho^{k+1}(y) + \xi_\rho^{d-k-2}(z) + \Big(\xi_\rho^{k+1}(x) \cap \xi_\rho^{d-k}(z) \Big)= \Rb^d
 \end{align*}
 is a direct sum.
 \item\label{item:hitchin4} If $\gamma \in \Gamma - \{1\}$, then 
 \begin{align*}
 \lambda_1(\rho(\gamma)) > \dots > \lambda_d(\rho(\gamma)).
 \end{align*}
 \item\label{item:hitchin5} If $\gamma \in \Gamma - \{1\}$, then $ \xi_\rho^k(\gamma^+)$ is the span of the eigenspaces of $\rho(\gamma)$ corresponding to 
 \begin{align*}
  \lambda_1(\rho(\gamma)), \dots, \lambda_k(\rho(\gamma)).
  \end{align*} 
\end{enumerate}

\subsection{The proof of Proposition~\ref{prop:3_hyperconvex_exterior_prod}}

Since $\rho$ is $P_k$-Anosov, Theorem~\ref{thm:SV_char_of_Anosov} implies that there exist $C,c>0$ such that 
\begin{align}\label{eqn: Anosov singular values}
\log  \frac{\mu_k(\rho(\gamma))}{\mu_{k+1}(\rho(\gamma))} \geq C d_S(1,\gamma) -c
 \end{align}
for all $\gamma \in \Gamma$ and $1 \leq k \leq d$. 

\begin{lemma}\label{lem: Hitchin 1} $\bigwedge^k \rho$ is $P_{1,2}$-Anosov. \end{lemma}

\begin{proof}
By Theorem~\ref{thm:SV_char_of_Anosov} it is enough to prove that there exist $A,a>0$ such that 
\begin{align}
\label{eqn:12}
\log  \frac{\mu_1}{\mu_{2}}\bigg(\bigwedge^k\rho(\gamma)\bigg) \geq A d_S(1,\gamma) -a
 \end{align}
 and
\begin{align}
\label{eqn:23}
\log  \frac{\mu_2}{\mu_{3}}\bigg(\bigwedge^k \rho(\gamma)\bigg) \geq A d_S(1,\gamma) -a
 \end{align}
for all $\gamma \in \Gamma$. Here, recall that we identify $\bigwedge^k\Rb^d\cong\Rb^D$ using the basis of $\bigwedge^k\Rb^d$ induced by the standard basis of $\Rb^d$.


The inequality \eqref{eqn:12} is an immediate consequence of \eqref{eqn:0evalue1} and \eqref{eqn: Anosov singular values}. To verify the inequality \eqref{eqn:23}, fix $\gamma \in \Gamma$. For convenience, we denote
\[\chi_i:=\mu_i\bigg(\bigwedge^k\rho(\gamma)\bigg)\] 
for all $1\le i\le D$, and
\[\sigma_i:=\mu_i(\rho(\gamma))\]
for all $ 1\le i\le d$. We previously observed, see Equation~\eqref{eqn:wedge eignevalues and singular values}, that there exist $1 \leq i_1 < \dots <i_k \leq d$ such that
\begin{align*}
\chi_3 = \sigma_{i_1}\cdots \sigma_{i_k}.
\end{align*}
Note that $i_j\geq j$ for all $j$. Further, 
$$
\chi_1 = \sigma_1 \cdots \sigma_k \quad \text{and} \quad \chi_2 = \sigma_1 \cdots \sigma_{k-1} \sigma_{k+1}. 
$$
So $(i_1,\dots i_k) \notin \{ (1,\dots,k), (1,\dots, k-1, k+1)\}$.

We consider two cases based on the value of $i_{k-1}$. 

\medskip

\noindent \textbf{Case 1:} Suppose $i_{k-1} = k-1$. Then $i_j = j$ for $j \leq k-1$ and $i_k \geq k$. Since 
\begin{align*}
(i_1,\dots, i_k) \notin \{ (1,\dots, k), (1,\dots,k-1,k+1)\}
\end{align*}
we must have $i_k \geq k+2$. So
\begin{align*}
\log \frac{\chi_2}{\chi_3} = \log  \left(\frac{\sigma_1} {\sigma_{i_1}}\cdots  \frac{\sigma_{k-1}}{\sigma_{i_{k-1}}} \frac{\sigma_{k+1}}{\sigma_{i_{k}}}\right) = \log \frac{\sigma_{k+1}}{\sigma_{i_k}} \geq \log  \frac{\sigma_{k+1}}{\sigma_{k+2}} \geq C \ell_S(\gamma)-c.
\end{align*}

\noindent \textbf{Case 2:} Suppose $i_{k-1} \geq k$. Then $i_k \geq k+1$. Since $i_j \geq j$ for all $j$, we have
 \begin{align*}
\log \frac{\chi_2}{\chi_3} = \log  \left(\frac{\sigma_1}{\sigma_{i_1}} \cdots \frac{\sigma_{k-2}}{\sigma_{i_{k-2}}}  \frac{\sigma_{k-1}}{\sigma_{i_{k-1}}} \frac{\sigma_{k+1}}{\sigma_{i_{k}}}\right) \geq \log \frac{\sigma_{k-1}}{\sigma_{i_{k-1}}} \geq \log \frac{\sigma_{k-1}}{\sigma_{k}} \geq C \ell_S(\gamma)-c.
\end{align*}
In either case 
\[
\log  \frac{\mu_2}{\mu_{3}}\bigg(\bigwedge^k \rho(\gamma)\bigg)=\log \frac{\chi_2}{\chi_3} \geq C \ell_S(\gamma) -c.\qedhere
\]
\end{proof}

Given subspaces $V_1,\dots,V_k\subset\Rb^d$, we will let $V_1\wedge\dots\wedge V_k$ denote the subspace of $\bigwedge^k\Rb^d$ that is spanned by $\{X_1\wedge\dots\wedge X_k:X_i\in V_i\}$. For $j\in\{1,2,D-2, D-1\}$ define maps 
\begin{align*}
\zeta^{j} : \partial_\infty \Gamma \rightarrow \Gr_j\bigg(\bigwedge^k \Rb^d\bigg)
\end{align*}  
by
\begin{align*}
\zeta^{1}(x) = \bigwedge^k \xi_\rho^k(x),
\end{align*}
\begin{align*}
\zeta^{2}(x) =\bigg( \bigwedge^{k-1} \xi_\rho^{k-1}(x) \bigg) \wedge \xi_\rho^{k+1}(x), 
\end{align*}
\begin{align*}
\zeta^{D-2}(x) =\xi_\rho^{d-k-1}(x) \wedge \bigg( \bigwedge^{k-1} \Rb^d \bigg)  + \xi_\rho^{d-k}(x) \wedge \xi_\rho^{d-k+1}(x) \wedge \bigg( \bigwedge^{k-2} \Rb^d \bigg), 
\end{align*}
\begin{align*}
\zeta^{D-1}(x) = \xi_\rho^{d-k}(x) \wedge  \bigg(\bigwedge^{k-1} \Rb^d \bigg).
\end{align*}
These maps are clearly continuous and $\bigwedge^k \rho$-equivariant.

\begin{lemma}\label{lem: Hitchin 2}
$\zeta^{1}, \zeta^{2}, \zeta^{D-2}, \zeta^{D-1}$ are the limit maps of $\bigwedge^k \rho$. 
\end{lemma}

\begin{proof} By the density of attracting fixed points in $\partial_\infty \Gamma$ and the continuity of the maps, it is enough to fix an infinite order element $\gamma \in \Gamma$ and then verify that if $j\in\{1,2,D-2, D-1\}$ and $\gamma^+ \in \partial_\infty \Gamma$ is the attracting fixed point of $\gamma$, then $\zeta^{j}(\gamma^+)$ is the attracting fixed point of $\bigwedge^k \rho(\gamma)$ in $\Gr_j(\bigwedge^k \Rb^d)$. 

By Property~\eqref{item:hitchin5} in Section~\ref{sec:prelim_Hitchin}, there exists a basis $v_1, \dots, v_d$ of $\Rb^d$ of eigenvectors of $\rho(\gamma)$ such that 
\begin{align*}
\xi_\rho^{j}(\gamma^+) = \Span\{ v_1,\dots, v_j\} \quad \text{for all} \quad j=1,\dots, d.
\end{align*}
Let $I_1 = \{ d-k+1, d-k+2, \dots, d\}$ and $I_2 = \{ d-k, d-k+2, d-k+3,\dots, d\}$. Then a calculation shows that 
\begin{align*}
\zeta^{1}(\gamma^+) = \left[ v_1 \wedge \dots \wedge v_k\right],
\end{align*}
\begin{align*}
\zeta^{2}(\gamma^+) = \left\{ v_1 \wedge \dots \wedge v_{k-1} \wedge (av_k+bv_{k+1}) : a,b \in \Rb\right\},
\end{align*}
\begin{align*}
\zeta^{D-2}(\gamma^+) = \Span\left\{ v_{i_1} \wedge \dots \wedge v_{i_k} : \{ i_1, \dots, i_k\} \notin \{ I_1, I_2\} \right\},
\end{align*}
\begin{align*}
\zeta^{D-1}(\gamma^+) = \Span\left\{ v_{i_1} \wedge \dots \wedge v_{i_k} : \{ i_1, \dots, i_k\} \neq I_1 \right\}.
\end{align*}
Since each $v_{i_1} \wedge \dots \wedge v_{i_k}$ is an eigenvector of $\bigwedge^k\rho(\gamma)$ with eigenvalue having modulus $\lambda_{i_1}(\rho(\gamma)) \cdots \lambda_{i_k}(\rho(\gamma))$, by Property~\eqref{item:hitchin4} in Section~\ref{sec:prelim_Hitchin}, we see that for each $j\in\{1,2,D-2, D-1\}$, the point $\zeta^{j}(\gamma^+)$ is the attracting fixed point of $\bigwedge^k \rho(\gamma)$ in $\Gr_j(\bigwedge^k \Rb^d)$. 
\end{proof}

\begin{lemma} \label{lem: Hitchin 3}
$\zeta^{1}(x) + \zeta^{1}(y) + \zeta^{D-2}(z)$ is a direct sum for all $x,y,z \in \partial_\infty \Gamma$ pairwise distinct.  
\end{lemma}

\begin{proof} 
Fix $x,y,z \in \partial_\infty \Gamma$ pairwise distinct, and choose a basis $v_1,\dots, v_d \in \Rb^d$ such that 
\begin{align*}
[v_j] = \xi_\rho^j(x) \cap \xi_\rho^{d-j+1}(z)
\end{align*}
for $1 \leq j \leq d$. Pick $u_1, \dots, u_k \in \Rb^d$ such that 
\begin{align*}
\xi_\rho^k(z) = \Span\{ u_1,\dots, u_k\}.
\end{align*}
Then $\zeta^{1}(y) = [ u_1 \wedge \dots \wedge u_k]$.

If $I = \{ 1,\dots, k-1, k+1\}$, then a computation shows that
\begin{align*}
\zeta^{1}(x) + \zeta^{D-2}(z) = \Span\left\{ v_{i_1} \wedge \dots \wedge v_{i_k} : \{ i_1, \dots, i_k\} \neq I \right\}.
\end{align*}
It is perhaps worth noting that we can also describe this sum as a kernel of a certain linear map: 
\begin{align*}
\zeta^{1}(x) + \zeta^{D-2}(z) = \ker \bigg( w \in \bigwedge^{k} \Rb^d \mapsto w \wedge v_k \wedge  v_{k+2} \wedge \dots \wedge v_{d} \in \bigwedge^{d} \Rb^d \bigg). 
\end{align*}

By Property~\eqref{item:hitchin3} in Section~\ref{sec:prelim_Hitchin},
\begin{align*}
\xi_\rho^k(y) + \left( \xi_\rho^k(x) \cap \xi_\rho^{d-k+1}(z) \right) + \xi_\rho^{d-k-1}(z) = \Rb^d
\end{align*}
is a direct sum. Since 
\begin{align*}
\left( \xi_\rho^k(x) \cap \xi_\rho^{d-k+1}(z) \right) + \xi_\rho^{d-k-1}(z) = \Span \{ v_k, v_{k+2}, \dots, v_d\}
\end{align*}
we see that 
\begin{align*}
(u_1 \wedge \dots \wedge u_k) \wedge (v_k \wedge v_{k+2} \wedge \dots \wedge v_d) \neq 0.
\end{align*}
This implies that  
\begin{equation*}
\zeta^{1}(x) +\zeta^{1}(y) + \zeta^{D-2}(z) = \bigwedge^k \Rb^d.\qedhere
\end{equation*}
\end{proof}
Together, Lemmas \ref{lem: Hitchin 1}, \ref{lem: Hitchin 2}, and \ref{lem: Hitchin 3} immediately imply Proposition~\ref{prop:3_hyperconvex_exterior_prod}.
 
\section{Real hyperbolic lattices}\label{sec:real_hyp_lattices}

The goal of this section is to justify Example \ref{cor:hyperbolic_lattices}. More precisely, we will prove the following proposition.

\begin{proposition}\label{thm:hyperbolic} Suppose $\tau: {\rm Isom}(\Hb^m_{\Rb}) \rightarrow \PGL_d(\Rb)$ is a representation, $\Gamma \leq {\rm Isom}(\Hb^m_{\Rb})$ is a convex co-compact subgroup, and $\rho = \tau|_{\Gamma} : \Gamma \rightarrow \PGL_d(\Rb)$ is the representation obtained by restriction. If $\tau$ is irreducible and proximal, then $\rho$ is $P_{1,m}$-Anosov and 
\begin{align*}
\xi_\rho^1(x) + \xi_\rho^1(y) + \xi_\rho^{d-m}(z)
\end{align*}
is a direct sum for all $x,y,z \in \partial_\infty \Gamma$ distinct. Thus, the same is true for sufficiently small deformations of $\rho$.
 \end{proposition}
 
 \begin{remark} Recall that a representation is proximal if its image contains a proximal element. \end{remark}

 \subsection{Preliminaries}

Consider the unit ball $\Bb_m \subset \Rb^m$. By realizing $\Rb^m$  as the affine chart 
\[\left\{\begin{bmatrix}
X\\
1
\end{bmatrix}\in\Pb(\Rb^{m+1}):X\in\Rb^n\right\}\]
of $\Pb(\Rb^{m+1})$, we may view $\Bb_m$ as a properly convex domain in $\Pb(\Rb^{m+1})$, and so we may endow $\Bb_m$ with its Hilbert metric $d = H_{\Bb_m}$. Then $(\Bb_m, d)$ is the \emph{Klein-Beltrami model} of the real hyperbolic $m$-space.

As usual, let $\PO(m,1)\leq\PGL_{m+1}(\Rb)$ denote the subgroup that leaves invariant the bilinear pairing  
 \[
\ip{X,Y} = X_1 Y_1 + \dots + X_m Y_m - X_{m+1} Y_{m+1} 
\]
on $\mathbb{R}^{m+1}$. Then we can identify $\PO(m,1)$ with ${\rm Isom}(\Hb^m_{\Rb})$ via fractional linear transformations, i.e. $\PO(m,1)$ acts by isometries on $\Bb_m\subset\Rb^m$ via the action
\begin{align*}
\begin{bmatrix} A & u \\ {^tv} & a \end{bmatrix} \cdot x = \frac{ Ax + u}{{^tv}x + a}.
\end{align*}

If $e_1,\dots, e_{m}$ is the standard basis on $\Rb^{m}\subset\Rb^{m+1}$ and $H$ is the $(m+1)$-by-$(m+1)$ matrix 
\[
H= \begin{bmatrix} 0 & e_1 \\ {^t e_1} & 0 \end{bmatrix},
\] 
then $e^{sH} \in \PO(m,1)$ for all $s \in \Rb$. Using the formula for the Hilbert metric, one can compute that 
$$
d(e^{sH} \cdot 0, 0) = \abs{s}.
$$
In fact, one can verify that the map $\gamma_0:\Rb\to\Bb_m$ given by 
\[
\gamma_0(s)=\tanh(s)e_1=e^{sH}\cdot 0
\] 
is a unit-speed geodesic in $\Bb_m$ with $-e_1$ and $e_1$ as its backward and forward endpoints respectively. 

A computation also verifies that $K:=\{g\in\PO(m,1):g\cdot 0=0\}$ is given by
\begin{align}\label{eqn:K}
K =  \left\{ \begin{bmatrix} A & 0 \\ 0 & \sigma \end{bmatrix}\in\PO(m,1): \sigma \in \{-1,1\}, \ A \in {\rm O}(m)\right\}.
 \end{align}
In particular, $K$ acts transitively on the set of unit vectors in $T_0\Bb_m$. Since $\PO(m,1)$ acts transitively on $\Bb_m$, this implies that $\PO(m,1)$ acts transitively on the unit tangent bundle of $\Bb_m$. Also, if $p\in\Bb_m$, then $d(0,p)=d(0,k\cdot p)$ for all $k\in K$. This, together with the $KAK$-decomposition theorem \cite[Theorem 7.39]{knapp}, implies the following observation.

\begin{observation}\label{obs:KAK} If $g \in \PO(m,1)$, then there exist $k_1, k_2 \in K$ such that 
\begin{align*}
g = k_1 e^{d(g \cdot 0, 0)H} k_2.
\end{align*}
\end{observation}

\subsection{Proof of Proposition~\ref{thm:hyperbolic}}
Let $\tau$, $\rho$, and $\Gamma$ satisfy the hypothesis of Proposition~\ref{thm:hyperbolic}. As described above, we identify $\Hb_{\Rb}^m = \Bb_m$ and ${\rm Isom}(\Hb_{\Rb}^m) = \PO(m,1)$. 

Let $\tau_0 := \tau|_{e^{\Rb \cdot H}}$ and let $\overline{\tau}_0 : e^{\Rb \cdot H} \rightarrow \SL_d(\Rb)$ be the lift of $\tau_0$ (since $\Rb$ is simply connected, such a lift exists). 

\begin{lemma}\label{lem:prox+2epsace}
After conjugating $\tau$, we may assume that $\overline{\tau}_0(e^{\Rb \cdot H})$ is a subgroup of the diagonal matrices. Moreover there exists $\lambda \in \frac{1}{2} \Nb$  such that the set of eigenvalues of $\overline{\tau}_0\left(e^{sH}\right)$ is 
\[
\left\{e^{s\lambda}, e^{s(\lambda-1)}, \dots, e^{-s\lambda} \right\}, 
\] 
the multiplicity of $e^{s\lambda}$ is 1, and the multiplicity of $e^{s(\lambda-1)}$ is $m-1$. 
\end{lemma} 

Since $\tau$ is irreducible, the proof of Lemma \ref{lem:prox+2epsace} is a standard argument from the theory of weight spaces. We give this argument in Appendix \ref{app:lem}. Assuming this lemma, we can prove Proposition \ref{thm:hyperbolic}.


Since $\Gamma$ is convex co-compact, we can identify $\partial_\infty \Gamma$ with its limit set in $\partial \Bb_m \simeq \partial_\infty \Hb^m_{\Rb}$. Lemma \ref{lem:prox+2epsace} and~\cite[Propositions 4.4 and 4.7]{GW2012} imply that:
\begin{enumerate}
\item $\rho = \tau|_{\Gamma}$ is $P_{1,m}$-Anosov, 
\item for $i=1,d-1,m,d-m$ there exists a $\tau$-equivariant map 
\[
\xi^{i}_\tau:\partial \Bb_m \rightarrow \Gr_i(\Rb^d)
\] 
such that $\xi_\rho^i = \xi^i_\tau|_{\partial_\infty \Gamma}$.
\end{enumerate} 

Further,  $\PO(m,1)$ acts transitively on triples of pairwise distinct points $x,y,z \in \partial \Bb_m$. Thus, to show that 
\begin{align*}
\xi_\rho^1(x) + \xi_\rho^1(y) + \xi_\rho^{d-m}(z)
\end{align*}
is a direct sum for all pairwise distinct $x,y,z \in \partial_\infty \Gamma$,  it is enough to show that 
\begin{align*}
\xi_\tau^1(x) + \xi_\tau^1(y) + \xi_\tau^{d-m}(z)
\end{align*}
is direct for some pairwise distinct $x,y,z \in \partial \Bb_m$. Fix $y,z \in \partial \Bb_m$ distinct. Then since $\tau$ is irreducible we must have
\begin{align*}
\Rb^d = \Span \{ \xi_\tau^1(x) : x \in \partial \Bb_m\}
\end{align*}
and so there exists some $x \in \partial \Bb_m$ such that 
\begin{align*}
\xi_\tau^1(x) + \xi_\tau^1(y) + \xi_\tau^{d-m}(z)
\end{align*}
is direct.

\appendix

\section{Theorem \ref{thm:cones}}\label{sec:Benoists appendix}

\begin{proof} First notice that $\Cc_\lambda(\Gamma)$ is invariant under conjugating $\Gamma$ in $\SL_d(\Rb)$, i.e. $\Cc_\lambda(\Gamma)=\Cc_\lambda(g\Gamma g^{-1})$ for all $g\in \SL_d(\Rb)$. Further, if $h \in \SL_d(\Rb)$, then from the geometric description of the Cartan projection given in Section \ref{sec:properties0}, there exists some $C > 0$ such that 
\begin{align*}
\norm{\mu(g) - \mu(hgh^{-1}) }_2 \leq C
\end{align*}
for all $g \in \SL_d(\Rb)$. Hence $\Cc_\mu(\Gamma)$ is also invariant under conjugating $\Gamma$ in $\SL_d(\Rb)$. 

Let $\sL_d(\Rb) = \kL + \pL$ denote the standard Cartan decomposition of $\sL_d(\Rb)$, that is 
\begin{align*}
\kL = \{ X \in \sL_d(\Rb) : {^tX} = -X\}\quad\text{ and }\quad \pL = \{ X \in \sL_d(\Rb) : {^tX} = X\}.
\end{align*}
Let $\gL$ denote the Lie algebra of $G$. Using Theorem 7 in~\cite{M1955} and conjugating $G$ we may assume that 
\begin{align*} 
\gL = \kL \cap \gL + \pL \cap \gL
\end{align*}
is a Cartan decomposition of $\gL$. Fix a maximal abelian subspace $\aL \subset \pL \cap \gL$.  By~\cite[Chapter V, Lemma 6.3(ii)]{H2001}, we may further conjugate $G$ by an element in $\SO(d)$ to ensure that $\aL$ is itself a subspace of the diagonal matrices. 
Finally fix a Weyl chamber $\aL^+$ of $\aL$. 

Next let $K \subset G$ denote the subgroup corresponding to $\kL \cap \gL$, let $A = \exp(\aL)$, and let $A^+ = \exp(\aL^+)$. By~\cite[Chapter IX, Theorem 1.1]{H2001}, each $g \in G$ can be written as 
\begin{align*}
g = k_1 \exp( \mu_G(g) ) k_2
\end{align*}
where $k_1, k_2 \in K$ and $\mu_G(g) \in \overline{\aL^+}$ is unique. The map $\mu_G : G \rightarrow \overline{\aL^+}$ is called the \emph{Cartan projection of $G$ relative to the decomposition $G = K \overline{A}^+ K$.} Since $K \subset \SO(d)$ and $\aL$ is a subspace of the diagonal matrices,  the diagonal entries of $\mu_G(g)$ coincide with the entries of $\mu(g)$ up to permuting indices. 

More precisely, let $\mathfrak d \subset \mathfrak{sl}_d(\Rb)$ denote the subspace of diagonal matrices and let $f: \mathfrak d \rightarrow \mathfrak d$ be the map 
$$
f({\rm diag}(x_1,\dots, x_d)) = {\rm diag}(x_{i_1}, \dots, x_{i_d})
$$
where $\{ i_1,\dots, i_d\} = \{1,\dots, d\}$ and $x_{i_1} \geq x_{i_2} \geq \cdots \geq x_{i_d}$. Then $f$ is continuous and 
$$
\mu(g) = f(\mu_G(g))
$$
for all $g \in G$. 



Every $g \in G$ can be written as a product $g=g_e g_h g_u$ of commuting elements, where $g_e$ is elliptic, $g_h$ is hyperbolic, and $g_u$ is unipotent. This is called the \emph{Jordan decomposition of $g$ in $G$}. The element $g_h$ is conjugate to a unique element $\exp(\lambda_G(g)) \in \overline{A^+}$ and the map $\lambda_G : G \rightarrow \overline{\aL^+}$ is called the \emph{Jordan projection}. Since $G$ is a semisimple real algebraic subgroup of $\SL_d(\Rb)$, the Jordan decomposition in $G$ coincides with the standard Jordan decomposition in $\SL_d(\Rb)$ \cite[Theorem~9.20]{FH}. Then, since $\aL$ is a subspace of the diagonal matrices, we must have 
$$
\lambda(g) = f(\lambda_G(g))
$$
for all $g \in G$.

Next define cones $\Cc_1, \Cc_2 \subset \aL^+$ as follows: 
\begin{align*}
\Cc_1 := \overline{\bigcup_{\gamma \in \Gamma} \Rb_{>0} \cdot \lambda_G(\gamma)}
\end{align*}
and 
\begin{align*}
\Cc_2 := \{ x \in \Rb^d : \exists \gamma_n \in \Gamma, \exists t_n \searrow 0, \text{ with } \lim_{n \rightarrow +\infty} t_n \mu_G(\gamma_n) =x\}.
\end{align*}
Then the main result in~\cite{B1997} says that $\Cc_1 = \Cc_2$. Then 
\begin{equation*}
\Cc_\mu(\Gamma) = f(\Cc_2) = f(\Cc_1) = \Cc_\lambda(\Gamma). \qedhere
\end{equation*}
\end{proof}

\section{Proof of Lemma \ref{lem:prox+2epsace}\label{app:lem} }

As in Section~\ref{sec:real_hyp_lattices}, suppose $\tau:\PO(m,1)\to\PGL_d(\Rb)$ is a proximal irreducible representation, let $H$ denote the $(m+1)$-by-$(m+1)$ matrix
\[
H= \begin{bmatrix} 0 & e_1 \\ {^t e_1} & 0 \end{bmatrix},
\]
and let $\overline{\tau}_0 : e^{\Rb \cdot H} \rightarrow \SL_d(\Rb)$ denote the lift of $\tau_0:=\tau|_{e^{\Rb \cdot H}}$. Let  $\mathfrak{sl}_d(\Rb)$ denote the Lie algebra of $\PGL_d(\Rb)$ and let $d\tau:\mathfrak{so}(m,1)\to\mathfrak{sl}_d(\Rb)$ be the derivative at the identity of the homomorphism $\tau:\PO(m,1)\to\PGL_d(\Rb)$.

 Let $\mathfrak{so}(m,1)$ denote the Lie algebra of $\PO(m,1)$. Explicitly,
 \begin{align*}
 \mathfrak{so}(m,1) = \left\{ \begin{bmatrix} A & u \\ {^tu} & 0 \end{bmatrix} : {^tA}=-A \right\}.
 \end{align*}
Further, if
 $$
 \pL^\prime : = \left\{ X \in \mathfrak{so}(m,1) : {^tX} = X \right\} \quad \text{and} \quad \kL^\prime := \left\{ X \in  \mathfrak{so}(m,1) : {^tX}=-X \right\},
 $$
 then 
 $$
 \mathfrak{so}(m,1) = \kL^\prime + \pL^\prime
 $$
 is a Cartan decomposition of $\mathfrak{so}(m,1)$.

\begin{lemma}\label{lem:the image of the Cartan is diagonal}
After conjugating $\tau$, we may assume that $\overline{\tau}_0(e^{\Rb \cdot H})$ is a subgroup of the diagonal matrices. 
\end{lemma}

\begin{proof} Let $\sL_d(\Rb) = \kL + \pL$ denote the standard Cartan decomposition of $\sL_d(\Rb)$, that is 
\begin{align*}
\kL = \{ X \in \sL_d(\Rb) : {^tX} = -X\} \quad \text{and} \quad \pL = \{ X \in \sL_d(\Rb) : {^tX} = X\}.
\end{align*}

Also, let $\gL := d\tau( \mathfrak{so}(m,1))$. Then $\gL$ is a simple Lie algebra isomorphic to $\mathfrak{so}(m,1)$ and $\gL = d\tau(\kL^\prime) + d\tau(\pL^\prime)$ is a Cartan decomposition of $\gL$. 

Then using Theorem 6 in~\cite{M1955} and Theorem 7.2 in~\cite[Chapter III]{H2001}, and conjugating $\tau$ we may assume that 
\begin{align*} 
d\tau(\kL^\prime) \subset \kL \quad \text{and} \quad d\tau(\pL^\prime) \subset \pL.
\end{align*}
Then $d\tau(\Rb\cdot H) \subset \pL$ and so by~\cite[Chapter V, Lemma 6.3(ii)]{H2001}, there exists some $k \in \SO(d)$ such that $\Ad(k)d\tau(\Rb\cdot H)$ is a subspace of the diagonal matrices in $\sL_d(\Rb)$. Thus $k\overline{\tau}_0(e^{\Rb \cdot H})k^{-1}$ is a subgroup of the diagonal matrices. 
\end{proof} 

\begin{lemma}\label{lem:the image of the Cartan is proximal} If $s \neq 0$, then $\overline{\tau}_0(e^{sH})$ is proximal and the eigenvalue with maximal modulus is a positive real number. \end{lemma} 

\begin{proof}Since $\tau$ is proximal, there exists $g \in \PO(m,1)$ such that $\tau(g)$ is proximal. Using Observation~\ref{obs:KAK}, for each $n \geq 0$ we can write $g^n = k_{1,n} e^{t_n H} k_{2,n}$ where $k_{1,n}, k_{2,n} \in K$ and $t_n \geq 0$. Since $K$ is compact, 
\begin{align*}
0 < \log \frac{\lambda_1}{\lambda_2}(\tau(g)) = \lim_{n \rightarrow +\infty} \frac{1}{n}\log \frac{\mu_1}{\mu_2}(\tau(g)^n) = \lim_{n \rightarrow +\infty} \frac{1}{n}\log \frac{\lambda_1}{\lambda_2}\left(\tau\left( e^{t_n H} \right)\right). 
\end{align*}
So for $n$ large, $\lambda_1(\tau(e^{t_nH})) > \lambda_2(\tau(e^{t_nH}))$. Since $\tau(e^{\Rb \cdot H})$ is a subgroup of the diagonal matrices, this implies that $\lambda_1(\tau(e^{sH})) > \lambda_2(\tau(e^{sH}))$ for all $s > 0$.  Since $e^{sH}$ is conjugate to $e^{-sH}$ in $\PO(m,1)$, $\tau(e^{sH})$, and hence $\overline{\tau}_0(e^{sH})$, is also proximal for every $s \neq 0$. 
 
Finally, since $\overline{\tau}_0(\id) = \id$ has all positive eigenvalues, we see that  the eigenvalue with maximal modulus of $\overline{\tau}_0(e^{sH})$ is positive for all $s \neq 0$.
\end{proof}

 Next let $e_1,\dots, e_{m+1}$ be the standard basis of  $\Rb^{m+1}$. Then define the vector  subspaces of $ \mathfrak{so}(m,1)$:
 \begin{align*}
 \aL & :=  \Rb \cdot H = \left\{ \begin{bmatrix} 0 & \lambda e_1 \\ \lambda {^te_1} & 0 \end{bmatrix} : \lambda \in \Rb \right\}, \\
 \gL_0 & :=   \left\{ \begin{bmatrix} A & \lambda e_1 \\ \lambda {^te_1} & 0 \end{bmatrix} : {^tA}=-A, \quad Ae_1 =0, \text{ and } \lambda \in \Rb \right\}, \\
\gL_{-1} & :=  \left\{ \begin{bmatrix} -u  {^te_1} +e_1{^tu} & u \\ {^tu} & 0 \end{bmatrix} :  \ip{u,e_1} =0\right\}, \text{ and}\\
\gL_{1} & :=  \left\{ \begin{bmatrix} u  {^te_1} - e_1  {^tu} & u \\ {^tu} & 0 \end{bmatrix} : \ip{u,e_1} =0\right\}.
 \end{align*}
Then $\mathfrak a\subset\mathfrak g_0$ is a maximal abelian subalgebra, and the decomposition 
\[
\mathfrak{so}(1,m) = \gL_0 + \gL_{-1} + \gL_{1}
\] 
is the associated (restricted) \emph{root space decomposition} of $\mathfrak{so}(1,m)$.

The following lemma states some basic properties of the root space decomposition, see \cite[Chapter II.1]{knapp}, and can be verified explicitly in this special case.

\begin{lemma}\label{obs:rootspace}\ \begin{enumerate}
\item Let  $\sigma\in\{0,1,-1\}$, and $Y \in \gL_\sigma$. Then 
\[[H,Y]=\sigma Y\quad\text{and}\quad \Ad \left( e^{s H} \right) Y= e^{\sigma s} Y.\]
\item Let $\alpha, \beta \in \{0,-1,1\}$. Then $[\gL_\alpha, \gL_\beta] \subset \gL_{\alpha+\beta}$, where $\gL_{-2}:=\{0\}=:\gL_2$. 
\end{enumerate}
\end{lemma}

The next lemma gives a description of the eigenvalues and eigenspaces of $\overline{\tau}_0(e^{sH})$.

\begin{lemma}\label{lem:weights} Let $e^\lambda$ denote the largest eigenvalue of $\overline{\tau}_0(e^{H})$ and let $V_\lambda \subset \Rb^d$ denote the eigenspace of $\overline{\tau}_0(e^{H})$ corresponding to $e^\lambda$. For $n \in \Zb_{\geq 0}$, define
\begin{align*}
V_{\lambda-(n+1)} := \Span\left( d\tau(\gL_{-1}) V_{\lambda-n}\right).
\end{align*}
Then:
\begin{enumerate}
\item If $v \in V_{\lambda-n}$, then $\overline{\tau}_0\left(e^{sH}\right)v = e^{s(\lambda -n)}v$.
\item If $Z \in \gL_0$, then $d\tau(Z)V_{\lambda-n} \subset V_{\lambda-n}$. 
\item If $Z \in \gL_{1}$, then $d\tau(Z)V_\lambda = \{0\}$ and $d\tau(Z) V_{\lambda-n} \subset V_{\lambda-(n-1)}$ when $n>0$.
\item $\lambda \in \frac{1}{2}\Nb$ and $\oplus_{n=0}^{2\lambda} V_{\lambda-n} = \Rb^d$.
\end{enumerate}
\end{lemma}

\begin{proof} (1): This is true by definition when $n=0$. So suppose that $n > 0$. It suffices to consider the case when $v = d\tau(Y)w$ for some $Y \in \gL_{-1}$ and $w \in V_{\lambda-(n-1)}$. Then by induction
\begin{align*}
\overline{\tau}_0\left(e^{sH}\right)d\tau(Y)w 
&= d\tau( \Ad(e^{sH})Y )\overline{\tau}_0\left(e^{sH}\right)w \\
&= d\tau( e^{-s}Y )\left(e^{s(\lambda-n+1)}w\right) \\
& = e^{s(\lambda-n)} d\tau( Y )w,
\end{align*}
where the second equality is a consequence of (1) of Lemma \ref{obs:rootspace}.

 (2): If $v_0 \in V_\lambda$, then 
 $$
\overline{\tau}_0\left(e^{H}\right) d\tau(Z)v_0 = d\tau( \Ad(e^H)Z )\overline{\tau}_0\left(e^H\right) v_0=e^\lambda d\tau(Z) v_0.
 $$
 So $d\tau(Z)V_\lambda \subset V_\lambda$. Thus we can suppose that $n > 0$. Fix $v \in d\tau(\gL_{-1}) V_{\lambda-(n-1)}$. Then $v = d\tau(Y)w$ for some $Y \in \gL_{-1}$ and $w \in V_{\lambda-(n-1)}$. Lemma \ref{obs:rootspace} part (2) implies that $[Z,Y] \in \gL_{-1}$ and so 
\begin{align*}
d\tau(Z)v = d\tau(Z)d\tau(Y)w = d\tau([Z,Y])w - d\tau(Y)d\tau(Z)w 
\end{align*}
is contained in $V_{\lambda-n}$ by induction. Hence $d\tau(Z)V_{\lambda-n} \subset V_{\lambda-n}$.

(3): If $v_0 \in V_\lambda$, then 
\begin{align*}
\overline{\tau}_0\left(e^H\right)d\tau(Z)v_0 = d\tau( \Ad(e^H)Z )\overline{\tau}_0\left(e^H\right)v_0  = e^{\lambda+1} d\tau( Z )v_0.
\end{align*}
Since $e^\lambda$ is the largest eigenvalue of $\overline{\tau}_0(e^H)$ we must have $d\tau(Z)v_0=0$. Since $v_0 \in V_\lambda$ was arbitrary, we then have $d\tau(Z)V_\lambda = \{0\}$. 

Next fix $n > 1$ and $v \in d\tau(\gL_{-1}) V_{\lambda-(n-1)}$. Then $v = d\tau(Y)w$ for some $Y \in \gL_{-1}$ and $w \in V_{\lambda-(n-1)}$. Lemma \ref{obs:rootspace} part (2) implies that $[Z,Y] \in \gL_{0}$, so
\begin{align*}
d\tau(Z)d\tau(Y)w = d\tau([Z,Y])w - d\tau(Y)d\tau(Z)w 
\end{align*}
is contained in $V_{\lambda-(n-1)}$ by (2) and induction. Hence $d\tau(Z)V_{\lambda-n} \subset V_{\lambda-(n-1)}$.
 
(4): By part (1), $\sum_{n \geq 0} V_{\lambda-n} = \oplus_{n \geq 0} V_{\lambda-n}$. By definition and parts (2) and (3), $\bigoplus_{n \geq 0} V_{\lambda-n}$ is a $d\tau(\gL)$-invariant and hence $\tau(G)$-invariant subspace. Since $\tau$ is irreducible, we then have $\oplus_{n \geq 0} V_{\lambda-n}=\Rb^d$. Since $e^H$ is conjugate to $e^{-H}$ in $\PO(m,1)$, we see that $e^{-\lambda}$ is the smallest eigenvalue of $\overline{\tau}_0\left(e^H\right)$. Hence $\lambda \in \frac{1}{2} \Nb$ and 
\begin{equation*}
\Rb^d = \oplus_{n \geq 0} V_{\lambda-n}=\oplus_{n=0}^{2\lambda} V_{\lambda-n}. \qedhere
\end{equation*}
\end{proof}

\begin{proof}[Proof of Lemma \ref{lem:prox+2epsace}] By Lemmas~\ref{lem:the image of the Cartan is diagonal}, ~\ref{lem:the image of the Cartan is proximal}, and~\ref{lem:weights},  it suffices to show that 
$$
\dim V_{\lambda -1} = m-1.
$$

Fix some non-zero $v_0 \in V_\lambda$, and consider the linear map $T: \gL_{-1} \rightarrow V_{\lambda-1}$ given by 
\begin{align*}
T(Y) = d\tau(Y)v_0.
\end{align*}

Since $\dim V_\lambda=1$, 
$$
V_{\lambda-1} = \Span\left( d\tau(\gL_{-1}) V_{\lambda}\right) = d\tau(\gL_{-1}) V_{\lambda}.
$$
So  $T$ is onto. Since $\dim \gL_{-1} = m-1$, it suffices to show that $\ker T = \{0\}$. To see this, let
\begin{align*}
 M := \left\{ k \in K: k \cdot e_1 = \pm e_1\right\}=\left\{ \begin{bmatrix} \sigma_1 & & \\ & A &  \\  & & \sigma_2 \end{bmatrix}: \sigma_1,\sigma_2 \in \{-1,1\}, \ A \in {\rm O}(m-1)\right\}.
 \end{align*}
 Then $M$ commutes with $e^{H}$. So $\Ad(M)$ preserves $\gL_{-1}$ and one can check that $\Ad(M)$ acts irreducibly on $\gL_{-1}$. 
 
 We claim that $\ker T$ is an $\Ad(M)$-invariant subspace. Since $M$ commutes with $e^H$, we see that $\tau(M)$ preserves $V_\lambda$. Then, since $M$ is compact and $\dim V_\lambda = 1$, 
 $$
 \tau(M)|_{V_\lambda} \leq \GL(V_\lambda) \cong \Rb^{\times}
 $$
 is compact and so $\tau(M)|_{V_\lambda} \leq \{ \pm \id_{V_\lambda}\}$. Define $\sigma : M \rightarrow \{1,-1\}$ by $\tau(k)|_{V_\lambda} = \sigma(k)\id_{V_\lambda}$. 
 
 Finally, if $Y \in \ker T$ and $k \in M$, then
\begin{align*}
T(\Ad(k)Y) = d\tau\left(\Ad(k)Y\right)v_0 = \tau(k)d\tau(Y) \tau(k^{-1}) v_0 = \sigma(k^{-1}) \tau(k)T(Y)v_0=0.
\end{align*}
So $\ker T$ is an $\Ad(M)$-invariant subspace. 

Then either $\ker T = \{0\}$ or $\ker T = \gL_{-1}$. If $\ker T = \gL_{-1}$, then $V_{\lambda - n} = 0$ for all $n > 0$. Then 
$$
1 = \dim V_\lambda=d
$$
and since $d > 1$ this is impossible. So $\ker T =\{0\}$ and hence $\dim V_{\lambda -1} = \dim \gL_{-1} = m-1$.
\end{proof}

\section{Geometric consequences of smoothness along a subset}\label{sec:regularity appendix} 

As in Section~\ref{sec:optimal1}, suppose that $M \subset \Pb(\Rb^d)$ is a topological $(m-1)$-dimensional manifold  which is $C^{\alpha}$ along a compact subset $N\subset M$ for some $\alpha >1$.  

In this Appendix we will record some geometric consequences of this definition. To that end, fix a distance $d_{\Pb}$ on $\Pb(\Rb^d)$ induced by a Riemannian metric on $\Pb(\Rb^d)$. Directly from the definition we obtain the following. 

\begin{observation}\label{obs:appendix3 estimate 1} There exists a constant $C_1 > 0$ such that: for every $a \in N$ there is a unique $m$-plane $T_a M \in \Gr_m(\Rb^d)$ that satisfies 
$$
d_{\Pb}(b, T_a M) \leq C_1 d_{\Pb}(a,b)^{\alpha}
$$
for all $b \in M$ (here, as usual, we view $T_a M$ as both a $m$-dimensional linear subspace of $\Rb^d$ and a $(m-1)$-dimensional projective subspace of $\Pb(\Rb^d)$). 
\end{observation} 

Next fix a distance $d_{m}$ on $\Gr_m(\Rb^d)$ induced by a Riemannian metric on $\Gr_m(\Rb^d)$. 

\begin{observation}\label{obs: appendix3 estimate 2}  There exists a constant $C_2 > 0$ such that: if $a,b \in N$, then 
$$
d_{m}(T_a M, T_b M) \leq C_2 d_{\Pb}(a,b)^{\alpha-1}.
$$
\end{observation} 

\begin{proof} By the compactness of $N$, it suffices to fix $a \in N$ and obtain the desired estimate for $b \in N$ sufficiently close to $a$. 

Fix local smooth coordinates around $a$ where there exists a continuous map $f:U\to\Rb^{d-m}$ defined on an open set $U\subset \Rb^{m-1}$, such that
\begin{enumerate}
\item $M$ coincides with the graph of $f$ near $a$.
\item For every $(u,f(u)) \in N$, $f$ is differentiable at $u$.
\item There are constants $C,\delta>0$ such that 
\[
\norm{f(u+h)-f(u)-df_u(h)}_2 \leq C\norm{h}_2^\alpha
\] 
for all $u\in U$ with $(u,f(u)) \in N$ and for all $h\in\Rb^{m-1}$ with $\norm{h}_2<\delta$.
\end{enumerate}

We can assume that $a=0$ in these local coordinates. Then it suffices to find $C_1, \delta_1 > 0$ such that: if $b=(u,f(u)) \in N$ satisfies $\norm{u}_2 < \delta_1$, then 
\begin{equation}
\label{eqn:desired estimate in the last appendix}
\norm{df_0(h) - df_u(h)}_2 \leq C_1 \norm{u}_2^{\alpha-1}\norm{h}_2
\end{equation}
for all $h \in \Rb^{m-1}$. Indeed, this implies that
\[\norm{df_0 - df_u} \leq C_1 \norm{u}_2^{\alpha-1},\]
($\norm{\cdot}$ is the operator norm induced by $\norm{\cdot}_2$), so the observation follows because $d_m(T_a M ,T_b M)$ is locally bi-Lipschitz to $\norm{df_0 - df_u}$ and $d_{\Pb}(a,b)$ is locally bi-Lipschitz to $\norm{u}_2$.

Let $\delta_1 : = \frac{\delta}{2}$ and $C_1:=C(2+2^\alpha)$. Let $(u,f(u)) \in N$ with $\norm{u}_2 < \delta_1$ and let $h \in \Rb^{m-1}$. In the special case where $\norm{h}_2=\norm{u}_2$,
\begin{align*}
 \norm{df_0(h) - df_u(h)}_2 & = \big\Vert -\big( f(u-u) - f(u) + df_u(u) \big) - \big( f(0+h) - f(0) - df_0(h) \big) \\
& \quad \quad \quad + \big( f(u+(h-u)) - f(u) - df_u(h-u) \big) \big\Vert_2 \\
& \leq C\norm{u}_2^\alpha + C\norm{h}_2^{\alpha} + C\norm{h-u}_2^{\alpha} \leq C(2+2^\alpha) \norm{u}_2^{\alpha-1}\norm{h}_2.
\end{align*}
For general $h$, first observe that if $h=0$, then Equation~\eqref{eqn:desired estimate in the last appendix} clearly holds. On the other hand, if $h\neq 0$, set $h':=\frac{\norm{u}_2}{\norm{h}_2}h$, and we have
\begin{align*}
 \norm{df_0(h) - df_u(h)}_2&= \frac{\norm{h}_2}{\norm{u}_2}\norm{df_0(h') - df_u(h')}_2\leq \frac{\norm{h}_2}{\norm{u}_2}C(2+2^\alpha) \norm{u}_2^{\alpha-1}\norm{h'}_2\\
 &=C(2+2^\alpha) \norm{u}_2^{\alpha-1}\norm{h}_2,\end{align*}
so Equation~\eqref{eqn:desired estimate in the last appendix} also holds. 
\end{proof} 

Next we specialize to the case when $m=2$ and let $\overline{\Phi} : N \times N \rightarrow \Gr_2(\Rb^d)$ denote the map 
$$
\overline{\Phi}(a,b) = \begin{cases} \Span\{ a,b\} & \text{ if } a \neq b \\ T_a M & \text{ if } a = b\end{cases}.
$$

\begin{observation}\label{obs: appendix3 estimate 3}   There exists a constant $C_3 > 0$ such that: if $a,b \in N$, then 
$$
d_{2}\left(T_a M, \overline{\Phi}(a,b)\right) \leq C_3 d_{\Pb}(a,b)^{\alpha-1}.
$$
\end{observation} 

\begin{proof} This follows from Observation~\ref{obs:appendix3 estimate 1}. 
\end{proof}

\bibliographystyle{alpha}
\bibliography{hilbert}

\end{document}